\documentclass[11pt,letterpaper]{amsart}
\usepackage{amsmath, amsthm, amssymb, url, color, hyperref, graphicx, esint}
\usepackage[final]{showkeys}
\usepackage{tabularx}
\usepackage{verbatim}
\usepackage[normalem]{ulem}
\usepackage{xcolor}
\hypersetup{
    colorlinks,
    citecolor=blue,
    filecolor=black,
    linkcolor=black,
    urlcolor=black
}

\usepackage{amsmath}
\usepackage[foot]{amsaddr}

\newcommand{\band}{\mathfrak{B}_q}
\newcommand{\hpg}{\mathcal{H}}

\newcommand{\ii}{\mathbf i}

\newcommand{\RR}{\mathbb{R}}
\newcommand{\CC}{\mathbb{C}}

\newcommand{\ZZ}{\mathbb{Z}}
\newcommand{\NN}{\mathbb{N}}
\newcommand{\Tr}{\operatorname{Tr}}

\newcommand{\E}{\mathbb{E}}
\newcommand{\D}{\mathbb{D}}
\newcommand{\EE}{\mathbb{E}}
\newcommand{\eps}{\varepsilon}

\newcommand{\cR}{\mathcal{R}}
\newcommand{\cQ}{\mathcal{Q}}
\newcommand{\cC}{\mathcal{C}}

\newcommand{\wtilde}{\widetilde}

\newcommand{\cX}{\mathcal{X}}

\newcommand{\cG}{\mathcal{G}}
\newcommand{\cH}{\mathcal{H}}

\newcommand{\wg}{\wtilde{g}}

\newcommand{\cF}{\mathcal{F}}

\newcommand{\cV}{\mathcal{V}}

\DeclareMathOperator*{\Res}{Res}

\newcommand{\TT}{\mathbb{T}}
\newcommand{\HH}{\mathbb{H}}

\renewcommand{\Im}{\operatorname{Im}}
\renewcommand{\Re}{\operatorname{Re}}

\newcommand{\wG}{\wtilde{G}}

\newcommand{\wDelta}{\wtilde{\Delta}}

\newcommand{\cZ}{\mathcal{Z}}
\newcommand{\cA}{\mathcal{A}}

\newcommand{\wcA}{\wtilde{\cA}}
\newcommand{\lgg}{l_0}

\newcommand{\QQ}{\mathbb{Q}}

\newcommand{\cW}{\mathcal{W}}
\newcommand{\cT}{\mathcal{T}}

\newcommand{\cY}{\mathcal{Y}}
\newcommand{\tpsi}{\tilde{\psi}}

\newcommand{\yi}[1]{{#1}}
\newcommand{\guillaume}[1]{{#1}}
\newcommand{\xin}[1]{{#1}}

\newcommand{\revise}[1]{{#1}}

\def\defeq{:=}
\def\bbH{\mathbb{H}}

\newtheorem{theorem}{Theorem}[section]
\newtheorem{corr}[theorem]{Corollary}
\newtheorem{lemma}[theorem]{Lemma}
\newtheorem{prop}[theorem]{Proposition}
\newtheorem{assump}[theorem]{Assumption}

\numberwithin{equation}{section}

\theoremstyle{definition}
\newtheorem{define}[theorem]{Definition}
\newtheorem{remark}[theorem]{Remark}

\def\beq{\begin{equation}}
\def\eeq{\end{equation}}
\def\ba{\begin{enumerate}[(a)]}
\def\bei{\begin{enumerate}[(i)]}
\def\be{\begin{enumerate}[(1)]}
\def\ee{\end{enumerate}}
\def\bi{\begin{itemize}}
\def\ei{\end{itemize}}
\def\beg{\begin{eg}}
\def\eeg{\end{eg}}
\def\bd{\begin{defn}}
\def\ed{\end{defn}}
\def\bt{\begin{theorem}}
\def\et{\end{theorem}}
\def\bl{\begin{lemma}}
\def\el{\end{lemma}}
\def\bfac{\begin{fact}}
\def\efac{\end{fact}}

\def\bc{\begin{corr}}
\def\ec{\end{corr}}
\def\bp{\begin{prop}}
\def\ep{\end{prop}}
\def\bo{\begin{observe}}
\def\eo{\end{observe}}
\def\bas{\begin{assump}}
\def\eas{\end{assump}}
\def\RR{\mathbb{R}}
\def\CC{\mathbb{C}}
\def\EE{\mathbb{E}}
\def\ZZ{\mathbb{Z}}
\def\NN{\mathbb{N}}
\def\QQ{\mathbb{Q}}

\def\beg{\begin{eg}}
\def\eeg{\end{eg}}

\newcommand{\wwp}{\widetilde{\wp}}

\newcommand{\wV}{\widetilde{V}}
\newcommand{\wphi}{\wtilde{\phi}}
\newcommand{\notion}[1]{{\textit{#1}}}
\newcommand{\T}{\mathsf{T}}

\begin{document}
\title{Probabilistic conformal blocks for Liouville CFT on the torus}
\date{\today}

\author{Promit Ghosal}\email{promit@mit.edu}\address{Department of Mathematics, Massachusetts Institute of Technology}
\author{Guillaume Remy}\email{remy@math.columbia.edu}\address{Institute for Advanced Study}
\author{Xin Sun} \email{xinsun@sas.upenn.edu}\address{Department of Mathematics, University of Pennsylvania}
\author{Yi Sun}\email{yisun@statistics.uchicago.edu}\address{Department of Statistics, University of Chicago}

\begin{abstract}
Virasoro conformal blocks are a family of important functions defined as power series via
the Virasoro algebra. They are a fundamental input to the conformal bootstrap program for 2D
conformal field theory (CFT) and are closely related to four dimensional supersymmetric gauge
theory through the Alday-Gaiotto-Tachikawa correspondence. The present work provides a
probabilistic construction of the $1$-point toric Virasoro conformal block for central change
greater than $25$. More precisely, we construct an analytic function using a probabilistic
tool called Gaussian multiplicative chaos (GMC) and prove that its power series expansion
coincides with the $1$-point toric Virasoro conformal block. The range $(25,\infty)$ of central charges
corresponds to Liouville CFT, an important CFT originating from 2D quantum gravity and bosonic
string theory. Our work reveals a new integrable structure underlying GMC and opens the door to
the study of non-perturbative properties of Virasoro conformal blocks such as their analytic
continuation and modular symmetry. Our proof combines an analysis of GMC with tools from CFT
such as Belavin-Polyakov-Zamolodchikov differential equations, operator product expansions, and
Dotsenko-Fateev type integrals.
\end{abstract}

\maketitle


\section{Introduction}
A conformal field theory (CFT) is a way to construct random functions on Riemannian manifolds
that transform covariantly under conformal (i.e. angle preserving) mappings.
Since the seminal work of Belavin-Polyakov-Zamolodchikov \cite{BPZ84}, two dimensional (2D)
CFT has grown into one of the most prominent branches of theoretical physics,  with applications to 2D statistical
physics and string theory, as well as far reaching consequences in mathematics; see e.g.\ \cite{DMS97}.
The paper~\cite{BPZ84} introduced a schematic program called the \emph{conformal bootstrap} to exactly solve 
correlation functions of a given 2D CFT in terms of its $3$-point sphere correlation functions and certain
power series called {\bf conformal blocks}. These conformal blocks are completely specified by the Virasoro algebra
that encodes the infinitesimal local conformal symmetries, and they only depend on the specific CFT
through a single parameter called the \emph{central charge}. 
Outside of CFT, conformal blocks are related to Nekrasov partition functions in gauge theory via the
Alday-Gaiotto-Tachikawa correspondence \cite{AGT10}, solutions to Painlev\'e-type equations \cite{GIL12},
and quantum Teichm\"uller theory and representation of quantum groups \cite{PT99, PT01, TV13},  
among other things.   

In this paper, we initiate a probabilistic approach to study the conformal blocks appearing in the conformal
bootstrap for an important 2D CFT called {\bf Liouville conformal field theory} (LCFT).  LCFT arose from
Polyakov's work on 2D quantum gravity and bosonic string theory  \cite{Pol81}; it was rigorously constructed
from the path integral formalism of quantum field theory on the sphere in \cite{DKRV16} and on other surfaces in
\cite{DRV16,HRV18,GRV19}. The construction is via {\bf Gaussian multiplicative chaos} (GMC), a random
measure defined by exponentiating the Gaussian free field (see e.g.\ \cite{RV-GMC,NB-GMC}).
LCFT depends on a coupling constant $\gamma \in (0,2)$ which is in bijection with the central charge $c$ via
\begin{equation}\label{eq:cQ}
c = 1 + 6Q^2 \in (25, \infty), \qquad  \qquad \textrm{where }Q = \frac{\gamma}{2} + \frac{2}{\gamma}.
\end{equation}
The present work gives a GMC representation of the conformal blocks with central charge $c \in (25, \infty)$
for a torus with one marked point. Given $\tau$ in the upper half plane, let $\mathbb{T}_\tau$ be the flat torus
with modular parameter $\tau$.  The $1$-point toric correlation function of LCFT, rigorously constructed in \cite{DRV16},
has the form $\langle e^{\alpha \phi(0)} \rangle_{\tau}$, where $\langle\cdots \rangle_{\tau}$ is the average over the
random field $\phi$ for LCFT on $\TT_\tau$ and $\alpha$ is called the vertex insertion weight.

The conformal bootstrap gives rise to a conjectural \emph{modular bootstrap equation} expressing $\langle e^{\alpha \phi(0)} \rangle_{\tau}$ via the \emph{$1$-point toric conformal block} $\cF^\alpha_{\gamma, P}(q)$:
\begin{equation} \label{eq:intro-bootstrap}
\langle e^{\alpha \phi(0)}\rangle_{\tau}
= \frac{1}{|\eta(q)|^2} \int_{-\infty}^\infty C_\gamma(\alpha, Q - \ii P, Q + \ii P)
|q|^{P^2} \cF^\alpha_{\gamma, P}(q) \cF^{\alpha}_{\gamma, P}(\bar{q}) dP.
\end{equation}
Here, $q = e^{\ii \pi \tau} $, $\eta(q)$ is the Dedekind eta function, and $C_\gamma(\alpha_1, \alpha_2, \alpha_3)$ is the Liouville $3$-point sphere correlation function. It has an exact expression called the DOZZ formula which was  first proposed in
\cite{DO94, ZZ96} and proved in~\cite{KRV19b}. The conformal block $\cF^\alpha_{\gamma, P}(q)$ is a $q$-power series
defined via the Virasoro algebra in~\cite{BPZ84}. Recently, the conformal bootstrap for LCFT was rigorously carried out for the sphere in~\cite{GKRV20} and \xin{for general closed surfaces (including the torus) in~\cite{GKRV21}.}

In this paper, we
use GMC to construct a function of $q$ analytic around $0$ whose series expansion is given by $\cF^\alpha_{\gamma, P}(q)$.

\subsection{Summary of results}
To state our results, we first review two ways to characterize the 1-point toric conformal block
$\cF^\alpha_{\gamma, P}(q)$ as a formal $q$-series with parameters $\gamma, P,\alpha$: Zamolodchikov's recursion
and the Alday-Gaiotto-Tachikawa (AGT) correspondence. 
From its original definition via the Virasoro algebra~\cite{BPZ84},
it was shown in \cite{Zam84,Zam87,HJS09} that $\cF^\alpha_{\gamma, P}(q)$, as a formal $q$-series,
is the unique solution to Zamolodchikov's recursion
\begin{equation} \label{eq:intro-recursion}
\cF^\alpha_{\gamma, P}(q) = \sum_{n, m = 1}^\infty
q^{2mn} \frac{R_{\gamma, m, n}(\alpha)}{P^2 - P_{m, n}^2} \cF^\alpha_{\gamma, P_{-m, n}}(q)
+ q^{\frac{1}{12}} \eta(q)^{-1},
\end{equation}
where $R_{\gamma, m, n}(\alpha)$ and $P_{m, n}$ are explicit constants defined in (\ref{eq:rmn-def}) and
(\ref{eq:pmn-def}). 
See \cite{Pog09,FL10, HJS09,CCY19} for more background on this recursion.

\revise{The AGT correspondence  \cite{AGT10}  specialized to $\cF^\alpha_{\gamma, P}(q)$, which is proven in
	\cite{FL10}, 
	asserts that as two formal power series we have}
\begin{equation} \label{eq:intro-block-nek}
\cF^\alpha_{\gamma, P}(q) = 
\left(q^{-\frac{1}{12}}\eta(q)\right)^{1 - \alpha(Q - \frac{\alpha}{2})} \cZ^\alpha_{\gamma, P}(q).
\end{equation}
Here, \revise{$\cZ^\alpha_{\gamma, P}(q)$ is  given by}
\begin{multline} \label{eq:intro-nek-def}
\cZ^\alpha_{\gamma, P}(q) := 1\\ + \sum_{k = 1}^\infty q^{2k} \!\!\!\!\!\! \sum_{\substack{Y_1, Y_2 \text{ Young diagrams} \\ |Y_1| + |Y_2| = k}}
\prod_{i, j = 1}^2 \prod_{s \in Y_i} \frac{(E_{ij}(s, P) - \alpha)(Q - E_{ij}(s, P) - \alpha)}{E_{ij}(s, P) (Q - E_{ij}(s, P))},
\end{multline}
where $E_{ij}(s, P)$ is an explicit product given by (\ref{eq:nek-eqn}). We also note that 
$q^{-\frac{1}{12}}\eta(q)=\prod_{n=1}^{\infty}(1-q^{2n})$ has an explicit $q$-series expansion. \revise{The function $\cZ^\alpha_{\gamma, P}(q)$  is the instanton part of the Nekrasov partition function of the four-dimensional $SU(2)$ supersymmetric gauge theory.}  
See e.g.\ \cite{Nek03,NekOk06,N16,CO12} for more background on the Nekrasov partition function and the AGT correspondence \revise{in general}.

We may now state our main result.
Here we only give a concise foretaste of the relevant constructions; a fully rigorous version of all definitions
will be given in Section~\ref{sec:block-def}.
 For $\gamma \in (0, 2)$ and for purely imaginary $\tau \in \ii \RR_{>0}$, Let $Y_\tau(x)$ be the Gaussian field on $[0, 1]$ with covariance
\[
\EE[Y_\tau(x)Y_\tau(y)] = - 2 \log |\Theta_\tau(x - y)| + 2 \log|q^{\frac{1}{6}} \eta(q)|,
\]
where $\Theta_\tau(x)$ is the Jacobi theta function (Appendix~\ref{sec:theta}).
Consider the GMC measure $e^{\frac{\gamma}{2} Y_\tau(x) } dx$, which is a random measure on $[0, 1]$  defined through regularization (Definition~\ref{def:gmc}). 
For $\alpha \in (-\frac{4}{\gamma}, Q)$, $q =e^{i\tau}\in (0, 1)$, and $P \in \mathbb{R}$, we define the
\emph{probabilistic $1$-point toric conformal block} by
\begin{equation} \label{eq:def-block}
\cG_{\gamma, P}^{\alpha}(q) \defeq \frac{1} {Z_{\gamma, P}^{\alpha}(q)}
\EE\left[\left(\int_0^1  |\Theta_\tau(x)|^{-\frac{\alpha\gamma}{2}} e^{\pi \gamma P x} {e^{\frac{\gamma}{2} Y_\tau(x)}dx} \right)^{-\frac{\alpha}{\gamma}} \right],
\end{equation}
where $Z_{\gamma, P}^{\alpha}(q)$ is an explicit function of $\alpha,\gamma, P, q$  from Definition \ref{def:conf-block} and Remark~\ref{rmk:Z}. 
Our main result Theorem~\ref{thm:main} below
shows that (\ref{eq:def-block}) gives a probabilistic construction of $\cF^\alpha_{\gamma, P}(q)$ which is non-perturbative, in contrast to Zamolodchikov's recursion, the AGT correspondence, and its original definition via the Virasoro algebra.

\newcommand{\radius}{\frac{1}{2}}

\begin{theorem}\label{thm:main}
	For $\gamma \in (0, 2)$, $\alpha \in (-\frac{4}{\gamma}, Q)$, 
	and $P \in \RR$, the probabilistic conformal block  $\cG^\alpha_{\gamma, P}(q)$ admits an analytic extension on  a complex neighborhood of $q=0$, whose $q$-series expansion around $q=0$  agrees with  $\cF^\alpha_{\gamma, P}(q)$ defined in~\eqref{eq:intro-block-nek}. In particular, the conformal block $\cF^\alpha_{\gamma, P}(q)$  has a  positive  radius of convergence.
	Moreover,  when $\alpha \in [0,Q)$, the analytic extension of $\cG^\alpha_{\gamma, P}(q)$ exists on an open set containing $\{ z \in \mathbb{C} : |z| < \frac{1}{2} \} \cup [0,1)$, which implies the radius of convergence of
	$\cF^\alpha_{\gamma, P}(q)$ is at least $\radius$.
\end{theorem}
The range $(-\frac{4}{\gamma}, Q)$ for $\alpha$  is the range in which the  1-point correlation
function $\langle e^{\alpha \phi(0)}\rangle_{\tau}$ in~\eqref{eq:intro-bootstrap} has a GMC expression
from the path integral formalism of LCFT~\cite{DRV16}.  Although our probabilistic construction of conformal blocks
also relies on GMC, we are not aware of a natural path integral interpretation.  From this perspective,
our Theorem~\ref{thm:main} reveals a new integrable structure underlying GMC.    

In light of the AGT correspondence, Theorem~\ref{thm:main} proves that the
Nekrasov partition function \eqref{eq:intro-nek-def} is analytic in $q$, resolving
a \revise{special case of a} conjecture in\cite{FM18}\footnote{More
	precisely, their conjecture was stated for  the  $4$-point spherical conformal block.
	In light of \cite{FLNO09, Pog09, HJS09}, the $1$-point toric conformal block is a special case of the $4$-point
	spherical conformal block under a parameter change.}. \xin{As a corollary of the conformal bootstrap, it was shown  in~\cite{GKRV20,GKRV21}  that $\cF^\alpha_{\gamma, P}(q)$ has convergence radius 1 for almost every real $P$.}  
Our work gives a complementary understanding
of conformal blocks by making tractable important analytic properties such
as the modular transformations and complex analyticity in the parameter $P \in \mathbb{C}$,
which seem out of reach from the methods of \cite{GKRV20,GKRV21}. See
Section \ref{sec:intro-outlook} for more discussion. 

\subsection{Summary of method} \label{sec:method-sum}
The key steps for proving Theorem~\ref{thm:main} are:
\begin{enumerate}
	\item[1.] 
	We  show by Cameron-Martin's theorem that  the GMC expression for $\cG^\alpha_{\gamma, P}(q)$ has the desired analytic properties in $q$ prescribed by Theorem~\ref{thm:main}. 
	
	\item[2.] We then characterize the series coefficients of $\cG^\alpha_{\gamma, P}(q)$ in terms of shift equations,
	which are difference equations in $\alpha$ with step size $2\chi$
	for $\chi\in \{\frac{\gamma}{2}, \frac{2}{\gamma}\}$ inspired by shift equations in \cite{Tes95} proposed
	for the DOZZ formula.  To prove series coefficients of $\cG^\alpha_{\gamma, P}(q)$ satisfy
	shift equations, we generalize the strategy of~\cite{KRV19b} based on the \emph{BPZ equation}
	and operator product expansion (OPE).
	Our setting differs from~\cite{KRV19b} in two important
	ways. First, instead of reducing to a hypergeometric equation, our BPZ equation is the PDE
	\begin{equation} \label{eq:intro-bpz}
	\Big(\partial_{uu} - l_\chi (l_\chi + 1) \wp(u) + 2 \ii \pi \chi^2 \partial_\tau\Big) \psi^\alpha_{\chi}(u, q) = 0
	\end{equation}
	where \revise{$\psi^\alpha_{\chi}(u, q)$ is a deformation of $\cG^\alpha_{\gamma, P}(q)$}, $l_\chi = \frac{\chi^2}{2} - \frac{\alpha \chi}{2}$ and $\wp$ is the Weierstrass's elliptic function.
	As a result, we must apply separation of variables to~\eqref{eq:intro-bpz} to obtain a  
	system of \emph{inhomogeneous} hypergeometric equations from which we show the $q$-series coefficients of
	$\cG^\alpha_{\gamma, P}(q)$ solve the shift equations~\eqref{eq:shift-n}. These equations are analogous
        to shift equations in \cite{Tes95,KRV19b} that uniquely specify the DOZZ formula, but they are coupled and inhomogeneous.
	\revise{Second, the OPE for $\chi = \frac{2}{\gamma}$ requires the
		reflection argument first performed in \cite{KRV19b} but  adapted to the intricate boundary setup in \cite{RZ20}.}

	\item[3.]To complete the proof we need to verify  that the $q$-series coefficients of $\cF^\alpha_{\gamma, P}(q)$ satisfy
	the same shift equations~\eqref{eq:shift-n}.  This challenge
	is new to our setting, as in~\cite{KRV19b} it follows easily from the explicit form of the DOZZ formula. \revise{We first prove that $\cF^\alpha_{\gamma, P}(q)=\cG^\alpha_{\gamma, P}(q)$  as formal $q$-series when $N:=-\frac{\alpha}{\gamma}\in \NN$. In this case} the GMC expression
	$\cG^\alpha_{\gamma, P}(q)$ equals the integral
	\begin{equation}\label{eq:intro-DF}
	\left(\int_0^1\right)^N \prod_{1 \leq i < j \leq N} |\Theta_\tau(x_i - x_j)|^{- \frac{\gamma^2}{4}}
	\prod_{i = 1}^N \Theta_\tau(x_i)^{-\frac{\alpha \gamma}{2}} e^{\pi \gamma P x_i} \prod_{i = 1}^N dx_i.
	\end{equation}
	We prove that this integral as a formal $q$-series satisfies Zamolodchikov's recursion (\ref{eq:intro-recursion})  thus  agrees with 
	the conformal block $\cF^\alpha_{\gamma, P}(q)$.  
	The integral~\eqref{eq:intro-DF} is a Dotsenko-Fateev integral \cite{DF84, DF85} and was proposed as a representation of the 1-point toric conformal block in \cite{FLNO09}. To our best knowledge, prior to our work this representation was not established mathematically.
	
	\item[4. ] \revise{We finally remove the constraint $N:=-\frac{\alpha}{\gamma}\in \NN$ 
		using the analyticity in $\gamma$.} More precisely, fixing
	$\alpha<0$ and viewing \eqref{eq:shift-n} as equations in $\gamma$, for $N:=-\frac{\alpha}{\gamma}\in \NN$
	we have $\cF^\alpha_{\gamma, P}(q)=\cG^\alpha_{\gamma, P}(q)$. Thus \eqref{eq:shift-n} for $\chi = \frac{\gamma}{2}$
	holds for series coefficients of $\cF^\alpha_{\gamma, P}(q)$. The explicit form for series coefficients of $\cF^\alpha_{\gamma, P}(q)$
	from the AGT correspondence shows they are rational in $\gamma$ and invariant under
	the exchange $\frac{\gamma}{2} \leftrightarrow \frac{2}{\gamma}$. This allows us to analytically 
	extend in $\gamma$ and prove that the series coefficients of $\cF^\alpha_{\gamma, P}(q)$ satisfy~\eqref{eq:shift-n}	for all $\gamma$ and both values of $\chi$. Now the uniqueness of the solution to~\eqref{eq:shift-n} gives $\cF^\alpha_{\gamma, P}(q)=\cG^\alpha_{\gamma, P}(q)$.
\end{enumerate}

We now situate these steps in the paper.  In Section \ref{sec:block-def}, we prove the analytic continuation property
of the probabilistic conformal block $\cG^\alpha_{\gamma, P}(q)$ prescribed by Theorem~\ref{thm:main}.  In Section \ref{sec:bpz},
we define deformed versions of $\cG^\alpha_{\gamma, P}(q)$, characterize their analytic properties,
and prove the BPZ equations~\eqref{eq:intro-bpz}. In Section \ref{sec:ode}, we perform separation of variables
for deformed probabilistic conformal blocks and derive from the BPZ equations a system of coupled inhomogenous hypergeometric
equations. In Section \ref{sec:opes}, we state operator product expansions for deformed conformal blocks
and perform analytic continuation in $\alpha$.  In Section \ref{sec:block-proof}, we prove  two shift equations on series coefficients of $\cG^\alpha_{\gamma, P}(q)$
and prove Theorem~\ref{thm:main} as outlined above.

\subsection{Outlook}\label{sec:intro-outlook} We now outline  two directions of  future work.

\noindent \textbf{Modular transformations for conformal blocks.}  
Set $q=e^{\ii \pi \tau}$ and $\tilde q = e^{-\ii \pi \tau^{-1}}$.  Verlinde~\cite{verlinde-fusion} and Moore-Seiberg~\cite{ms-groupoid}
conjectured that
\begin{equation}\label{eq:mod-inv}
{\tilde q}^{\frac{P^2}{2}} \cF^{\alpha}_{\gamma, P}(\tilde q) = \int_{\mathbb{R}} \mathcal{M}_{\gamma,\alpha}(P, P')q^{\frac{P'^2}{2}} \cF^\alpha_{\gamma, P'}(q) dP'
\end{equation}
for a certain explicit \emph{modular kernel} $\mathcal{M}_{\gamma,\alpha}(P, P')$.  
Namely, the modular group $\mathrm{PSL}_2(\ZZ)$ acts linearly on the span of $\{q^{P^2/2} \cF^{\alpha}_{\gamma, P}(q): P \in \RR\}$.
The explicit formula for $\mathcal{M}_{\gamma,\alpha}(P, P')$ was derived by Ponsot  and Teschner~\cite{PT99} under the assumption that 
there exists a kernel $\mathcal{M}_{\gamma,\alpha}(P, P')$ satisfying~\eqref{eq:mod-inv}.
However, the equation~\eqref{eq:mod-inv} itself is still open as a mathematical question. 

In work in progress, we plan to prove \eqref{eq:mod-inv}  for $\alpha\in [0,Q)$ based on our probabilistic construction of  $\cF^\alpha_{\gamma, P}(q)$.
More precisely, we will use the BPZ equation, GMC techniques, and the explicit form of the modular kernel to show   
${\tilde q}^{\frac{P^2}{2}} \cG^{\alpha}_{\gamma, P}(\tilde{q}) = \int_{\mathbb{R}} \mathcal{M}_{\gamma,\alpha}(P, P')q^{\frac{P'^2}{2}} \cG^\alpha_{\gamma, P'}(q) dP'$
for $q\in (0,1)$, where  $\cG^{\alpha}_{\gamma, P}(q)$ is the GMC in Theorem~\ref{thm:main}.
We can then use the $\mathrm{PSL}_2(\ZZ)$ action  to analytically continue  $\cG^{\alpha}_{\gamma, P}(q)$ to the unit disk.
This means $\cF^{\alpha}_{\gamma, P}(q)$ has convergence radius $1$ in this range of $\alpha$
and~\eqref{eq:mod-inv} holds. Once proved, \eqref{eq:mod-inv} also allows us to
view the conformal block as a meromorphic function in $P \in \mathbb{C}$, as the right hand side
of \eqref{eq:mod-inv} is meromorphic in $P$ with explicit poles provided by the meromorphic
function $\mathcal{M}_{\gamma,\alpha}(P, P')$.

\smallskip

\noindent \textbf{4-point spherical conformal blocks.} 
As discussed in Section~\ref{sec:method-sum},  the GMC expression of $\cF^{\alpha}_{\gamma, P}(q)$ 
is a Dotsenko-Fateev type integral  when $-\frac{\alpha}{\gamma}\in \NN$. Such an  integral representation is available under certain specializations of parameters for more general conformal blocks,  including the 4-point spherical case; see \cite{mms-block,dv-block}.
This allows us to propose a GMC expression for 4-point spherical conformal blocks and hence an analog of Theorem~\ref{thm:main},
which we hope to prove in future work. Moreover, similar to~\eqref{eq:mod-inv}, there is a linear transformation on the span of 4-point spherical conformal blocks called the \emph{fusion transformation} which is responsible for the \emph{crossing symmetry} of the conformal bootstrap for the 4-point sphere; see~\cite[Eq (1.16)]{GKRV20}. 
We also hope to establish the fusion transformation and use it to study the analytic continuation of conformal blocks. 
As a long term goal, we hope to extend our GMC framework to conformal blocks on a genus-$g$ surface with $n$ points and to explore their symmetries predicted in~\cite{verlinde-fusion, ms-groupoid, PT99}.

\subsection*{Acknowledgments}
We thank K. Aleshkin, G. Baverez, J. Dub\'edat, A. Litvinov, R. Rhodes, and V. Vargas
for helpful discussions. We also thank C. Garban, R. Rhodes, and V. Vargas for organizing a conference on
probability and QFT on the beautiful island of Porquerolles where much of this work was discussed. GR
was supported by an NSF postdoctoral fellowship DMS-1902804. XS was
supported by a Junior Fellow award from the Simons Foundation and by NSF DMS-1811092, DMS-2027986, and Career Award DMS-2046514. YS was supported by a Junior Fellow award from the Simons Foundation and NSF
DMS-1701654, DMS-2039183, and DMS-2054838.

\section{Probabilistic construction of the conformal block} \label{sec:block-def}

The main purpose of this section is to give the precise definition of the probabilistic conformal block $\cG^\alpha_{\gamma, P}(q)$ based on GMC, and prove its analytic continuation properties  prescribed by Theorem~\ref{thm:main}. We also reduce Theorem~\ref{thm:main} to a variant Theorem~\ref{thm:nek-block} whose proof occupies the rest of the paper.  
We will use the following notations. 
Let $\CC$ be the complex plane. If $K\subset U\subset \CC$ and $U$ is open, we say that  $U$ is a complex neighborhood of $K$.
Let $\NN$ be the set of positive integers and $\NN_0=\NN\cup \{0\}$. 
Let $\bbH$ be the upper half plane and $\D$ be the unit disk. For $\tau\in \bbH$, let $q=q(\tau)=e^{\ii \pi \tau}\in \D$ so that $\tau\in \ii\RR_{>0}$ implies $q\in(0,1)$. 
We recall the Jacobi theta function $\Theta_\tau$ and the Dedekind  eta function $\eta$ from Appendix~\ref{sec:theta}.
Throughout Sections~\ref{sec:block-def}---\ref{sec:shift}, we view $\gamma\in (0,2)$ as a fixed parameter and set  $Q = \frac{\gamma}{2} + \frac{2}{\gamma}$ as in~\eqref{eq:cQ}.

\subsection{Definition of Gaussian multiplicative chaos} \label{sec:gmc-def}
We begin by introducing Gaussian multiplicative chaos (GMC), the probabilistic object which enables our construction. Let  $\{a_n\}_{n \geq 1}$, $\{b_n\}_{n \geq 1}$, $\{a_{n, m}\}_{n, m \geq 1}$, $\{b_{n, m}\}_{n, m \geq 1}$  be sequences of i.i.d. standard real Gaussians. For  $\tau \in \bbH$,
we define the Gaussian fields $Y_\tau$, $Y_\infty$ and $F_{\tau}$ on $[0, 1]$ as follows: 
\begin{align} 
Y_{\tau}(x) &:= Y_{\infty}(x) + F_{\tau}(x); \label{eq:y-tau-def}  \\ \label{eq:y-inf-def}
Y_{\infty}(x) &:= \sum_{n \geq 1} \sqrt{\frac{2}{n}} \Big( a_n \cos(2 \pi n x) + b_n \sin(2 \pi n x) \Big); \\ \label{eq:f-tau-def}
F_{\tau}(x) &:=  2 \sum_{n,m \geq 1} \frac{q^{nm}}{\sqrt{n}} \Big( a_{n,m} \cos(2 \pi n x) + b_{n,m} \sin(2 \pi n x) \Big). 
\end{align}
The series~\eqref{eq:y-inf-def}  converges almost surely in the  Sobolev space $H^{-s}([0,1])$ with $s>0$; see  \cite[Section 4.2]{Dub-JAMS}.
The series~\eqref{eq:f-tau-def}  converges almost surely in the uniform topology  thanks to  the term $q^{nm}$, hence $F_{\tau}$ is continuous on $[0,1]$. Note that by construction, $F_{\tau}$ and  $Y_{\infty}$ are independent.


Both $Y_\infty$ and $Y_\tau$ are examples of log-correlated fields, whose covariance kernels have a logarithmic singularity along the diagonal.
Although $Y_\infty$ is not pointwise defined, we use the intuitive notation $\EE[Y_{\infty}(x) Y_{\infty}(y)] $ to represent its covariance kernel. Namely, for appropriate  test functions $f_1, f_2$  
\begin{multline*}
\int_0^1 \int_0^1 dx dy f_1(x) f_2(y)  \EE[Y_{\infty}(x) Y_{\infty}(y)]=\\
\EE\left[ \left(  \int_0^1 dx  f_1(x) Y_{\infty}(x) \right) \left(\int_0^1 dy f_2(y) Y_{\infty}(y) \right) \right].
\end{multline*} 
See Appendix \ref{sec:gmc-app} for more background on log-correlated fields.

\begin{lemma} \label{lem:field-covar}
Let $\tau\in \ii\RR_{>0}$. The covariance kernels of $Y_\infty$ and $Y_\tau$ are:
\begin{align}
\EE[Y_{\infty}(x) Y_{\infty}(y)] &=  -2 \log |2 \sin(\pi (x - y))|, \label{eq:cov1}\\
\EE[Y_{\tau}(x) Y_{\tau}(y)] &= - 2 \log |\Theta_\tau(x - y)| + 2 \log|q^{1/6} \eta(q)| \label{eq:cov2}.
\end{align}
Furthermore, \xin{with probability 1 for each $x\in[0,1]$ the $q$-power series in~\eqref{eq:f-tau-def} defining $F_\tau(x)$ is convergent for $|q|<1$.}
Finally for $\tau\in \ii \RR_{>0}$,
\begin{equation}\label{eq:cov3}
\EE[F_{\tau}(x)^2] = 4 \sum_{n,m \geq 1} \frac{q^{2nm}}{n} = - 4 \log | q^{-1/12} \eta(q)|  \qquad \textrm{for each }  x\in[0,1].
\end{equation}
\end{lemma}
 
\begin{proof}
For~\eqref{eq:cov1}, notice that $\EE[Y_\infty(x) Y_\infty(y)]$ equals
\begin{multline*}
\EE[Y_\tau(x) Y_\infty(y)]= \sum_{n \geq 1} \frac{2}{n} \cos(2\pi n (x - y)) = - 2 \log |2 \sin (\pi (x - y))|.
\end{multline*}
For~\eqref{eq:cov2}, notice that
\begin{align*}
\EE&[Y_\tau(x) Y_\tau(y)] = \EE[Y_\infty(x) Y_\infty(y)] + \sum_{n, m \geq 1} \frac{4 q^{2nm}}{n} \cos(2\pi n (x - y))\\
&= - 2 \log |\Theta_\tau(x - y)| + 2 \log|q^{1/6} \eta(q)|. 
\end{align*}
\guillaume{In the last equality we have used the formulas for $\Theta_{\tau}$ and $\eta$ recalled in Appendix \ref{subsec:special}. Finally}
\xin{
the last two assertions on $F_{\tau}$ are immediate from the definition \eqref{eq:f-tau-def}.}
\end{proof}

\begin{remark}
The identity \eqref{eq:cov2} does not hold if $\tau\notin \ii\RR_{>0}$. 
Indeed,  $\EE[Y_{\tau}(x) Y_{\tau}(y)]$ is analytic in $\tau\in \bbH $ while the right hand side of~\eqref{eq:cov2} is not.
\end{remark}

We now introduce the Gaussian Multiplicative Chaos (GMC) measures $e^{\frac{\gamma}{2} Y_\infty(x)} dx$ and
$e^{\frac{\gamma}{2}Y_\tau(x)} dx$ on $[0,1]$ for $\gamma \in (0,2)$ and $\tau \in \ii\RR_{>0}$. Because the fields $Y_\infty(x)$
and $Y_\tau(x)$ are generalized functions that are not pointwise defined,  we use the following regularization
procedure. For $N \in \NN$, define 
\begin{align*}
Y_{\infty, N}(x) &= \sum_{n = 1}^N \sqrt{\frac{2}{n}} \Big( a_n \cos(2 \pi n x) + b_n \sin(2 \pi n x) \Big),\\
Y_{\tau, N}(x) &= Y_{\infty, N}(x) + 2\sum_{n, m = 1}^{\infty} \frac{q^{nm}}{\sqrt{n}} \Big(a_{n, m} \cos(2\pi n x) + b_{n,m}\sin(2\pi nx)\Big).
\end{align*}
\begin{define}[GMC] \label{def:gmc}
For $\gamma \in (0,2)$ and $\tau \in \ii\RR_{>0}$, we define 
\begin{align*}
e^{\frac{\gamma}{2} Y_\infty(x)} dx &:= \lim_{N \to \infty} e^{\frac{\gamma}{2} Y_{\infty, N}(x) - \frac{\gamma^2}{8} \EE[Y_{\infty, N}(x)^2]} dx;\\
e^{\frac{\gamma}{2} Y_\tau(x)} dx &:= \lim_{N \to \infty} e^{\frac{\gamma}{2} Y_{\tau, N}(x) - \frac{\gamma^2}{8} \EE[Y_{\tau, N}(x)^2]} dx.
\end{align*}
Here the convergence is in probability under the weak topology of measures on $[0,1]$.
We use $e^{\frac{\gamma}{2} Y_\infty(x)}dx$ and $e^{\frac{\gamma}{2} Y_\tau(x)}dx$ as convenient abuse of notation for the random measures obtained by this limiting procedure.
\end{define}

 In Definition~\ref{def:gmc}  $\gamma \in (0,2)$ is required for the regularization procedure to give non-trivial limits. Note that due to our normalization, the measures $e^{\frac{\gamma}{2} Y_\tau(x)} dx$ and $e^{ \frac{\gamma}{2} F_\tau(x) } e^{\frac{\gamma}{2} Y_{\infty}(x) } dx$
do not coincide.  Instead, by Definition~\ref{def:gmc} and \guillaume{Lemma~\ref{lem:field-covar}}, we have
\begin{equation}\label{eq:normalization}
e^{\frac{\gamma}{2} Y_\tau(x)} dx =e^{- \frac{\gamma^2}{8}\EE[F_\tau(0)^2]} e^{ \frac{\gamma}{2} F_\tau(x) } e^{\frac{\gamma}{2} Y_{\infty}(x) } dx \quad \textrm{for }\tau\in \ii\RR_{>0}.
\end{equation}
\xin{Although $e^{\frac{\gamma}{2} Y_\tau(x)}dx$ makes sense for all $\tau \in \mathbb{H}$, we restricted to $\tau\in \ii\RR_{>0}$  because this is the range relevant to the definition of the probabilistic conformal block below.  For more background on GMC see \cite{RV-GMC,NB-GMC} and our Appendix~\ref{sec:gmc-app}.}

\subsection{Definition and analyticity of $\cG^\alpha_{\gamma, P}(q)$}\label{subsec:GMC-block}
We are ready to  define the probabilistic conformal block $\cG^\alpha_{\gamma, P}(q)$.   
\begin{define}\label{def:conf-block}
For $ \alpha \in (-\frac{4}{\gamma},Q)$, $\tau \in \ii\RR_{>0}$, $q=e^{\ii \pi \tau}\in(0,1)$, and $P \in \RR$, we define  the \notion{probabilistic $1$-point toric conformal block} $ \cG_{\gamma, P}^{\alpha}(q)$ by
\xin{\begin{equation}\label{eq:def-G-block}
\cG_{\gamma, P}^{\alpha}(q) \defeq \frac{1} {Z_{\gamma, P}^{\alpha}(q)}
\EE\left[\left(\int_0^1  |\Theta_\tau(x)|^{-\frac{\alpha\gamma}{2}} e^{\pi \gamma P x} {e^{\frac{\gamma}{2} Y_\tau(x)}dx} \right)^{-\frac{\alpha}{\gamma}} \right],
\end{equation}
where the normalization $Z_{\gamma, P}^{\alpha}(q)$ is given by
\begin{multline}\label{eq:Z-normalizatoin}
Z_{\gamma, P}^{\alpha}(q) := q^{\frac{1}{12}(\frac{\alpha \gamma}{2} + \frac{\alpha^2}{2} - 1)} \times 
\eta(q)^{ \alpha^2 + 1 - \frac{\alpha \gamma}{2}} \\
\!\!\!\! \times  \EE\left[\left(\int_0^1 (2\sin(\pi x))^{-\frac{\alpha \gamma}{2}} e^{\pi \gamma P x} {e^{\frac{\gamma}{2} Y_\infty(x)}dx} \right)^{-\frac{\alpha}{\gamma}}\right].
\end{multline}}
\end{define}
Note that by Lemma~\ref{lem:GMC-moment}, for  $\alpha$, $q$, $P$  as in Definition~\ref{def:conf-block} we have
\begin{equation}\label{eq:moment1}
\EE\left[\left(\int_0^1  |\Theta_\tau(x)|^{-\frac{\alpha\gamma}{2}} e^{\pi \gamma P x} e^{\frac{\gamma}{2} Y_\tau(x)}dx \right)^{-\frac{\alpha}{\gamma}} \right]<\infty.
\end{equation}

The following proposition shows that $\cG_{\gamma, P}^{\alpha}(q)$ has the desired analytic continuation property  prescribed by Theorem~\ref{thm:main}.
\begin{prop}\label{prop:analytic-G}
For $\gamma \in (0, 2)$, $\alpha \in (-\frac{4}{\gamma}, Q)$,  and $P \in \RR$, the probabilistic conformal block  $\cG^\alpha_{\gamma, P}(q)$ admits an analytic extension on  a complex neighborhood of $q=0$. 	Moreover,  when $\alpha \in [0,Q)$, the analytic extension of $\cG^\alpha_{\gamma, P}(q)$ exists on an open set containing $ \{ z \in \mathbb{C} : |z| < \frac{1}{2} \} \cup [0,1)$.
\end{prop}
Before proving Proposition~\ref{prop:analytic-G}, we introduce a variant of $\cG^\alpha_{\gamma, P}(q)$ which is more convenient. 
For $\alpha \in (-\frac{4}{\gamma}, Q)$, $q \in (0, 1)$, and $P\in\RR$, define	
\xin{\begin{multline} \label{eq:ca-def}
\cA^q_{\gamma, P}(\alpha) := q^{\frac{1}{12}(- \alpha \gamma - \frac{2\alpha}{\gamma} + 2)} \times 
\eta(q)^{\alpha \gamma + \frac{2\alpha}{\gamma} - \frac{3}{2}\alpha^2 - 2} \times e^{\ii\alpha^2\pi/2} \\
\!\!\!\! \times  \EE\left[\Big(\int_0^1 |\Theta_\tau(x)|^{-\frac{\alpha\gamma}{2}} e^{\pi \gamma P x} e^{\frac{\gamma}{2} Y_\tau(x)} dx \Big)^{-\frac{\alpha}{\gamma}} \right].
\end{multline}}%
Here we use the notation $\cA^q_{\gamma, P}(\alpha)$ instead of $\cA^\alpha_{\gamma, P}(q)$ because in later sections we mostly view $\cA^q_{\gamma, P}(\alpha)$
as a function of $\alpha$ with $q,\gamma,P$ being parameters. 
\begin{lemma} \label{lem:a-props}
Fix $\gamma\in(0,2)$ and $P \in \mathbb{R}$. The following hold:
\begin{itemize}
\item[(a)]  For $\alpha \in (-\frac{4}{\gamma}, Q)$, the function $q\mapsto \cA^q_{\gamma, P}(\alpha)$ 
admits an analytic extension on a complex neighborhood of $q=0$. Moreover, when $\alpha \in [0,Q)$,
the analytic extension exists  on an open set containing $ \{ z \in \mathbb{C} : |z| < \frac{1}{2} \} \cup [0,1)$.

\item[(b)] There exists an open set in $\CC^2$ containing
$\{(\alpha, q) : \alpha \in (-\frac{4}{\gamma}, Q) \textrm{ and } q = 0\}$ on which
$ (\alpha,q)  \mapsto \cA^q_{\gamma, P}(\alpha)$ admits an analytic extension. As a consequence, for $\alpha\in (-\frac{4}{\gamma}, Q)$, consider $\{\cA_{\gamma, P, n}(\alpha)\}_{n\ge 0}$ defined by
\begin{equation} \label{eq:a-exp}
\cA^q_{\gamma, P}(\alpha) = \sum_{n = 0}^\infty \cA_{\gamma, P, n}(\alpha) q^n,\qquad \textrm{for } |q| \textrm{ sufficiently small}.
\end{equation}
Then for $n\in\NN_0$,  $\alpha\mapsto \cA_{\gamma, P,n}(\alpha)$ can be analytically extended
to a complex neighborhood of $(-\frac{4}{\gamma}, Q)$.

\end{itemize}
\end{lemma}
We postpone the proof of Lemma~\ref{lem:a-props} to Section~\ref{subsec:analytic-bl} and proceed to
derive Proposition~\ref{prop:analytic-G} from it. Define normalized versions of $\cA^q_{\gamma, P}$
and $\cA_{\gamma, P, n}$ by
\begin{equation}\label{eq:Atitlde}
\wcA^q_{\gamma, P}(\alpha) := \frac{\cA^q_{\gamma, P}(\alpha)}{\cA_{\gamma, P, 0}(\alpha)} \qquad \text{ and }
\qquad \wcA_{\gamma, P, n}(\alpha) := \frac{\cA_{\gamma, P, n}(\alpha)}{\cA_{\gamma, P, 0}(\alpha)}.
\end{equation}
\begin{proof}[Proof of Proposition~\ref{prop:analytic-G} given Lemma~\ref{lem:a-props}]

Eq~\eqref{eq:ca-def}, \eqref{eq:a-exp}, and~\eqref{eq:Atitlde} give
\revise{\begin{align}
Z_{\gamma, P}^{\alpha}(q)&= q^{\frac{1}{12}(\frac{\alpha \gamma}{2} + \frac{\alpha^2}{2} - 1)} \eta(q)^{ \alpha^2 + 1 - \frac{\alpha \gamma}{2}} \guillaume{e^{- \ii \alpha ^2 \pi/2}} \mathcal{A}_{\gamma,P,0}(\alpha),\label{eq:link-Z-A}\\
\cG_{\gamma, P}^{\alpha}(q) &=  \left(q^{-\frac{1}{12}}\eta(q)\right)^{1 - \alpha(Q - \frac{\alpha}{2})} \wcA^q_{\gamma, P}(\alpha). \label{eq:block-def}
\end{align}}%
Lemma~\ref{lem:eta} yields that $\left(q^{-\frac{1}{12}}\eta(q)\right)^{1 - \alpha(Q - \frac{\alpha}{2})}$ is a convergent power series for $|q|<1$. Using~\eqref{eq:block-def} and Lemma~\ref{lem:a-props} (a), we get Proposition~\ref{prop:analytic-G}.
\end{proof}

\begin{remark}\label{rmk:Z}
\xin{In Proposition \ref{prop:DenominatorBlocks}
we will give an explicit formula for $\mathcal{A}_{\gamma, P, 0}(\alpha)$, which by~\eqref{eq:link-Z-A} yields an explicit formula  for $Z_{\gamma, P}^{\alpha}(q)$ in Definition \ref{def:conf-block}.}
\end{remark}

\subsection{Nekrasov partition function and Zamolodchikov's recursion}
We now define the $1$-point toric conformal block in physics  via the AGT correspondence. The corresponding Nekrasov partition function is the formal $q$-series
\begin{equation} \label{eq:nek-def}
\cZ^\alpha_{\gamma, P}(q) := 1 + \sum_{n = 1}^\infty \cZ_{\gamma, P, n}(\alpha) q^{2n},
\end{equation}
where
\begin{equation} \label{eq:nek-exp-def}
\cZ_{\gamma, P, n}(\alpha) := \sum_{\substack{(Y_1, Y_2) \text{ Young diagrams} \\ |Y_1| + |Y_2| = n}}
\prod_{i, j = 1}^2 \prod_{s \in Y_i} \frac{(E_{ij}(s, P) - \alpha)(Q - E_{ij}(s, P) - \alpha)}{E_{ij}(s, P) (Q - E_{ij}(s, P))}
\end{equation}
for
\begin{equation} \label{eq:nek-eqn}
E_{ij}(s, P) := \begin{cases} \ii P - \frac{\gamma}{2} H_{Y_j}(s) + \frac{2}{\gamma} (V_{Y_i}(s) + 1) & i = 1, j = 2 \\
- \frac{\gamma}{2} H_{Y_j}(s) + \frac{2}{\gamma} (V_{Y_i}(s) + 1) & i = j \\
-\ii P - \frac{\gamma}{2} H_{Y_j}(s) + \frac{2}{\gamma} (V_{Y_i}(s) + 1) & i = 2, j = 1. \end{cases}
\end{equation}
Here, we draw a Young diagram $Y$ for a partition $\lambda$ in the first quadrant
with unit squares so the top right corner of each square has positive coordinates. In (\ref{eq:nek-eqn}),
for a unit square $s$ with top right corner $(i, j)$, we define $H_Y(s) = \lambda_j' - i$ and
$V_Y(s) = \lambda_i - j$, where $\lambda'$ is the transposed partition to $\lambda$. 
\begin{define}\label{def:block-series}
The \notion{$1$-point toric conformal block} is the formal $q$-series 
\begin{equation} \label{eq:block-def-nek}
\cF^\alpha_{\gamma, P}(q) := [q^{-\frac{1}{12}}\eta(q)]^{1 - \alpha(Q - \frac{\alpha}{2})} \cZ^\alpha_{\gamma, P}(q).
\end{equation}
\end{define}%
Recall from Lemma~\ref{lem:eta} that for each $\beta\in\RR$,
$[q^{-\frac{1}{12}} \eta(q))]^\beta$ is a power series in $q$ convergent for $|q|<1$. So the right side of~\eqref{eq:block-def-nek} indeed is a power series.  We will not use the precise expression~\eqref{eq:nek-exp-def} for $\cZ^\alpha_{\gamma, P}(q)$
beyond the fact that $\cZ_{\gamma, P, k}(\alpha)$ is a rational function in $P,Q,\alpha$.

We now review the Zamolodchikov's
recursion characterizing  $\cF^\alpha_{\gamma, P}(q)$.  Let
\begin{equation} \label{eq:rmn-def}
R_{\gamma, m, n}(\alpha) := \frac{2 \prod\limits_{j = -m}^{m - 1} \prod\limits_{l = -n}^{n - 1} (Q - \frac{\alpha}{2} + \frac{j \gamma}{2} + \frac{2l}{\gamma})}{\prod\limits_{(j, l) \in S_{m, n}} (\frac{j\gamma}{2} + \frac{2l}{\gamma})}
\end{equation}
for $S_{m, n} := \{(j, l) \in \ZZ^2 \mid 1 - m \leq j \leq m, 1 - n \leq l \leq n, (j, l) \notin \{(0, 0), (m, n)\}\}$ and
\begin{equation} \label{eq:pmn-def}
P_{m, n} := \frac{2 \ii n}{\gamma} + \frac{\ii \gamma m}{2}.
\end{equation}
\begin{prop}[Zamolodchikov's recursion] \label{prop:torus-z-recursion}
The formal $q$-series  $\cF^\alpha_{\gamma, P}(q)$ defined by~\eqref{eq:block-def-nek} satisfies
\begin{equation} \label{eq:z-recur}
\cF^\alpha_{\gamma, P}(q) = \sum_{n, m = 1}^\infty
q^{2mn} \frac{R_{\gamma, m, n}(\alpha)}{P^2 - P_{m, n}^2} \cF^\alpha_{\gamma, P_{-m, n}}(q)
+ q^{\frac{1}{12}} \eta(q)^{-1}.
\end{equation}
\end{prop}

Proposition \ref{prop:torus-z-recursion} is a concrete identity in terms of the rational functions $\cZ_{\gamma, P, k}(\alpha)$
defined in~\eqref{eq:nek-exp-def}. 
This is proven  rigorously in an elementary way in \cite[Section~2]{FL10}, although overall  \cite{FL10} is a  theoretical physics paper.    
The fact that $\cF^\alpha_{\gamma, P}(q)$ from  Definition~\ref{def:block-series} agrees with
the original definition  in \cite{BPZ84} 
via the Virasoro algebra was also proven in~\cite{FL10,N16}.  


\begin{remark}
We parameterize the conformal block using $P$ and $\alpha$ because these
are convenient coordinates for our GMC expressions.  In mathematical physics,
it is more common to represent it as a function of \emph{conformal dimension}
$\Delta_\alpha = \frac{\alpha}{2}(Q - \frac{\alpha}{2})$ corresponding to
\emph{momentum} $\alpha$ and the \emph{intermediate dimension} $\Delta = \frac{1}{4}(Q^2 + P^2)$
corresponding to momentum $Q + \ii P$. 
\end{remark}

\subsection{A one-step reduction}Recall $\wcA^q_{\gamma, P}(\alpha)$ in~\eqref{eq:Atitlde} which by~\eqref{eq:a-exp} in Proposition~\ref{prop:analytic-G} can be viewed as a $q$-series with coefficients  $\wcA_{\gamma, P, n}(\alpha)$.  
We now  reduce Theorem~\ref{thm:main} to a statement on 
$\wcA^q_{\gamma, P}(\alpha)$ whose proof occupies the remainder of this paper. 
\begin{theorem} \label{thm:nek-block}
For $\gamma \in (0, 2)$, $\alpha \in (-\frac{4}{\gamma}, Q)$, and $P \in \RR$,  we have
\begin{equation} \label{eq:nek-block} 
\cZ^\alpha_{\gamma, P}(q) = \wcA^q_{\gamma, P}(\alpha)
\end{equation}
as formal $q$-series.
Namely, $\cZ_{\gamma, P,\revise{n}}(\alpha)= \wcA_{\gamma, P, n}(\alpha) $ for all $n\ge 1$.
\end{theorem}

\begin{proof}[Proof of Theorem~\ref{thm:main} given Theorem~\ref{thm:nek-block}]
For $\gamma \in (0, 2)$, $\alpha \in (-\frac{4}{\gamma}, Q)$, and $P \in \RR$, by
Theorem \ref{thm:nek-block}, \eqref{eq:block-def} and (\ref{eq:block-def-nek}),  we
have  $\cF^\alpha_{\gamma, P}(q) = \cG^\alpha_{\gamma, P}(q)$ as  formal $q$-series.
Combined with Proposition~\ref{prop:analytic-G}, we obtain Theorem~\ref{thm:main}.
\end{proof}
%

\subsection{Analyticity: proof of Lemma~\ref{lem:a-props}} \label{subsec:analytic-bl}
The \xin{following lemma based on Cameron-Martin's theorem is the starting point of our proof.}
\begin{lemma}\label{lem:reduce} \xin{We have 
 $\cA^q_{\gamma, P}(\alpha)=[q^{-\frac1{12}}\eta(q)]^{\alpha (Q-\frac{\alpha}{2}) -2}e^{\ii\alpha^2\pi/2} \hat\cA^q_{\gamma, P}(\alpha)$ where 
\begin{equation*} 
\hat\cA^q_{\gamma, P}(\alpha)= \mathbb{E}    \left[e^{\frac{\alpha}{2}F_\tau(0)} \left( \int_0^1 e^{\frac{\gamma}{2}F_\tau(x)  } (2\sin(\pi x))^{-\alpha\gamma/2}  e^{ \pi \gamma P x } e^{\frac{\gamma}{2} Y_{\infty}(x)}dx  \right)^{- \frac{\alpha}{\gamma}} \right]
\end{equation*}
and where $Y_\infty$ is as in~\eqref{eq:y-inf-def} and $F_\tau$ is as in~\eqref{eq:f-tau-def}.}
\end{lemma}
\guillaume{
\begin{proof}
First note that since $Y_{\tau} = Y_{\infty} + F_{\tau}$ with $Y_{\infty}$ and  $F_{\tau}$ independent, by combining equations \eqref{eq:cov1} and \eqref{eq:cov2} one can write:
\begin{align*}
|\Theta_{\tau}(x)|^{-\frac{\alpha \gamma}{2}} = (q^{1/6} \eta(q))^{- \frac{\alpha \gamma}{2}} (2\sin(\pi x))^{-\alpha\gamma/2}  e^{ \frac{\gamma}{2} \EE[Y_{\tau}(x) \cdot \frac{\alpha}{2}F_{\tau}(0)] }.
\end{align*}
	By using this formula and Cameron-Martin's Theorem~\ref{thm:Girsanov} one obtains:
	\begin{align*}
	& \mathbb{E} \left[ \left( \int_0^1 |\Theta_{\tau}(x)|^{-\frac{\alpha \gamma}{2}}  e^{ \pi \gamma P x }  e^{\frac{\gamma}{2} Y_{\tau}(x) }dx  \right)^{- \frac{\alpha}{\gamma}} \right] \nonumber\\
		& = \left( q^{1/6} \eta(q)\right)^{\frac{\alpha^2}{2}} e^{- \frac{\alpha^2}{8}\EE[F_{\tau}(0)^2] }
		\mathbb{E} \left[e^{\frac{\alpha}{2}F_\tau(0)} \left( \int_0^1 (2\sin(\pi x))^{-\alpha\gamma/2}  e^{ \pi \gamma P x } e^{\frac{\gamma}{2} Y_{\tau}(x)   } dx  \right)^{- \frac{\alpha}{\gamma}} \right].\nonumber
	\end{align*}
	Next by using equation~\eqref{eq:normalization} and then \eqref{eq:cov3} the previous line becomes:
	\begin{align*}
		&\left( q^{1/6} \eta(q)\right)^{\frac{\alpha^2}{2}} e^{{(\frac{\alpha\gamma}{8} - \frac{\alpha^2}{8})}\EE[F_{\tau}(0)^2] }
		\hat \cA^q_{\gamma,P}(\alpha)= \left( q^{1/6} \eta(q)\right)^{\frac{\alpha^2}{2}} (q^{-1/12} \eta(q))^{\frac{\alpha^2-\alpha\gamma}{2}}\hat \cA^q_{\gamma,P}(\alpha).
	\end{align*}
	Then by the following simplification
\begin{align*}
&q^{\frac{1}{12}(- \alpha \gamma - \frac{2\alpha}{\gamma} + 2)} \times 
\eta(q)^{\alpha \gamma + \frac{2\alpha}{\gamma} - \frac{3}{2}\alpha^2 - 2} \times e^{\ii\alpha^2\pi/2} \times \left( q^{1/6} \eta(q)\right)^{\frac{\alpha^2}{2}} (q^{-1/12} \eta(q))^{\frac{\alpha^2-\alpha\gamma}{2}}\\
&= e^{\ii\alpha^2\pi/2} (q^{-\frac1{12}}\eta(q))^{\alpha (Q-\frac{\alpha}{2}) -2},
\end{align*}
we deduce the claim of the lemma.
\end{proof}
}
\xin{ By Lemma~\ref{lem:eta} we can analytically extend $[q^{-\frac1{12}}\eta(q)]^{\alpha (Q-\frac{\alpha}{2}) -2}$   to $q\in \D$. Thus by Lemma~\ref{lem:reduce} we can reduce Lemma~\ref{lem:a-props} (a) to the following.}
\begin{lemma}\label{lem:hatA}
Lemma~\ref{lem:a-props} (a) holds  with $\hat\cA^q_{\gamma, P}(\alpha)$ in place of $\cA^q_{\gamma, P}(\alpha)$.
\end{lemma}

\guillaume{To prove Lemma~\ref{lem:hatA}, first notice that by equation \eqref{eq:y-inf-def} one has:

\begin{equation}\label{eq:sin-cos}
		\mathbb{E}[\alpha_n Y_{\infty}(x) ] = \sqrt{\frac{2}{n}} \cos(2 \pi n x) \quad \textrm{and}\quad
		\mathbb{E}[\beta_n Y_{\infty}(x) ] = \sqrt{\frac{2}{n}} \sin(2 \pi n x).
	\end{equation}
}
\xin{Recall~\eqref{eq:f-tau-def} that $F_{\tau}(x)=  2 \sum_{n,m \geq 1} \frac{q^{nm}}{\sqrt{n}} \left( a_{n,m} \cos(2 \pi n x) + b_{n,m} \sin(2 \pi n x) \right)$, where $a_{n,m},b_{n,m}$ are i.i.d standard Gaussians.} 
By~\eqref{eq:sin-cos},  
\begin{equation}\label{eq:Ftau}
\revise{F_\tau(x)} = \sqrt{2} \sum_{m,n = 1}^\infty q^{nm} ( a_{n,m} \mathbb{E}[a_n Y_{\infty}(x) ]  + b_{n,m} \mathbb{E}[b_n Y_{\infty}(x) ] ).
\end{equation}
Applying Cameron-Martin's theorem (Theorem \ref{thm:Girsanov}) to $Y_\infty$ while conditioning on $\{a_{m,n}, b_{m,n} \}$, we obtain
\begin{align} \nonumber
\hat A^q_{\gamma,P}(\alpha) &=
\EE\Big[ e^{\frac{\alpha}{2} F_{\tau}(0)} \left( \int_0^1 \sin(\pi x)^{-\frac{\alpha\gamma}{2}} e^{ \pi \gamma P x }\right.\\ 
&\phantom{===} \left.\revise{e^{\frac{\gamma}{2} Y_{\infty}(x)}} e^{ \frac{\gamma}{\sqrt{2}}
\sum_{m,n = 1}^\infty q^{nm} ( a_{n,m} \mathbb{E}[a_n Y_{\infty}(x) ]  + b_{n,m} \mathbb{E}[b_n Y_{\infty}(x) ] )  } dx  \right)^{- \frac{\alpha}{\gamma}} \Big] \nonumber\\
&= \EE\Big[ e^{\frac{\alpha}{2} F_{\tau}(0)}\cQ(q)
\big( \int_0^1 \sin(\pi x)^{-\alpha\gamma/2}  e^{ \pi \gamma P x } e^{\frac{\gamma}{2} Y_{\infty}(x) }dx   \big)^{- \frac{\alpha}{\gamma}} \Big], \label{eq:hatA-expression}
\end{align}
where
\begin{multline}\label{eq:Qq}
\cQ(q)\defeq \exp\bigg(\sqrt{2}  \sum_{m,n = 1}^\infty q^{nm} ( a_{n,m} a_n  + b_{n,m} b_n)\\ - \sum_{n = 1}^\infty \Big(\sum_{m = 1}^\infty q^{ nm} a_{m,n} \Big)^2 - \sum_{n = 1}^\infty\Big( \sum_{m = 1}^\infty q^{ nm} b_{m,n} \Big)^2 \bigg).
\end{multline}

Although $\hat A^q_{\gamma,P}(\alpha)$ is originally only defined for $q\in (0,1)$, \xin{the function
$e^{\frac{\alpha}{2} F_{\tau}(0)} \cQ(q)$ contains all the $q$ dependence and is almost surely an analytic
function defined for $|q|<1$. 
We now record a basic fact on the analyticity of expectations.}
\begin{lemma}\label{lem:analytic-basic}
Let $f(\cdot)$ be a random analytic function on a planar domain $D$. Suppose for each compact   $K\subset D$ we have $\max_{z\in K}\E[|f(z)|]<\infty$.
Then $\EE[f(\cdot)] $  is analytic on $D$. Moreover,  $\EE[\left|\frac{d^n}{d\revise{z^n}}f(z)\right|]<\infty$  and $\frac{d^n}{dz^n}\EE[f(z)]=\EE[\frac{d^n}{d\revise{z^n}}f(z)]$ for each $z\in D$ and $n\in \NN$.
\end{lemma}
\begin{proof}
Consider $K_0=\{z: |z-z_0|\le r \}\subset D$ for some $z_0\in D$ and $r>0$. Let $M_0=\max_{z\in K_0}\E[|f(z)|]<\infty$.
Since $\frac{d^n}{d\revise{z^n}}f(z_0)= \frac{n!}{2\pi \ii}\oint_{\partial K_0} f(w)(w-z_0)^{-n-1}dw$,  we have $\EE[\left|\frac{d^n}{d\revise{z^n}}f(z_0)\right|]\le n! M_0r^{-n}$. Therefore, by Fubini's Theorem, if $|z-z_0|<r$ then
$\EE[f(z)]=\sum_{n=0}^{\infty}\frac{1}{n!}\EE[\frac{d^n}{d\revise{z^n}}f(z_0)](z-z_0)^n$. Varying $z_0$ and $r$, we conclude.
\end{proof}

In light of Lemma~\ref{lem:analytic-basic}, for each $\alpha\in(-\frac{4}{\gamma},Q)$
and open set $U\subset \D$, the function $q\mapsto \hat A^q_{\gamma,P}(\alpha)$ admits an analytic extension
on $U$ if we have, for every compact $K \subset U $: 
\begin{equation}\label{eq:upper-bound}
\underset{q \in K}{\sup} \; \EE\Big[\left| e^{\frac{\alpha}{2} F_{\tau}(0)}\cQ(q) \right|
\Big( \int_0^1 \sin(\pi x)^{-\alpha\gamma/2}  e^{ \pi \gamma P x } e^{\frac{\gamma}{2} Y_{\infty}(x) }dx \Big)^{- \frac{\alpha}{\gamma}}\Big]<\infty.
\end{equation}
\xin{We will use H\"{o}lder's inequality to prove~\eqref{eq:upper-bound}. The key is the moment bounds for $\cQ(q)$ summarized in the following lemma.}

\begin{lemma} \label{lem:q-bound}
\xin{Firstly, we have $\max_{|q|<\frac12,\;  p\in [1,1.2]}  \EE [|\cQ(q)|^{p}] <\infty$.
Moreover, there exists an open neighborhood $U$ of $[0,1)$ such that 
\begin{equation}\label{eq:Q2bound}
         \limsup_{p \to 1^+}  \max_{q\in K}\EE[|\cQ(q)|^{p}]<\infty \quad \textrm{for each compact } K\subset U.
\end{equation}
Finally, for each $M>1$ there exists $\eps_M>0$, such that 
\begin{equation}\label{eq:Q3bound}
 \max_{|q|<\eps_M,\;  p\in [1,M]}  \EE [|\cQ(q)|^{p}] <\infty.       
\end{equation}}
\end{lemma}
\begin{proof}
Using  the i.i.d Gaussians $(a_n)_{n \geq 1}$, $(a_{n,m})_{n,m \geq 1}$ in~\eqref{eq:Qq}, we define:
\[
A_n(q,p) := \EE \left[ \left|e^{ \sqrt{2}  \sum_{m = 1}^\infty q^{nm}   {a_{n,m} a_n } } e^{-  \sum_{m_1,m_2 = 1}^\infty q^{ nm_1 +n m_2}  {a_{m_1,n} a_{m_2,n} }   } \right|^{p} \right].
\]
Then we can write $\EE[|\cQ(q)|^{p}]$ as:
\begin{align*}
&\EE\left[ |\cQ(q)|^{p} \right]\\
&= \EE\left[ \left|e^{ \sqrt{2}  \sum_{m,n = 1}^\infty q^{nm}  {( a_{n,m} a_n  + b_{n,m} b_n ) } - \sum_{n = 1}^\infty \left( \sum_{m = 1}^\infty q^{ nm}  {a_{m,n}} \right)^2 - \sum_{n = 1}^\infty \left( \sum_{m = 1}^\infty q^{ nm} {b_{m,n}} \right)^2 } \right|^{p} \right]\\
&=  \EE \left[ \left|e^{ \sqrt{2}  \sum_{m,n = 1}^\infty q^{nm}  {( a_{n,m} a_n  + b_{n,m} b_n )} - \sum_{n = 1}^\infty \sum_{m_1,m_2 = 1}^\infty q^{ nm_1 +n m_2} {(a_{m_1,n} a_{m_2,n}  + b_{m_1,n} b_{m_2,n})}  } \right|^{p} \right]\\
&= \prod_{n = 1}^\infty A_n(q, p)^2.
\end{align*}
Here we uses that { $(a_n)_{n \geq 1}$, $(a_{n,m})_{n,m \geq 1}$, $(b_n)_{n \geq 1}$, and $(b_{n,m})_{n,m \geq 1}$,} are independent and identically distributed. We now compute 
\begin{align*}
A_n(q, p)  &= \EE \left[ e^{ \sqrt{2}  \sum_{m = 1}^\infty p \mathrm{Re}(q^{nm}) { a_{n,m} a_n}  } e^{-  \sum_{m_1,m_2 = 1}^\infty p \mathrm{Re}(q^{ nm_1 +n m_2}) {a_{m_1,n} a_{m_2,n} }   }  \right]\\
&= \EE \left[ e^{\sum_{m_1,m_2 = 1}^\infty \left( p^2 \mathrm{Re}(q^{nm_1}) \mathrm{Re}(q^{nm_2}) - p \mathrm{Re}(q^{ nm_1 +n m_2}) \right)  {a_{m_1,n} a_{m_2,n}}} \right]\\
&= \EE \left[ e^{\sum_{m_1,m_2 = 1}^\infty \left((p^2 - p) \mathrm{Re}(q^{nm_1}) \mathrm{Re}(q^{nm_2}) + p \Im(q^{nm_1}) \Im(q^{n m_2}) \right)  {a_{m_1,n} a_{m_2,n}}} \right],
\end{align*}
where in the second line we computed the expectation over {$a_n$.}  Define now the Gaussian random variables
\[
X_n := \sum_{m = 1}^\infty \Re(q^{nm}) {a_{n, m}} \qquad \text{and} \qquad Y_n := \sum_{m = 1}^\infty \Im(q^{nm}) {a_{n, m}}.
\]
Then $(X_n, Y_n)$ is a bivariate Gaussian with covariance matrix $\left[\begin{matrix} R_n & S_n \\ S_n & T_n \end{matrix}\right]$
for
\[
R_n = \sum_{m = 1}^\infty \Re(q^{nm})^2, \quad S_n =  \sum_{m = 1}^\infty \Re(q^{nm}) \Im(q^{nm}), \quad\textrm{and}\quad T_n = \sum_{m = 1}^\infty \Im(q^{nm})^2.
\]
We find that $(X_n, Y_n) \overset{d} = (\sqrt{R_n} Z, \frac{S_n}{\sqrt{R_n}} Z + \frac{\sqrt{R_nT_n - S_n^2}}{\sqrt{R_n}} W)$, where $Z, W$
are independent standard Gaussians. In these terms, we have that
\begin{align*}
A_n(q, p) &= \EE\left[e^{(p^2 - p) X_n^2 + p Y_n^2}\right]
= \EE\left[\exp\left(\left[\begin{matrix} Z \\ W \end{matrix}\right]^\T M_n \left[\begin{matrix} Z \\ W \end{matrix}\right]\right)\right]
\end{align*}
with the matrix
\begin{equation} \label{eq:m-def}
M_n = \left[\begin{matrix} (p^2 - p) R_n + p \frac{S_n^2}{R_n} & \frac{p S_n\sqrt{R_nT_n - S_n^2}}{R_n}\\
\frac{p S_n\sqrt{R_nT_n - S_n^2}}{R_n} & p \frac{R_nT_n - S_n^2}{R_n}\end{matrix}\right].
\end{equation}
\guillaume{Notice that:
\[
\Tr(M_n) = (p^2 - p)R_n + p T_n \quad\textrm{and}\quad  \det(M_n) = p^2(p-1) (R_n T_n -S_n^2).
\]
This implies that $\Tr(M_n) > 0$ and $\det(M_n) > 0$.}
\xin{We now prove our first bound on $\EE[|\cQ(q)|^{p}]$.
Let $\lambda_1$ and $\lambda_2$ be the two eigenvalues of $M_n$. Let $W_1$ and $W_2$ be two independent standard Gaussians. When $\Tr(M_n)<\frac12$, we have
\begin{equation}
A_n(q, p) = \EE[e^{\lambda_1 W_1^2 + \lambda_2 W_2^2}]
= \frac{1}{\sqrt{1 - 2 \Tr(M_n) + 4 \det(M_n)}}.    
\end{equation} 
Since $\EE[|\cQ(q)|^{p}] =\prod_{n = 1}^\infty A_n(q, p)^2$ and $\det(M_n)>0$,   we have
\begin{equation}\label{eq:Tr-bound}
 \EE[|\cQ(q)|^{p}]  \le \prod_{n=1}^\infty\frac{1}{1 - 2 \Tr(M_n)}\textrm{ when } \Tr(M_n)<\frac12\textrm{ for all }n.
\end{equation} 
Since $R_n\ge 0$ and 
\begin{equation}
 R_n+T_n = \sum_{m=1}^{\infty}|q|^{2nm}=\frac{|q|^{2n}}{1 - |q|^{2n}},
\end{equation}
for $p\in [1,1.2]$ and $|q|<\frac12$ we have 
\begin{align}\label{eq:conv-b1}
\Tr(M_n)  \leq  (p^2 - 2p)R_n + p (R_n+T_n)\le 
\frac{p^2 |q|^{2n}}{1 - |q|^{2n}}\le 0.48.
\end{align}
For $C_0$ small enough such that $\frac1{1-2x}\le e^{-C_0x}$ for $x\in [0,0.48]$, by~\eqref{eq:Tr-bound} we have $ \EE[|\cQ(q)|^{p}]\le e^{-C_0\sum_{n=1}^\infty \Tr(M_n)}$. Therefore~\eqref{eq:conv-b1} then yields that
\begin{equation}\label{eq:p12}
\sum_{n=1}^\infty \Tr(M_n)  \le \sum_{n=1}^\infty \frac{p^2 |q|^{2n}}{1 - |q|^{2n}}\le 1   \textrm{ for }p\in [1,1.2]  \textrm{ and } |q|<\frac12.
\end{equation} 
Therefore   $\EE[|\cQ(q)|^{p}]<e^{-C_0}\le 1$ for $p\in [1,1.2]$ and $|q|<\frac12$ as desired. 

To prove~\eqref{eq:Q2bound}, let $f(q)=\sum_{n=1}^\infty T_n$ which is continuous in $q\in \D$. Since $T_n = 0$ for $q \in [0, 1)$, we have $f(q)=0$ for $q\in [0,1)$. 
Therefore, there exists a neighborhood $U$ of
$[0, 1)$ such that $f(q)\in [0,0.1]$ for each $q\in U$. Since 
\[
\Tr(M_n)  \leq  (p^2 - p)(R_n+T_n) + (2p-p^2)T_n \le 
\frac{p(p-1) |q|^{2n}}{1 - |q|^{2n}}+f(q),
\]
for each compact $K\subset U$,  $ \limsup_{p \rightarrow 1^+}  \max_{q\in K, n\in \mathbb N}\Tr(M_n)< 0.1$.
Since 
\begin{equation}\label{eq:Trsum}
\Tr(M_n) \le (p^2+p)(R_n+T_n)=\frac{(p^2+p)|q|^{2n}}{1-|q|^{2n}},
\end{equation}
 we get~\eqref{eq:Q2bound} for this choice of $U$.

It remains to prove the last bound.  For each $M>1$, we can find $\eps\in (0,0.1)$ such that $\Tr(M_n)<0.1$  for $|q|<\eps$ and $p\in [1,M]$. Now 
by~\eqref{eq:Trsum},  we have $\max_{|q|<\eps}  \sum_{n=1}^\infty\Tr (M_n) <\infty$ hence 
$\max_{|q|<\eps,\;  p\in [1,M]}  \EE |\cQ(q)|^{p}] <\infty$.} 
\end{proof}

\begin{proof}[Proof of Lemma~\ref{lem:hatA}] 
\xin{It suffices to prove (\ref{eq:upper-bound}),}
By H\"{o}lder inequality, for  $p_1, p_2, p_3 \in (1, \infty)$ with $\frac{1}{p_1} + \frac{1}{p_2} + \frac{1}{p_3} = 1$ we have
\begin{align}\label{eq:sec2_holder}
& \EE \left[ \left|e^{\frac{\alpha}{2} F_{\tau}(0)}\cQ(q) \right|
\Big( \int_0^1 \sin(\pi x)^{-\alpha\gamma/2}  e^{ \pi \gamma P x } e^{\frac{\gamma}{2} Y_{\infty}(x) }dx\Big)^{- \frac{\alpha}{\gamma}} \right]\\ \nonumber
\leq  \EE\left[e^{ \frac{p_1 |\alpha|}{2}  \left|F_\tau(0)\right|} \right]^{\frac{1}{p_1}}  &  \EE\left[|\cQ(q)|^{p_2}\right]^{\frac{1}{p_2}}  
\EE\left[\Big(\int_0^1\sin(\pi x)^{-\alpha\gamma/2}e^{\pi \gamma P x} e^{\frac{\gamma}{2} Y_{\infty}(x) }dx\Big)^{- \frac{\alpha p_3}{\gamma}} \right]^{\frac{1}{p_3}}.
\end{align}
\xin{By the finiteness of exponential moments of Gaussian random variables, we have $\EE \left[ e^{ \frac{p_1 |\alpha|}{2} |F_{\tau}(0)|} \right]^{\frac{1}{p_1}} < \infty$ for all $|q|<1$ and $p_1 \in (1,\infty)$. By Item 1 of Lemma~\ref{lem:GMC-moment}, the third expectation in the last line is finite if and only if  $ -\frac{\alpha p_3}{\gamma}<\frac{4}{\gamma^2} \wedge \frac{2}{\gamma} (Q -\alpha)$. This holds true} 
\begin{equation}\label{eq:p3}
\xin{\textrm{for $p_3>1$  if $\alpha\in [0,Q)$; and  for $p_3\in (1,-\frac{4}{\alpha\gamma})$ if $\alpha\in (-\frac{4}{\gamma},0)$.}}
\end{equation}
We now divide the proof into two cases.

\noindent \emph{\bf Case 1: $\alpha \in [0, Q)$.}  In this case, we need to prove (\ref{eq:upper-bound}) for
both $U = \{q \in \CC : |q| < \frac{1}{2}\}$ and $U$ being a small enough complex neighborhood of $[0, 1)$. 
\xin{By Lemma~\ref{lem:q-bound}, for both choices of $U$, we have that $\limsup_{p_2\to 1^{+}} \max_{q\in K}  \EE\left[|\cQ(q)|^{p_2}\right] < \infty$ 
for each compact $K\subset U$. Therefore we can find  $p_2 \in (1, \infty)$
close to $1$ and $p_1 \in (1,\infty)$, and $p_3 \in (1,\infty)$
such that $\frac{1}{p_1} + \frac{1}{p_2} + \frac{1}{p_3} = 1$.
This gives~\eqref{eq:upper-bound} for both choices of $U$ hence  Lemma~\ref{lem:hatA} for $\alpha \in [0, Q)$.}

\noindent \emph{\bf Case 2: $\alpha \in (-\frac{4}{\gamma}, 0)$.}  In this
case, we need to prove (\ref{eq:upper-bound}) for $q$ in a neighborhood $U$
of $0$. \xin{In light of~\eqref{eq:p3}, set $\bar p_2 = \frac{1}{1 + \frac{\gamma \alpha}{4}}>1$ so that $\frac{1}{\bar p_2}+(-\frac{\alpha\gamma}{4})=1$. By Lemma~\ref{lem:q-bound} for some $\eps=\eps_{2\bar p_2}>0$ we have
$\max_{|q|<\eps,\; p_2\in [1,2\bar p_2]} \EE[|\cQ(q)|^{p_2}] < \infty$.  By the GMC moment bound~\eqref{eq:GMC-moment1} in 
Lemma~\ref{lem:GMC-moment}, there exists $p_1 > 1$, $p_2\in [1,2\bar p_2]$ close to $\bar p_2$,  and $p_3 > 1$ close to
$-\frac{4}{\alpha \gamma}$ such that (\ref{eq:upper-bound})  holds for $U=\{ |q|<\eps_{2\bar p_2} \}$.} 
This gives Lemma \ref{lem:hatA} for $\alpha \in (-\frac{4}{\gamma}, 0)$.
\end{proof}

\begin{proof}[Proof of Lemma~\ref{lem:a-props}] Since Lemma~\ref{lem:a-props} (a) follows from Lemma~\ref{lem:hatA}, it remains to prove  Lemma~\ref{lem:a-props} (b). The analyticity in $\alpha$ can be treated by the well-established  argument first done in the bulk case in  \cite{KRV19b} and extended to the boundary case in~\cite{RZ20}.  We start by mapping everything to the half plane which is the setting of ~\cite{RZ20}.
Let $X_{\mathbb{H}}$ be the Gaussian field on $\mathbb{H}$ with covariance given by \eqref{eq:cov-X}.  
Let $\psi(z) = - \ii \frac{z-1}{z+1}$ be the conformal map from $\mathbb{D}$ to the upper-half plane $\mathbb{H}$. 
Let $\phi(x)=\psi(e^{2\pi\ii x}) = \tan(\pi x)$ for $x\in [0, 1/2) \cup (1/2, 1]$. Then $X_{\mathbb{H}}\circ \phi$ equals in law with $Y_{\infty}$. We assume that $Y_\infty=X_{\mathbb{H}}\circ \phi$ by enlarging the sample space of $Y_\infty$. Then the measure  $|(\phi^{-1}(y)')|e^{\frac{\gamma}{2}X_\bbH(y)} dy$ on $\RR$ is the pushforward  of the measure  $e^{\frac{\gamma}{2}Y_\infty(x)} dx$ under $\phi$. 
By performing this change of variable we have 
\begin{align}
\hat A^q_{\gamma,P}(\alpha)&= 
\mathbb{E} \left[ e^{\frac{\alpha}{2} F_{\tau}(0)} 	\cQ(q)
\left( \int_0^1 \sin(\pi x)^{-\alpha\gamma/2}  e^{ \pi \gamma P x } e^{\frac{\gamma}{2} Y_{\infty}(x) }dx   \right)^{- \frac{\alpha}{\gamma}} \right]\nonumber\\
   &=\EE\left[ e^{\frac{\alpha}{2} F_{\tau}(0)} 	\cQ(q) \left(\int_{\mathbb{R}}     |y|^{-\frac{\alpha\gamma}{2}} f_1(y)e^{\frac{\gamma}{2} X_{\mathbb{H}}(y)}   dy \right)^{-\frac{\alpha}{\gamma}} \right]\label{eq:ana_alpha}
\end{align}
where in the last expectation $e^{\frac{\alpha}{2} F_{\tau}(0)} 	\cQ(q)$ and $X_{\mathbb{H}}$ are conditionally independent given $Y_\infty$,  and $f_1: \mathbb{R} \mapsto (0, \infty)$  is such that the measure  $|y|^{-\frac{\alpha\gamma}{2}} f_1(y) dy$ is the pushforward  of $\sin(\pi x)^{-\frac{\alpha\gamma}{2}} e^{\pi \gamma P x} dx$ under $\phi$. We can check that $f_1$ is bounded and continuous.
We now repeat the steps of \cite[Lemma~5.8]{RZ20} to show that this last expectation is analytic in $\alpha$. Let $\mathbb{R}_{r} = \mathbb{R} \setminus (- e^{-r/2} , e^{- r/2})$ where $r>0$. Let $X_{\mathbb{H}, r }(0)$ be the radial part of $X_{\mathbb{H}}(0)$ obtained by taking the mean of $X_{\mathbb{H}}(0)$ over the upper-half circles of radius $e^{-r/2}$ centered at $0$, and let $\overline{X}_{\mathbb{H}, r }(0) = X_{\mathbb{H}, r }(0) - X_{\mathbb{H}, 0 }(0) $. Consider the quantity:
\begin{align*}
\mathcal{E}_{r} = \EE\left[ e^{\frac{\alpha}{2} \overline{X}_{\mathbb{H}, r}(0) - \frac{\alpha^2}{2} \mathbb{E}[\overline{X}_{\mathbb{H}, r}(0)^2]  } e^{\frac{\alpha}{2} F_{\tau}(0)} 	\cQ(q) \left(\int_{\mathbb{R}_r} f_1(y)e^{\frac{\gamma}{2} X_{\mathbb{H}}(y)}   dy \right)^{-\frac{\alpha}{\gamma}} \right].
\end{align*}
By applying Cameron-Martin’s theorem (Theorem \ref{thm:Girsanov}) to $e^{\frac{\alpha}{2} \overline{X}_{\mathbb{H}, r}(0)}$ in the above display, we arrive at:
$$\EE\left[ e^{\frac{\alpha}{2} F_{\tau}(0)} 	\cQ(q) \left(\int_{\mathbb{R}_r} f_1(y)e^{\frac{\gamma}{2} X_{\mathbb{H}}(y)+\frac{\alpha\gamma}{2}\mathbb{E}[ \overline{X}_{\mathbb{H}, r}(0)X_{\mathbb{H}}(y)]}   dy \right)^{-\frac{\alpha}{\gamma}} \right].$$
We know that $\mathbb{E}[ \overline{X}_{\mathbb{H}, r}(0)X_{\mathbb{H}}(y)]\to -\log|y|$ as $r \rightarrow + \infty $. Therefore the large $r$ limit of the quantity in the above display recovers the right hand side of \eqref{eq:ana_alpha}. Now we will choose $\alpha$ complex and write $\alpha = a + \ii b $. We want to show that there exists a neighborhood $V$ of $(-\frac{4}{\gamma}, Q)$ in $\mathbb{C}$ such that for all compact sets contained in $V$, $\mathcal{E}_{r}$ converges uniformly over the compact set. Record that $\overline{X}_{r+ t}(0) - \overline{X}_{r}(0 ) $ is a Brownian motion in $t \geq 0$. By applying Theorem \ref{thm:Girsanov} to the real part of the insertion, namely $e^{\frac{a}{2} \overline{X}_{\mathbb{H}, r}(0)}$, and by upper bounding the imaginary part one gets, for some $c >0$
\begin{align*}
& \left \vert \mathcal{E}_{r+1} - \mathcal{E}_r \right \vert  \nonumber
\le  c \,e^{\frac{r+1}{2}  b^2}\mathbb{E} \left[  \left|  Z_{r+1}^{-\alpha / \gamma}  - Z_r^{- \alpha /\gamma}\right| \right],
\end{align*}
 where $Z_r := e^{\frac{\alpha}{2} F_{\tau}(0)} 	\cQ(q) \int_{\mathbb{R}_{r}} |y|^{-\frac{a \gamma}{2}}  f_1(y) e^{\frac{\gamma}{2} X_{\mathbb{H}}(y)} dy $. Set now $W_r := Z_{r+1} - Z_{r}$. We want to estimate 
\begin{align*}
\E[ |(Z_r + W_r)^{-\alpha/\gamma} - Z_r^{-\alpha/\gamma}|] &\le \E[\mathbf{1}_{|W_r| < \epsilon}|(Z_r + W_r)^{-\alpha/\gamma} - Z_r^{-\alpha/\gamma}|]\\
&+ \E[  \mathbf{1}_{|W_r| \ge \epsilon}|(Z_r + W_r)^{-\alpha/\gamma} - Z_r^{-\alpha/\gamma}|],
\end{align*}
where $\epsilon>0$ will be fixed at the end. Then for some constant $c>0$,
\begin{align*}
&\E[ \mathbf{1}_{|W_r| < \epsilon}|(Z_r + W_r)^{- \alpha /\gamma} - Z_r^{- \alpha /\gamma}|]   \le \epsilon |\frac{\alpha}{\gamma}| \sup_{u\in [0,1]} \E[  |(1-u)Z_r + u W_r|^{-\textnormal{Re}(\alpha)/\gamma-1}] \le c \, \epsilon.
\end{align*}
For the other term, we use the H\"older inequality with $\lambda > 1$ such that $\frac{\lambda}{\lambda-1} \textnormal{Re}( - \frac{\alpha}{\gamma}) <  \frac{2}{\gamma}(Q-a) \wedge \frac{4}{\gamma^2}$, and $0< m < \frac{4}{\gamma^2}$. For some $c,c' >0$
\begin{align*}
&\E[ \mathbf{1}_{|W_r| \ge \epsilon}|(Z_r + W_r)^{- \alpha/\gamma} - Z_r^{-\alpha/\gamma}|] \le c\, \mathbb{P}(|W_r| \ge \epsilon)^{\frac{1}{\lambda}} \le c \epsilon^{-\frac{m}{\lambda}}\E[|W_r|^m]^{\frac{1}{\lambda}}\\ \nonumber
&\le c\epsilon^{-\frac{m}{\lambda}} \E \left[ \left\vert e^{\frac{\alpha}{2} F_{\tau}(0)} 	\cQ(q) \right \vert^m  \left| \int_{\mathbb{R}_{r+1} \setminus \mathbb{R}_r } |y|^{- \frac{a \gamma}{2}}  f_1(y) e^{\frac{\gamma}{2}X_{\mathbb{H}}(y)} dy \right|^m \right]^{\frac{1}{\lambda}}\\ \nonumber
&\le c' \epsilon^{-\frac{m}{\lambda}} \left( e^{-\frac{r}{2}((1+\frac{\gamma^2}{2}-\frac{\gamma a}{2})m - \frac{\gamma^2 m^2}{2})}\right)^{\frac{1}{\lambda}} =: c'\epsilon^{-\frac{m}{\lambda}} e^{-\frac{\theta}{\lambda}r},
\end{align*}
where in the last step $\theta \in \mathbb{R} $ is defined by the last equality and we have used the multifractal scaling property of the GMC, see e.g.~\cite[Section 3.6]{BP-notes} or \cite[Section 4]{RV-GMC}. We can choose a suitable $m$ such that $\theta>0$. Now take $\epsilon = e^{-\eta r}$ with $\eta = \frac{\theta}{\lambda + m}$. Then for some $c,c'>0$:
\begin{align*}
\left \vert \mathcal{E}_{r+1} - \mathcal{E}_r \right \vert \le c\,e^{\frac{r+1}{2}  b^2}(\epsilon + \epsilon^{-\frac{m}{\lambda}} e^{-\frac{\theta}{\lambda}r}) \le c'\, e^{-(\eta - \frac{1}{2} b^2)r}.
\end{align*}
Hence if one chooses the open set $V$ in such a way that $\frac{1}{2} b^2 < \eta$ always holds, all the above inequalities  hold true and thus we have shown that $\mathcal{E}_r$ converges locally uniformly.
This gives the first claim in Lemma~\ref{lem:a-props} (b).

 For the second claim, by  Lemma~\ref{lem:a-props} (a), $\cA^q_{\gamma,P}(\alpha)$ is analytic in $q$. Therefore for a small  enough contour $\mathcal C$ around the origin, we have by \guillaume{the Cauchy integral formula that} $\cA_{\gamma,P,n}(\alpha)= \frac{n!}{2\pi \ii}\oint_{\mathcal C} \cA^q_{\gamma,P}(\alpha)q^{-n-1}dq$. Then the analyticity in $\alpha $ of  $\cA^q_{\gamma,P}(\alpha)$ implies the result for $\cA_{\gamma,P,n}(\alpha)$.
\end{proof}
\xin{Recall $\bar p_2=\frac{1}{1+\frac{\alpha\gamma}{4}}$  and $\eps_{2\bar p_2}$ from Lemma~\ref{lem:q-bound}. 
We introduce:
\begin{define} \label{def:r-alpha}
Let $r_\alpha=\frac12$ for $\alpha\in [0,Q)$ and $r_\alpha=\eps_{2\bar p_2}$ for $\alpha\in (-\frac{\alpha}{\gamma}, 0)$.
\end{define}
By definition  $\max_{|q|<r_\alpha,\;  p\in [1,2\bar p_2]}  \EE |\cQ(q)|^{p}] <\infty$. 
Moreover, from the proof of Lemma~\ref{lem:a-props}, for each $\alpha\in (-\frac{4}{\gamma},Q)$ the convergent radius of $\mathcal A^q_{\gamma,P}$ in $q$ is at least $r_\alpha$. 
We will need $r_\alpha$ in Section~\ref{sec:bpz}.
}

\section{BPZ equation for deformed conformal blocks} \label{sec:bpz}
Following the outline in Section~\ref{sec:method-sum},  we now introduce a deformation of the conformal block (Definition~\ref{def:u-block}) and prove that it satisfies the \notion{BPZ equation} (Theorem \ref{thm:bpz}). Throughout this section we fix $\gamma\in(0,2)$, $Q=\frac{\gamma}{2}+\frac{2}{\gamma}$, 
and $P\in\RR$. 
\xin{We define the deformed conformal block, denoted by 
$\psi^\alpha_\chi(u, \tau)$, in a restricted range of $\tau$. For our application in Section~\ref{sec:ode}, we need that for each $u\in \bbH$, $\psi^\alpha_\chi(u, \tau)$ is defined when  $\Im(\tau)$ is sufficiently large. Our choice of range below may not be the largest possible but it suffices for this purpose and avoids further technicalities in justifying the definition.
To describe our range of $(u,\tau)$ we first record a basic fact on the theta function, which we prove in Appendix~\ref{subsec:u-theta}.}

\begin{lemma}\label{lem:g}
For $\tau\in\bbH$ and $q=e^{\ii\pi\tau}$, let $\xin{\band} := \{u: 0 < \Im(u) < \frac{3}{4} \Im(\tau)\}$.
There exists $q_0>0$ such that if $q\in (0,q_0)$, then $\Im (\xin{\frac{\log \Theta'_\tau(u)}{\log \Theta_\tau(u)}})<0$ for  $u\in \band$.
\end{lemma}
Choose $q_0$ satisfying Lemma~\ref{lem:g}. For $q \in (0, q_0)$, we define 
\begin{equation} \label{eq:fnu-def}
\xin{f(u,q)} \defeq \int_0^1 \Theta_{\tau}(u+x)^{\frac{\gamma\chi}{2}} |\Theta_{\tau}(x)|^{-\frac{\alpha\gamma}{2}}
e^{\pi \gamma P x} e^{\frac{\gamma}{2} Y_\tau(x)} dx \qquad \textrm{for } u\in\band.
\end{equation}
\xin{As explained in Appendix~\ref{subsec:u-theta},  $u\mapsto f(u,q)$ is almost surely analytic and nonzero on $\band$; see Lemma~\ref{lem:f-nu}. Moreover, $\log f(u,q)$ and hence  the fractional powers
of $\xin{f(u,q)}$ is specified in Definition~\ref{def:fractional-power}.} \guillaume{In order to define the $u$-deformed conformal block we need the following proposition.}
\begin{prop}\label{prop:u-block}
Fix $\chi\in\{\frac{\gamma}{2}, \frac{2}{\gamma}\}$ and $\alpha \in (-\frac{4}{\gamma}+\chi, Q)$.
For  $q\in (0,q_0)$ and $u\in\band$, we have $\EE\big[  |\xin{f(u,q)}|^{-\frac{\alpha}{\gamma} + \frac{\chi}{\gamma}}\big]<\infty$. Moreover, $\EE\big[ \xin{f(u,q)}^{-\frac{\alpha}{\gamma} + \frac{\chi}{\gamma}} \big]$
is analytic in $u$ on $\band$.  \xin{Finally, define the domain}
\begin{equation} \label{eq:d-def}
D^\alpha_\chi\defeq\{(u,q) : |q|<r_{\alpha-\chi} \textrm{ and } u\in\band \}
\end{equation}\xin{with $r_\alpha$ from Definition~\ref{def:r-alpha}}.
Then $\EE\big[ \xin{f(u,q)}^{-\frac{\alpha}{\gamma} + \frac{\chi}{\gamma}} \big]$  admits a bi-holomorphic extension in $(u,q)$ on $D^\alpha_\chi$.
\end{prop}
We defer the proof of Proposition~\ref{prop:u-block} to Section \ref{subsec:proof-prop} \xin{and
proceed to introduce the deformed conformal block. We will need the following expression:}
\begin{equation}\label{eq:lchi}
l_\chi = \chi^2/2 - \alpha \chi/2.
\end{equation}
\begin{define}\label{def:u-block}
For $\alpha \in (-\frac{4}{\gamma} + \chi, Q)$ and $\chi\in\{\frac{\gamma}{2}, \frac{2}{\gamma}\}$, define
\xin{\begin{equation} \label{eq:gen-block}
\Sigma^\alpha_{\chi}(u, q)\defeq	C(q) e^{\chi Pu \pi}   \Theta_{\tau}(u)^{- l_\chi} \EE\left[ \xin{f(u,q)}^{-\frac{\alpha}{\gamma} + \frac{\chi}{\gamma}}   \right] \qquad \textrm{for } (u,q)\in D^\alpha_\chi,
\end{equation}
where $C(q)\defeq q^{ \frac{\gamma l_\chi}{12 \chi} - \frac{1}{6} \frac{l_\chi^2}{\chi^2}-  \frac{1}{6 \chi^2}l_\chi(l_\chi+1) }
\Theta'_{\tau}(0)^{- \frac{2 l_\chi^2}{3 \chi^2} + \frac{l_\chi}{3}  + \frac{4 l_{\chi}}{3 \gamma \chi}} 
e^{-\frac12\ii \pi \alpha \gamma(-\frac{\alpha}{\gamma} + \frac{\chi}{\gamma})}$,
and $\EE\left[ \xin{f(u,q)}^{-\frac{\alpha}{\gamma} + \frac{\chi}{\gamma}}   \right]$ is extended to $D^\alpha_\chi$ as in Proposition~\ref{prop:u-block}.}  Define
\begin{equation}\label{eq:u-block}
\psi^\alpha_\chi(u, \tau) \defeq e^{\left(\frac{P^2}{2} + \frac{1}{6 \chi^2}l_\chi(l_\chi+1)  \right)\ii \pi \tau} \Sigma^\alpha_\chi(u,e^{\ii \pi \tau}) \quad \textrm{for }(u,e^{\ii \pi \tau})\in D^\alpha_\chi.
\end{equation}We call $ \psi^\alpha_\chi(u, \tau) $ the \emph{$u$-deformed conformal block}.
\end{define} 
\xin{Although $\psi^\alpha_\chi(u, \tau)$ and $\Sigma^\alpha_\chi(u,q)$ only differ by a simple factor that could have been absorbed in the definition of $C(q)$ from~\eqref{eq:gen-block}, we introduce both because $\psi^\alpha_\chi(u, \tau)$ is convenient for the BPZ equation while $\Sigma^\alpha_\chi(u,q)$ is convenient for  Section~\ref{sec:ode}, which concerns the $q$-series coefficients of  $\Sigma^\alpha_\chi(u,q)$. To get a more compact expression for $\psi^\alpha_\chi(u, \tau)$,} recall from Section~\ref{subsec:GMC-block} that $\Theta_{\tau}(x)^{-\alpha \gamma/2} = e^{-\ii \pi \alpha \gamma/2} |\Theta_{\tau}(x)|^{-\alpha \gamma/2}$ for $x\in(0,1)$ and $q\in(0,1)$.
By Definition \ref{def:u-block}, for $q\in(0,q_0\wedge r_{\alpha -\chi})$ we have the following expression
\begin{equation}\label{eq:q-block}
\psi^\alpha_\chi(u,\tau)= \revise{ \cW(e^{\ii \pi \tau}) }e^{\chi Pu \pi} \EE\left[\Big( \int_0^1 \cT(u, x)e^{\pi \gamma P x} e^{\frac{\gamma}{2} Y_\tau(x)} dx \Big)^{-\frac{\alpha}{\gamma} + \frac{\chi}{\gamma}}\right],
\end{equation}
for $\cW(q) \defeq  q^{\frac{P^2}{2} + \frac{\gamma l_\chi}{12 \chi} - \frac{1}{6} \frac{l_\chi^2}{\chi^2}} \Theta'_{\tau}(0)^{- \frac{2 l_\chi^2}{3 \chi^2} + \frac{l_\chi}{3} + \frac{4 l_{\chi}}{3 \gamma \chi} }$ and $\cT(u, x) := \frac{\Theta_\tau(u + x)^{\frac{\gamma}{2}\chi}}{\Theta_\tau(u)^{\frac{\gamma \chi}{2}} \Theta_\tau(x)^{\frac{\alpha\gamma}{2}}}$.

\xin{
We now state the BPZ equation for $\psi^\alpha_\chi(u,\tau)$ which we prove in Section~\ref{sec:bpz-proof}.}%
\begin{theorem} \label{thm:bpz}
For $\chi\in\{\frac{\gamma}{2}, \frac{2}{\gamma}\}$ and $\alpha \in (-\frac{4}{\gamma} + \chi, Q)$,  we have
\begin{equation} \label{eq:bpz1}
\Big(\partial_{uu} - l_\chi(l_\chi + 1) \wp(u) + 2\ii \pi \chi^2 \partial_\tau\Big)
\psi^\alpha_{\chi}(u, \tau)= 0 \qquad \textrm{for }(u,e^{\ii \pi \tau})\in D^\alpha_\chi,
\end{equation}
where $\wp$ is Weierstrass elliptic   function from Appendix~\ref{subsec:special}.
\end{theorem}
  
\subsection{Proof of Proposition~\ref{prop:u-block}}\label{subsec:proof-prop}
The \xin{first assertion follows from Lemma~\ref{lem:GMC-moment}. In fact,~\eqref{eq:GMC-moment2} implies that $\max_{u\in K}\EE\big[  |\xin{f(u,q)}|^{-\frac{\alpha}{\gamma} + \frac{\chi}{\gamma}}\big]<\infty$
for each compact $K\subset\band$.} For the second assertion, since the function $\xin{u\mapsto f(u,q)}^{-\frac{\alpha}{\gamma} + \frac{\chi}{\gamma}}$
is  almost surely analytic on $\band$, by Lemma~\ref{lem:analytic-basic} its expectation
$\EE\big[ \xin{f(u,q)}^{-\frac{\alpha}{\gamma} + \frac{\chi}{\gamma}} \big]$ is also analytic in $u$ on $\band$. To complete the proof of Proposition~\ref{prop:u-block}, it remains to establish the analytic extension of $\EE\big[ \xin{f(u,q)}^{-\frac{\alpha}{\gamma} + \frac{\chi}{\gamma}} \big]$ to $D^\alpha_\chi$.

Similarly to the proof of Lemma~\ref{lem:a-props}(a),
we use {Cameron-Martin's} theorem to manipulate the expression into a form
  with no $q$-dependence in the GMC moment. Recall the i.i.d. standard real Gaussians $a_n, b_n, a_{n, m}, b_{n, m}$ from~\eqref{eq:y-tau-def}---\eqref{eq:f-tau-def} in
  Section \ref{sec:gmc-def}.  For $q\in \D$ and
  $u\in \band$,  set $u' := u - \frac{\tau}{2}$ and
\begin{align}\label{eq:Xbound}
\cX(u,q) &\defeq -\chi
\sum_{n,m =1}^{\infty} \frac{1}{ \sqrt{2n}} \left( {(a_n + \ii b_n )}q^{(2m-2)n} e^{2 \pi \ii u n} + {(a_n - \ii b_n )}  q^{2nm} e^{-2 \pi \ii u n}   \right);\\
\cY(u,q) &\defeq 2 \chi \sum_{m,n,k \geq 1} a_{n,m}   \frac{q^{(2k-1 +m)n }}{\sqrt{n}}  \cos(2 \pi u' n)\\
&\phantom{===========} \nonumber
- 2 \chi \sum_{m,n,k \geq 1} b_{n,m} \frac{q^{(2k-1+m)n}}{\sqrt{n}}  \sin(2 \pi u' n).
\end{align}
Since $|q|^{3/2}<|e^{2\pi \ii u}|< 1< |e^{-2\pi \ii u}|<|q|^{-3/2}$ when $u\in \band$, the
series for both $\cX(u, q)$ and $\cY(u, q)$ converge almost surely in $q\in\D$,
and $e^{\cX(u,q)}$ and $e^{\cY(u, q)}$ have finite moments of all orders. \guillaume{Recall also $\mathcal{Q}(q)$ from \eqref{eq:Qq}.}
Now define 
\begin{multline}\label{eq:psi_ana_girsa}
\tilde{\psi}_\chi^\alpha(u,q) := \EE \left[  e^{\frac{\alpha}{2} F_{\tau}(0)} \cQ(q) e^{\cY(u,q)} e^{ \cX(u,q) -\frac12 \E[\cX(u,q)^2]}\right.\\
\Big( \int_0^1 (2\sin(\pi x))^{-\frac{\alpha \gamma}{2}} 
\revise{e^{- \ii \pi \frac{\chi \gamma}{2} x}} e^{  \pi \gamma P x } e^{\frac{\gamma}{2} Y_{\infty}(x) } dx  \Big)^{-\frac{\alpha}{\gamma} + \frac{\chi}{\gamma}}\Big].
\end{multline}
\xin{Proposition~\ref{prop:u-block} follows from the two lemmas below.}
\begin{lemma}\label{lem:psiana}
\xin{The function  $(u, q)\mapsto \tpsi^\alpha_\chi(u, q)$ admits
an analytic extension to $D^\alpha_\chi$.}
\end{lemma}
\begin{lemma}\label{lem:u12}
\xin{For $q\in(0,q_0\wedge r_{\alpha-\chi})$ and $\Im u=\frac12\Im\tau$, we have
\begin{align}\label{eq:bi-holo}
\EE\left[ f(u,q)^{-\frac{\alpha}{\gamma} + \frac{\chi}{\gamma}} \right]=
C_1(q)
\left( - \ii e^{-\ii \pi u} q^{1/6} \eta(q) \right)^{\frac{\chi}{2}(\chi -\alpha)}\tilde{\psi}_\chi^\alpha(u,q)
\end{align}
for $C_1(q)\defeq [q^{1/6} \eta(q)]^{\frac{ \alpha(\alpha - \chi)}{2}} e^{(\frac{\alpha \gamma}{8} - \frac{\gamma \chi}{8}  - \frac{\alpha^2}{8}  )\EE[F_{\tau}(0)^2] }$.}
\end{lemma}
\begin{proof}[\xin{Proof of Proposition~\ref{prop:u-block} given Lemmas~\ref{lem:psiana} and~\ref{lem:u12}}]
\xin{Since we have proved the first two assertions, we know that} $\EE\left[ f(u,q)^{-\frac{\alpha}{\gamma} + \frac{\chi}{\gamma}}\right]$ is analytic in $u\in\band$.
Therefore \eqref{eq:bi-holo} holds
not only for $\Im(u) = \frac{1}{2} \Im(\tau)$ but for all $u\in \band$.
\xin{By Lemma~\ref{lem:psiana},} the right hand side of (\ref{eq:bi-holo}) provides the desired analytic continuation in $(u,q)$ for $\EE\left[ f(u,q)^{-\frac{\alpha}{\gamma} + \frac{\chi}{\gamma}}\right]$ as in the last assertion.
\end{proof}

\begin{proof}[Proof of Lemma~\ref{lem:psiana}]
\xin{Our proof is parallel to that of Lemma~\ref{lem:hatA} so we will be brief here.} 
For $p_1, p_2, p_3 \in (1, \infty)$
with $\frac{1}{p_1} + \frac{1}{p_2} + \frac{1}{p_3} = 1$, H\"older's inequality yields that
\begin{align}\label{eq:holder_prop32}
\EE &\left[ \Big|e^{\frac{\alpha}{2} F_{\tau}(0)}  e^{\mathcal{Y}(u,q)} e^{ \cX(u,q) -\frac12 \E[\cX(u,q)^2]}\cQ(q) \Big|\right. \\ \nonumber
&\phantom{====}\Big \vert \int_0^1 (2 \sin(\pi x))^{-\alpha\gamma/2} \revise{e^{- \ii \pi \frac{\chi \gamma}{2} x}}  e^{ \pi \gamma P x } e^{\frac{\gamma}{2} Y_{\infty}(x) }dx\Big \vert ^{- \frac{\alpha}{\gamma} + \frac{\chi}{\gamma}} \Big]\\ \nonumber
&\leq \EE\left[ \left| e^{\frac{\alpha}{2} F_{\tau}(0)}  e^{\mathcal{Y}(u,q)} e^{ \cX(u,q) -\frac12 \E[\cX(u,q)^2]}\right|^{p_1} \right]^{\frac{1}{p_1}} \cdot  \EE\left[|\cQ(q)|^{p_2}\right]^{\frac{1}{p_2}} \\ \nonumber
&\phantom{====}\EE\Big[\Big \vert \int_0^1(2 \sin(\pi x))^{-\alpha\gamma/2} \revise{e^{- \ii \pi \frac{\chi \gamma}{2} x}} e^{\pi \gamma P x} e^{\frac{\gamma}{2} Y_{\infty}(x) }dx\Big \vert^{ \frac{(\chi -\alpha) p_3}{\gamma}} \Big]^{\frac{1}{p_3}}
\end{align}
We now choose ranges for $p_1, p_2, p_3$ for which each of the
three terms on the right side of (\ref{eq:holder_prop32}) is finite
\xin{for $\alpha \in (-\frac{4}{\gamma} + \chi, Q)$ and $|q| < r_{\alpha - \chi}$. This will imply that $\tilde{\psi}_\chi^\alpha(u,q)$ admits a bi-holomorphic extension to $D^\alpha_\chi$.}

For the first term, because a Gaussian
random variable has finite exponential moments, for any $p_1 >1$ we have
\begin{align*}
\EE\left[ \left| e^{\frac{\alpha}{2} F_{\tau}(0)}  e^{\mathcal{Y}(u,q)} e^{ \cX(u,q) -\frac12 \E[\cX(u,q)^2]}\right|^{p_1} \right]^{\frac{1}{p_1}} < \infty.
\end{align*}
As in the proof of Lemma~\ref{lem:hatA}, to analyze the other two terms we divide into cases based on the sign
of $\alpha - \chi$.  For $\alpha - \chi$ positive, the exponent $\frac{(\chi -\alpha) p_3}{\gamma}$ in the third term
is negative, meaning that the third term is finite for arbitrarily large $p_3$ \revise{thanks to Lemma \ref{lem:GMC-moment}}.  Also choosing $p_1$ arbitrarily
large, it remains to check that the second term is finite for $p_2$ close to $1$, which follows by Lemma \ref{lem:q-bound}
applied with $\alpha - \chi$ in place of $\alpha$.

For $\alpha - \chi$ negative, the third term is finite if $1 < p_3 < -\frac{4}{(\alpha - \chi)\gamma}$. \xin{Recall Definition~\ref{def:r-alpha} of $r_\alpha$.} Choosing $p_1$
arbitrarily large and $p_3$ close to $-\frac{4}{(\alpha - \chi) \gamma}$, it suffices to check that the second term
is finite for $p_2$ close to $\frac{1}{1 + \frac{\gamma(\alpha - \chi)}{4}}$, which again follows by Lemma \ref{lem:q-bound}
applied with $\alpha - \chi$ in place of $\alpha$.
 \end{proof}

\begin{proof}[Proof of Lemma~\ref{lem:u12}]
Assume $q \in (0, q_0 \wedge r_{\alpha - \chi})$ and $\Im u = \frac{1}{2} \Im \tau$.
\guillaume{Similarly as in the proof of Lemma \ref{lem:reduce}, by
Theorem \ref{thm:Girsanov}} we have 
\begin{align*}
\EE&\Big[ f(u,q)^{-\frac{\alpha}{\gamma} + \frac{\chi}{\gamma}}\Big]
= \EE\Big[\Big( \int_0^1  |\Theta_{\tau}(x)|^{-\frac{\alpha\gamma}{2}} \Theta_\tau(u + x)^{\frac{\gamma}{2}\chi} e^{\pi \gamma P x} e^{\frac{\gamma}{2} Y_\tau(x)}dx \Big)^{-\frac{\alpha}{\gamma} + \frac{\chi}{\gamma}}\Big] \\
& =  C_1(q)
\EE \Big[ e^{\frac{\alpha}{2} F_{\tau}(0) }  \Big( \int_0^1 e^{\frac{\gamma}{2} F_{\tau}(x) }  
(2\sin (\pi x))^{-\frac{\alpha \gamma}{2}} \Theta_{\tau}( x + u)^{\frac{\chi \gamma}{2}} e^{  \pi \gamma P x } e^{\frac{\gamma}{2} Y_{\infty}(x) } dx\Big)^{-\frac{\alpha}{\gamma} + \frac{\chi}{\gamma}} \Big].
\end{align*}	
Then by \eqref{eq:Ftau} and Theorem \ref{thm:Girsanov}, we get the following analog of~\eqref{eq:hatA-expression}
\begin{multline}\label{eq:Efnu}
\EE\left[ f(u,q)^{-\frac{\alpha}{\gamma} + \frac{\chi}{\gamma}}   \right]\\=C_1(q)
\EE \Big[ e^{\frac{\alpha}{2} F_{\tau}(0)} 
\cQ(q)
\Big( \int_0^1 (2\sin(\pi x))^{-\frac{\alpha \gamma}{2}} \Theta_{\tau}( x + u)^{\frac{\chi \gamma}{2}} e^{  \pi \gamma P x } e^{\frac{\gamma}{2} Y_{\infty}(x) }  dx\Big)^{-\frac{\alpha}{\gamma} + \frac{\chi}{\gamma}} \Big]
\end{multline}
for $\cQ(q)$ from~\eqref{eq:Qq}. We now claim that 
\begin{equation} \label{eq:Theta-cX}
\Theta_{\tau}(u+x) = - \ii \revise{ e^{-\ii \pi (u +x)} } q^{\frac{1}{6}} \eta(q) e^{\frac{1}{\chi}\E[Y_\infty(x)\cX(u,q)]}.
\end{equation} 
To see \eqref{eq:Theta-cX}, recalling that $u' =  u-\frac{\tau}{2}$, by~\eqref{eq:half-up} we have 
\begin{equation}
\Theta_{\tau}(u+x) = - \ii \revise{e^{-\ii \pi (u+x)}} q^{\frac{1}{6}} \eta(q)
\prod_{m=1}^{\infty}  (1 - q^{2m-1} e^{2 \pi \ii (u'+x)} )(1 - q^{2m-1} e^{-2 \pi \ii (u'+x)} ).
\end{equation}
Using \revise{$1-z=\exp\{ - \sum_{n=1}^{\infty} \frac{z^n}{n} \}$ } for $|z|<1$ and recalling \eqref{eq:sin-cos}, we have  
\begin{align*}
&\prod_{m=1}^{\infty}  (1 - q^{2m-1} e^{2 \pi \ii (u'+x)} )(1 - q^{2m-1} e^{-2 \pi \ii (u'+x)} )\\
&=  \exp\Big\{ -\sqrt{2} \sum_{n,m =1}^{\infty} \frac{q^{(2m-1)n}}{\sqrt{n}} { \left( \cos(2 \pi u'n)\EE[a_n Y_{\infty}(x )] - \sin(2 \pi u'n)\EE[b_n Y_{\infty}(x )] \right) } \Big\}. 
\end{align*}
Now, \eqref{eq:Theta-cX} follows from the observation that
\begin{align}\label{eq:u-u'}
\cX(u,q) = -\chi \sqrt{2}  \sum_{n,m =1}^{\infty} \frac{q^{(2m-1)n}}{\sqrt{n}} { \left( \cos(2 \pi u'n) a_n - \sin(2 \pi u'n) b_n \right).}
\end{align}
Since \xin{$q \in (0, q_0)$ and $\Im u'=0$,} \eqref{eq:u-u'} implies that   
\begin{equation}\label{eq:u'real}
\cX(u,q)\in\RR\quad \textrm{and}\quad \E [\cX(u,q)^2] =2\chi^2  \sum_{n,m =1}^{\infty} \frac{q^{2(2m-1)n}}{n}.
\end{equation}
By~\eqref{eq:Efnu} and \eqref{eq:Theta-cX}, we have
\begin{align} \nonumber
&\left( \int_0^1 (2\sin(\pi x))^{-\frac{\alpha \gamma}{2}} \Theta_{\tau}( x + u)^{\frac{\chi \gamma}{2}} e^{  \pi \gamma P x } e^{\frac{\gamma}{2} Y_{\infty}(x) } dx  \right)^{-\frac{\alpha}{\gamma} + \frac{\chi}{\gamma}} \\ \nonumber
&= [- \ii e^{-\ii \pi u} q^{1/6} \eta(q)]^{\frac{\chi}{2}(\chi -\alpha)} \times \\ \nonumber
&\quad \left( \int_0^1 (2\sin(\pi x))^{-\frac{\alpha \gamma}{2}} e^{\frac{\gamma}{2}\E[Y_\infty(x)\cX(u,q)]}
\revise{e^{- \ii \pi \frac{\chi \gamma}{2} x}} e^{  \pi \gamma P x } e^{\frac{\gamma}{2} Y_{\infty}(x) } dx  \right)^{-\frac{\alpha}{\gamma} + \frac{\chi}{\gamma}}.
\end{align}
Since \xin{$\cX(u,q)$ is a real Gaussian}, applying Cameron-Martin's  theorem again with
respect to the randomness of $(a_n)_{n \geq 1}$, $(b_n)_{n \geq 1}$ while freezing
  $(a_{n,m})_{n,m \geq 1}$, $(b_{n,m})_{n,m \geq 1}$, we get (\ref{eq:bi-holo}) \xin{by comparing with~\eqref{eq:psi_ana_girsa}.} 
\end{proof}

\guillaume{To finish this section }\xin{we record an analytic extension result including $\alpha$, which is needed in the next subsection and in Sections \ref{sec:ode} and \ref{sec:opes}.}
\begin{lemma} \label{lem:psi-props}
	Given $\chi\in\{\frac{\gamma}2, \frac{2}{\gamma} \}$, there exists an open set  in $\CC^3$ containing 
	$\{ (\alpha, u,q ): \alpha\in  (- \frac{4}{\gamma} + \chi, Q ), \xin{\Im  u>0}, q=0\}$ on which
	$ (\alpha, u,q)  \mapsto \Sigma^\alpha_{\chi}(u, q)$ has an analytic continuation.
\end{lemma}
\begin{proof}
Thanks to Proposition \ref{prop:u-block},
we have the desired analyticity with respect to $u$ and $q$. As explained
in the proof of Lemma \ref{lem:a-props} (b), analyticity in $\alpha$ follows
from a straightforward adaptation of the arguments of \cite[Lemma~5.6]{RZ20}. 
\end{proof}


\subsection{Proof of Theorem~\ref{thm:bpz}}  \label{sec:bpz-proof}
This proof is conceptually straightforward: we compute\xin{
$\Big(\partial_{uu} - l_\chi(l_\chi + 1) \wp(u) + 2\ii \pi \chi^2 \partial_\tau\Big)
\psi^\alpha_{\chi}(u, \tau)$ and find that it equals 0. In practice, we need to regularize $\psi^\alpha_{\chi}(u, \tau)$ so that there is no analytic issue when taking derivatives, and the calculation
is quite involved, requiring integration by parts and several identities on the theta function.

When checking~\eqref{eq:bpz1} we fix $\chi\in \{\frac{\gamma}{2},\frac{2}{\gamma}\}$ and assume throughout that
\begin{equation}~\label{eq:bpz-range}
    q\in(0,q_0\wedge r_{\alpha-\chi}),\quad u\in\band\quad\textrm{and}\quad \alpha \in (-\gamma + \chi, \frac{2}{\gamma}).
\end{equation}
Recall from~\eqref{eq:u-block} that $ \psi^\alpha_\chi(u,q)$ and $\Sigma^\alpha_\chi(u,q)$ are related by a simple factor. By the analyticity of $\Sigma^\alpha_\chi(u,q)$ in $(u,q)$ from Proposition~\ref{prop:u-block} and in $(\alpha,u,q)$ from Lemma~\ref{lem:psi-props}, once~\eqref{eq:bpz1} is verified for the range~\eqref{eq:bpz-range}, we know that it holds for the entire range claimed in Theorem~\ref{thm:bpz}.
The advantage of working with the range~\eqref{eq:bpz-range} is that we can use the expression~\eqref{eq:q-block} for $ \psi^\alpha_\chi(u,q)$, and the restriction on $\alpha$ will avoid certain singularity issues.
}

\xin{We introduce a regularization of $\psi^\alpha_{\chi}(u,\tau)$.
Recall $Y_\infty, F_\tau,Y_\tau$ from~\eqref{eq:y-tau-def}---~\eqref{eq:f-tau-def}. We first realize $Y_\infty$ as the restriction of a free boundary GFF on a half-infinite cylinder.} 
Recall $X_\bbH$ from Appendix~\ref{sec:gmc-app}. 
For \revise{$x\in  [0,1] \times \RR_{\ge 0}$, meaning $x \in \mathbb{C}$ with $\Re(x) \in [0,1]$ and $\Im(x) \in \RR_{\ge 0}$,} let $\phi(x)=- \ii \frac{e^{2\pi\ii x}-1}{e^{2\pi \ii x}+1}\in \guillaume{\bbH\cup \RR \cup \{ \infty \}}$.
Then $\phi$ conformally \revise{maps}  the half cylinder $\cC_+$ obtained by gluing the two vertical boundaries of  $[0,1]\times \RR_{\ge 0}$ to $ \guillaume{\bbH\cup \RR \cup \{ \infty \}} $ (\guillaume{with $\phi(1/2) =\infty$}). 
Let $Y_\infty(x)= X_\bbH(\phi(x))$ for $x\in \cC_+$. Then $Y_\infty(x)$ is a  free boundary GFF on $\cC_+$.

\guillaume{
Now fix $\eps>0$ small. Define $ \eta_{\eps}(\cdot) = \eps^{-2} \eta(\eps^{-1} \cdot )$, where $\eta $ is a radial smooth function supported on the unit disk satisfying $\int_{\mathbb{H}} \eta(x) d^2x =1 $. For $x\in (0,1)$, let $Y^\eps_\infty(x)$  be defined by the convolution:
\begin{align*}
Y_{\infty}^{\eps}(x) := (Y_{\infty} \ast \eta_{\eps})(x) = \int_{\mathcal{C}_+} Y_{\infty}(x -x') \eta_{\eps}(x')d^2x'.
\end{align*}
Let $Y^\eps_\tau\defeq Y^\eps_\infty+F_\tau$.}
Recall $\cT(u, x)$ \xin{and $\cW(q)$ from the expression~\eqref{eq:q-block} of  $\psi^\alpha_{\chi}(u,\tau)$ and define $s := -\frac{\alpha}{\gamma} + \frac{\chi}{\gamma}$.} Let
\begin{equation}\label{eq:psi_reg}
\psi^\alpha_{\chi, \eps}(u,\tau)=  \cW(e^{\ii \pi \tau}) e^{\chi Pu \pi} \EE\left[\Big( \int_0^1 \cT(u, x)e^{\pi \gamma P x} e^{\frac{\gamma}{2}Y^\eps_\tau(x)-\frac{\gamma^2}{8}\EE[Y^\eps_\tau(x)^2]} dx \Big)^{\xin{s}}\right].
\end{equation}
\xin{We first compute $\partial_{uu}\psi^{\alpha}_{\chi, \eps}$. For $\eps>0$, let 
\begin{align}
    V_\eps(u,\tau)&\defeq \int_0^1 e^{\frac{\gamma}{2} Y^\eps_\tau(x)-\frac{\gamma^2}{8}\EE[Y^\eps_\tau(x)^2]} \cT(u, x) e^{\pi \gamma P x} dx;\label{eq:V1-eps-def}\\
    \cV_{1,\eps}(u, y)&\defeq 
\EE\Big[  V_{\eps}(u,\tau)^{s - 1} e^{\frac{\gamma}{2} Y^\eps_\tau(y)-\frac{\gamma^2}{8}\EE[Y^\eps_\tau(y)^2]}  \Big].\label{eq:V1}
\end{align}
Here we omit the dependence of $\cV_{1,\eps}$ on $\tau$ for notation simplicity.
\guillaume{Computing the $u$ derivatives of} $\psi^{\alpha}_{\chi, \eps}$ in~\eqref{eq:psi_reg} we have} 
\[
\partial_u \psi^{\alpha}_{\chi, \eps}(u,q)= \chi P \pi \psi^{\alpha}_{\chi, \epsilon}(u,q) + s \cW(q) e^{\pi \chi Pu}\!\! \int_0^1 \partial_u\cT(u, y) e^{\pi \gamma P y} \cV_{1,\eps}(u,y) dy;
\]
\xin{\begin{equation}\label{eq:u-der}
\partial_{uu}\psi^{\alpha}_{\chi, \eps}(u,q)  = (\chi P \pi)^2 \psi^{\alpha}_{\chi, \eps}(u,q) 
+ \Xi^{\mathrm{uu}}_{1,\eps}  + \Xi^{\mathrm{uu}}_{2,\eps},
\end{equation}
where
\begin{align*}  
\cV_{2,\eps}(u, y, z) &\defeq \EE\Big[
V_{\eps}(u,\tau)^{s - 2}   e^{\frac{\gamma}{2} Y^\eps_\tau(y)-\frac{\gamma^2}{8}\EE[Y^\eps_\tau(y)^2]}  e^{\frac{\gamma}{2} Y^\eps_\tau(z)-\frac{\gamma^2}{8}\EE[Y^\eps_\tau(z)^2]} \Big];\\
\Xi_{1,\eps}^{\mathrm{uu}}&\defeq  2\chi P \pi s  \cW(q) e^{\pi \chi P u}\!\! \int_0^1 \partial_u\cT(u, y) e^{\pi \gamma P y} \cV_{1, \eps}(u,y) dy \\
&\quad \quad \quad \quad \quad \quad \quad + s \cW(q) e^{\pi \chi Pu}\!\! \int_0^1 \partial_{uu}\cT(u, y) e^{ \pi \gamma P y} \cV_{1,\eps}(u,y) dy;\\
\Xi^{\mathrm{uu}}_{2,\eps}&\defeq s(s - 1) \cW(q) e^{\pi \chi P u}\!\! \int_0^1 \int_0^1 \partial_{u}\cT(u, y) \partial_u\cT(u, z) e^{\pi \gamma P (y+z)} \cV_{2, \eps}(u,y,z) dy dz.
\end{align*}
The next lemma summarizes some basic properties of $\cV_{1,\eps}(u, y)$ and  $\cV_{2, \eps}(u, y, z)$.
\begin{lemma}\label{lem:V1}
Suppose $(q,u,\alpha)$ are in the range~\eqref{eq:bpz-range}. Let
\begin{equation}\label{eq:V-Girsanov}
\cV_1(u, y) \defeq \EE\Big[\Big(\int_0^1 e^{\frac{\gamma^2}{4}\EE[Y_\tau(x)Y_\tau(y)]} \cT(u, x) 
e^{\pi \gamma P x} e^{\frac{\gamma}{2} Y_\tau(x)} dx \Big)^{s - 1}\Big].
\end{equation}
Then for each fixed $u\in\band$, the function $\cV_1(u,\cdot)$  is bounded continuous on $[0,1]$.
Moreover,  $\lim_{\eps\to 0}\cV_{1,\eps}(u, y) =\cV_1(u, y)$ and 
\begin{equation}\label{eq:V1-V2}
\int_0^1    \cT(u, z) e^{ \pi \gamma P z} \cV_{2, \eps}(u, y, z) dz
=  \cV_{1,\eps}(u, y).
\end{equation}
Finally, \(\psi^\alpha_\chi(u, q)=  \cW(q) e^{\pi \chi P u} \int_0^1  \cT(u, z) e^{\pi \gamma P z} \cV_1(u, z) dz.\)
\end{lemma}
\begin{proof}
Since $\alpha\in (\chi-\gamma,\frac2\gamma)$ by~\eqref{eq:bpz-range}, we have  $s-1<0$. By the GMC moment bound~\eqref{eq:GMC-moment3} in Lemma~\ref{lem:GMC-moment} we get the assertion on the boundedness and continuity of $\cV_1$. Cameron-Martin's Theorem (Theorem \ref{thm:Girsanov}) yields \begin{equation}\label{eq:V1-Girsanov}
\cV_{1,\eps}(u, y) =\EE\Big[\Big(\int_0^1 e^{\frac{\gamma^2}{4}\EE[Y^\eps_\tau(x)Y_\tau^\eps(y)]} \cT(u, x) 
e^{\pi \gamma P x} e^{\frac{\gamma}{2} Y^\eps_\tau(x) -\frac{\gamma^2}{8}\EE [Y^\eps_\tau(x)^2] } dx \Big)^{s - 1}\Big]
\end{equation}
hence $\lim_{\eps\to 0}\cV_{1,\eps}(u, y) =\cV_1(u, y)$. The equation~\eqref{eq:V1-V2} follows from Fubini theorem and  the definition of $\cV_{1, \eps}$ and $\cV_{2, \eps}$.
The last assertion follows from the expression~\eqref{eq:q-block} of $\psi^\alpha_\chi(u, q)$. 
\end{proof}
}

We now compute $\partial_\tau \psi^\alpha_{\chi, \eps}(u, q)$.
\begin{lemma}\label{lem:partialtau}
\xin{Suppose $(q,u,\alpha)$ are in the range~\eqref{eq:bpz-range}.} We have
\begin{align*}
\partial_\tau \psi^\alpha_{\chi, \eps}(u, q) = &\ii \pi \Big(\frac{P^2}{2} + \frac{\gamma l_\chi}{12\chi} - \frac{1}{6} \frac{l_\chi^2}{\chi^2}\Big) \psi^{\alpha}_{\chi, \eps}(u,q) 
+ \Big(-\frac{2 l_\chi^2}{3 \chi^2} + \frac{l_\chi}{3} + \frac{2}{3} s\Big) \frac{\partial_\tau \Theta_\tau'(0)}{\Theta_\tau'(0)} \psi^{\alpha}_{\chi, \eps}(u,q),\\
&+\Xi^{\mathrm{tau}}_{1,\eps}+\Xi^{\mathrm{tau}}_{2,\eps}
\end{align*} where 
\xin{
\begin{align*}
    \Xi^{\mathrm{tau}}_{1,\eps}&=s \cW(q) e^{\pi \chi P u} \int_0^1 \partial_{\tau} \cT(u, y) e^{\pi \gamma P y} \cV_{1,\eps}(u, y) dy,\\
    \Xi^{\mathrm{tau}}_{2,\eps}&=\frac{\gamma^2 s(s-1)}{4}  \cW(q) e^{\pi \chi P u} \int_0^1 \int_0^1 \left(\frac{\ii\pi}{6} - \frac{\partial_\tau \Theta_\tau(y - z)}{\Theta_\tau(y - z)} + \frac{1}{3}\frac{\partial_\tau \Theta'_\tau(0)}{\Theta'_\tau(0)}  \right)\\
&\phantom{=============} \times \cT(u, y) \cT(u, z) e^{\pi \gamma P y + \pi \gamma P z} \cV_{2, \eps}(u, y, z) dy dz.
\end{align*}}
\end{lemma}

\begin{proof} Recall $\xin{V_{\eps}(u,\tau)} =\int_0^1 e^{\frac{\gamma}{2} Y^\eps_\tau(x)-\frac{\gamma^2}{8}\EE[Y^\eps_\tau(x)^2]} \cT(u, x) e^{\pi \gamma P x} dx.$
\revise{
Taking the $\tau$-derivative of~\eqref{eq:q-block}, we obtain 
\begin{equation}\label{eq:tau-der1}
\partial_\tau \psi^\alpha_{\chi,\eps}(u, q) =\partial_\tau(\log \cW(q)) \psi^{\alpha}_{\chi, \eps}(u,q)+ s  \cW(q) e^{\pi \chi P u}\EE\left[\xin{ V_{\eps}(u,\tau)}^{s-1} \partial_\tau \xin{ V_{\eps}(u,\tau)}\right].
\end{equation}
Note that 
\begin{equation} \label{eq:der-W}
\partial_\tau(\log \cW(q)) = \ii \pi \Big(\frac{P^2}{2} + \frac{\gamma l_\chi}{12\chi} - \frac{1}{6} \frac{l_\chi^2}{\chi^2}\Big) + \Big(-\frac{2 l_\chi^2}{3 \chi^2} + \frac{l_\chi}{3} + \frac{2}{3} s\Big) \frac{\partial_\tau \Theta_\tau'(0)}{\Theta_\tau'(0)},
\end{equation}
and
\begin{multline} \label{eq:der-V}
\EE\left[\xin{ V_{\eps}(u,\tau)}^{s-1} \partial_\tau \xin{ V_{\eps}(u,\tau)}\right]=
\int_0^1 \partial_{\tau} \cT(u, y) e^{\pi \gamma P y} \cV_{1,\eps}(u, y) dy\\ 
+ \int_0^1 \cT(u, y) e^{\pi \gamma P y} \EE\left[\xin{ V_{\eps}(u,\tau)}^{s-1} \partial_{\tau}[ e^{\frac{\gamma}{2} Y^\eps_\tau(y)-\frac{\gamma^2}{8}\EE[Y^\eps_\tau(y)^2]}  ]dy\right], 
\end{multline}
where 
\begin{align}\label{eq1:sec3}
\partial_{\tau}[&e^{\frac{\gamma}{2} Y^\eps_\tau(y)-\frac{\gamma^2}{8}\EE[Y^\eps_\tau(y)^2]}  ]dy \\ \nonumber
&=\partial_{\tau}\left[e^{- \frac{\gamma^2}{8}\EE[F_\tau(0)^2]} e^{ \frac{\gamma}{2} F_\tau(y) }\right] e^{\frac{\gamma}{2} Y^\eps_{\infty}(y)-\frac{\gamma^2}{8}\EE[Y^\eps_{\infty}(y)^2]} dy\\ \nonumber
&= \partial_\tau\Big(\frac{\gamma}{2} F_\tau(y) - \frac{\gamma^2}{8}\EE[F_\tau(0)^2]\Big) \times e^{\frac{\gamma}{2} Y^\eps_\tau(y)-\frac{\gamma^2}{8}\EE[Y^\eps_\tau(y)^2]} dy.
\end{align}
We claim that 
\begin{align} \label{eq:der-F}
\EE&\Big[\xin{ V_{\eps}(u,\tau)}^{s-1} \partial_{\tau}[e^{\frac{\gamma}{2} Y_{\tau}(y)}]dy\Big] \\ \nonumber
&= \EE \Big[ \frac{\gamma^2}{4}(s-1) e^{\frac{\gamma}{2} Y^\eps_\tau(y)-\frac{\gamma^2}{8}\EE[Y^\eps_\tau(y)^2]}\\ \nonumber
&\phantom{==}\int_0^1 \EE[F_{\tau}(z)\partial_{\tau}F_{\tau}(y)]\cT(u, z) e^{\pi \gamma P z}   \xin{ V_{\eps}(u,\tau)}^{s - 2} e^{\frac{\gamma}{2} Y^\eps_\tau(z)-\frac{\gamma^2}{8}\EE[Y^\eps_\tau(z)^2]} dz dy \Big].
\end{align}
\xin{Since $q = e^{\ii \pi \tau}$ and we restrict to real $q$, we have $\partial_\tau=\ii \pi q \partial_q$ where $\partial_q$ a real derivative.} By (\ref{eq:y-tau-def}) and \eqref{eq:Theta'0},
we find that
\begin{multline*}
\EE[\partial_\tau F_\tau(y)F_\tau(z)] = 4\pi \ii \sum_{m, n = 1}^\infty m q^{2nm} \cos(2\pi n(y - z))
= \frac{1}{2} \partial_\tau \EE[F_\tau(y)F_\tau(z)]\\
= - \partial_\tau \log\left| q^{-\frac{1}{6}}\frac{\Theta_\tau(y - z)}{\eta(q)}\right|
=\frac{\ii\pi}{6} - \frac{\partial_\tau \Theta_\tau(y - z)}{\Theta_\tau(y - z)} + \frac{1}{3}\frac{\partial_\tau \Theta'_\tau(0)}{\Theta'_\tau(0)}.
\end{multline*}

Combining with~\eqref{eq:tau-der1}---\eqref{eq:der-F}, we obtain Lemma~\ref{lem:partialtau}.

It remains to prove~\eqref{eq:der-F}. By Cameron-Martin's Theorem (Theorem\ref{thm:Girsanov}), 
\begin{align*}
& \EE\Big[\xin{ V_{\eps}(u,\tau)}^{s-1} \partial_\tau F_\tau(y)  e^{\frac{\gamma}{2} Y^\eps_\tau(y)-\frac{\gamma^2}{8}\EE[Y^\eps_\tau(y)^2]} dy \Big] \\
&= \frac{d}{d \delta}_{\vert \delta =0 } \EE\Big[\xin{ V_{\eps}(u,\tau)}^{s-1} e^{\delta \partial_\tau F_\tau(y)  - \frac{\delta^2}{2} \EE[\partial_\tau F_\tau(y) ^2] } e^{\frac{\gamma}{2} Y^\eps_\tau(y)-\frac{\gamma^2}{8}\EE[Y^\eps_\tau(y)^2]} dy \Big] \\
&= \frac{d}{d \delta}_{\vert \delta =0 } \EE\Big[\Big(\int_0^1\cT(u, z) e^{\pi \gamma P z} e^{\frac{\gamma}{2} Y^\eps_\tau(z)-\frac{\gamma^2}{8}\EE[Y^\eps_\tau(z)^2]} e^{ \frac{\gamma \delta}{2}\EE[Y_{\tau}(z) \partial_\tau F_\tau(y)] } dz \Big)^{s-1} \\
&\phantom{===========} \times e^{\frac{\gamma}{2} Y^\eps_\tau(y)-\frac{\gamma^2}{8}\EE[Y^\eps_\tau(y)^2]} e^{ \frac{\gamma \delta}{2} \EE[Y_{\tau}(y) \partial_\tau F_\tau(y)]}dy \Big] \\
&= \frac{\gamma}{2}(s-1)  \EE \Big[ 
\Big(\int_0^1 \EE[F_{\tau}(z)\partial_{\tau}F_{\tau}(y)]\cT(u, z) e^{\pi \gamma P z}   \xin{ V_{\eps}(u,\tau)}^{s - 2} e^{\frac{\gamma}{2} Y^\eps_\tau(z)-\frac{\gamma^2}{8}\EE[Y^\eps_\tau(z)^2]}dz\Big) \\
&\phantom{===========}
\times e^{\frac{\gamma}{2} Y^\eps_\tau(y)-\frac{\gamma^2}{8}\EE[Y^\eps_\tau(y)^2]} dy \Big] \\
&+ \frac{\gamma}{2} \EE[F_{\tau}(0) \partial_\tau F_{\tau}(0) ] \EE\Big[\xin{ V_{\eps}(u,\tau)}^{s-1}  e^{\frac{\gamma}{2} Y^\eps_\tau(y)-\frac{\gamma^2}{8}\EE[Y^\eps_\tau(y)^2]} dy \Big].
\end{align*}
This computation combined with \eqref{eq1:sec3} and the fact that $\partial_\tau\EE[F_\tau(0)^2]=2\EE[F_\tau(0)\partial_\tau F_\tau(0)]$ implies \eqref{eq:der-F}. }
\end{proof}

To obtain the desired cancellation, we need to perform an integration by parts on \xin{the first term of 
$\Xi_{1,\eps}^{\mathrm{uu}}$ from} $\partial_{uu} \psi^\alpha_{\chi,\eps}(u, q)$. 
\begin{lemma}\label{lem:parts} 
\xin{Suppose $(q,u,\alpha)$ are in the range~\eqref{eq:bpz-range}.}  Let:
\begin{equation*}
\xin{\hat\Xi^{\mathrm{uu}}_2\defeq} \gamma P \pi\int_0^1 \partial_u\cT(u, y) e^{\pi \gamma P y} \cV_{1, \eps}(u,y) dy+\int_0^1 \partial_{uy}\cT(u, y) e^{\pi \gamma P y} \cV_{1, \eps}(u,y) dy.    
\end{equation*}
Then with an error term $o_{\eps}(1)$ uniform in $y,z \in [0,1]^2$ we have: 
\begin{align}
\xin{\hat\Xi^{\mathrm{uu}}_2} = &\frac{\chi\gamma^3(s-1)}{8}
\int_0^1\int_0^1   \left(   \frac{\Theta_{\tau}'(y - z)}{\Theta_{\tau}(y - z)} \Big(\frac{\Theta_\tau'(u + y)}{\Theta_\tau(u + y)} - \frac{\Theta_\tau'(u + z)}{\Theta_\tau(u + z)}\Big) +o_{\eps}(1) \right)\label{eq:part} \\ \nonumber
&
\phantom{===} \times \cT(u,y)\cT(u, z)  e^{\pi \gamma P(y+z)} \cV_{2, \eps}(u, y, z) dy dz. 
\end{align}
\end{lemma}

\begin{proof}[\xin{Proof of Lemma~\eqref{lem:parts}}]
Given $u\in\band$, we can check that for some constant $c>0$, as $y\to 0$ or $y\to 1$, we have:
\begin{equation}\label{eq:partialT}
\partial_u \cT(u, y)  = \frac{\gamma \chi}{2} \left( \frac{\Theta'_{\tau}(u+y)}{\Theta_{\tau}(u+y)} - \frac{\Theta'_{\tau}(u)}{\Theta_{\tau}(u)} \right) \cT(u, y) \sim c\sin^{1-\frac{\alpha\gamma}{2}}(\pi y).
\end{equation}
\revise{Similarly, $ \partial_{uu} \cT(u, y) \sim  c\sin^{1-\frac{\alpha\gamma}{2}}(\pi y)$. Since we assume $\alpha<\frac{2}{\gamma}$, both $\partial_u \cT$ and $\partial_{uu} \cT$ are continuous in $y$  on $[0,1]$. 
We also record that for $\cT(u, y)$, one has $ \cT(u, y) \sim  c\sin^{-\frac{\alpha\gamma}{2}}(\pi y)$, which is an integrable singularity for $\alpha < \frac{2}{\gamma}$. }

By integration by parts we have 
\begin{align}
& \pi \gamma P \int_0^1 \partial_u\cT(u, y) e^{\pi \gamma P y} \cV_{1,\eps}(u,y) dy
= \int_0^1 \partial_u\cT(u, y) [\partial_y e^{\pi \gamma P y}] \cV_{1,\eps}(u,y) dy   \nonumber\\
=& -\int_0^1 \partial_{uy}\cT(u, y) e^{\pi \gamma P y} \cV_{1,\eps}(u,y) dy   -\int_0^1 \partial_u\cT(u, y) e^{\pi \gamma P y} \partial_y\cV_{1,\eps}(u, y) dy. \label{eq:part2}
\end{align}
Here there are no boundary terms since $\alpha < \frac{2}{\gamma}$.
Recalling \eqref{eq:partialT}, one has $\partial_u\cT(u, y)$ for $y=0$ or $1$.

\xin{Recall~\eqref{eq:V1-Girsanov} where we applied Cameron-Martin's Theorem (Theorem \ref{thm:Girsanov}) to 
$\cV_{1, \eps}$. 
}
Applying $\partial_y$ to~\eqref{eq:V1-Girsanov} and using Cameron-Martin's Theorem \ref{thm:Girsanov} again, we find that:
\[
\partial_y \cV_{1,\eps}(u, y) = \frac{\gamma^2}{4}(s-1)  \int_0^1 \cT(u, z)  \partial_y \EE[Y^\eps_\tau(y)Y^{\eps}_\tau(z)] e^{\pi \gamma P z} \cV_{2,\eps}(u, y, z) dz.
\]
Let $K_\eps(\cdot)$  be such that  $\EE[Y^\eps_\tau (x) Y^\eps_\tau(y)]=K_\eps(x-y)$ \guillaume{and denote $K'_\eps(y) := \partial_y K_\eps(y)$}. 
Therefore $\int_0^1 \partial_u\cT(u, y) e^{\pi \gamma P y}  \partial_y \cV_{1,\eps}(u, y) dy$ equals
\begin{align*}
&\frac{\gamma^2(s-1)}{4} \int_0^1\int_0^1 \frac{\partial_u\cT(u, y)}{\cT(u,y)} \partial_y \EE[Y^\eps_\tau(y)Y^\eps_\tau(z)]\\
&\phantom{=================} \cT(u,y)\cT(u, z)  e^{\pi \gamma P(y+z)} \cV_{2,\eps}(u, y, z) dy dz\\
=&\frac{\gamma^2(s-1)}{8} \int_0^1\int_0^1 \left(\frac{\partial_u\cT(u, y)}{\cT(u,y)}  -\frac{\partial_u\cT(u, z)}{\cT(u,z)}   \right) K'_\eps(y-z)\\
&\phantom{=================} \cT(u,y)\cT(u, z)  e^{\pi \gamma P(y+z)} \cV_{2,\eps}(u, y, z) dy dz,
\end{align*}
where we have used that  $\cT(u, y) \cT(u, z) e^{\pi \gamma P y + \pi \gamma P z} \cV_{2,\eps}(u, y, z) dy dz$
is symmetric under interchange of $y$ and $z$ and that the derivative $K'_\eps$ is an odd function.
\guillaume{
Note that $\sin(\pi y) K'_\eps(y)$ is a bounded continuous function on $[0,1]$ and uniformly converges to $\sin(\pi y) \partial_y\EE[Y_\tau(0)Y_\tau (y)]$. One can check this claim simply by explicitly writing out the convolution defining the smoothing procedure:
\begin{align*}
K_{\eps}(y) &= \int_{\mathcal{C}_+} \int_{\mathcal{C}_+} \mathbb{E}[Y_{\infty}(y-x)Y_{\infty}(x')] \eta_{\eps}(x) \eta_{\eps}(x') d^2x d^2x' + \mathbb{E}[F_{\tau}(y-x)F_{\tau}(x')],\\
 K'_{\eps}(y) &= \int_{\mathcal{C}_+} \int_{\mathcal{C}_+} \partial_y \mathbb{E}[Y_{\infty}(y-x)Y_{\infty}(x')] \eta_{\eps}(x) \eta_{\eps}(x') d^2x d^2x' + \partial_y \mathbb{E}[F_{\tau}(y-x)F_{\tau}(x')]\\
& = \int_{\mathcal{C}_+} \int_{\mathcal{C}_+} \frac{\pi}{\tan(\pi(y -x - x'))} \eta_{\eps}(x) \eta_{\eps}(x') d^2x d^2x' + \partial_y \mathbb{E}[F_{\tau}(y-x)F_{\tau}(x')].
\end{align*}
}
 Then since $\frac{\partial_u\cT(u, y)}{\cT(u,y)}  $ is uniformly bounded in $y \in [0,1]$ for a $u$ satisfying $\mathrm{Im}(u)>0$, the quantity $\left(\frac{\partial_u\cT(u, y)}{\cT(u,y)}  -\frac{\partial_u\cT(u, z)}{\cT(u,z)}   \right) K'_\eps(y-z) $ uniformly converges to 
\begin{multline*}
\frac{\gamma \chi}{2}  \partial_y \EE[Y_\tau(y)Y_\tau(z)] \Big(\frac{\Theta_\tau'(u + y)}{\Theta_\tau(u + y)} - \frac{\Theta_\tau'(u + z)}{\Theta_\tau(u + z)}\Big)\\=
-\chi\gamma \frac{\Theta_{\tau}'(y - z)}{\Theta_{\tau}(y - z)} \Big(\frac{\Theta_\tau'(u + y)}{\Theta_\tau(u + y)} - \frac{\Theta_\tau'(u + z)}{\Theta_\tau(u + z)}\Big).
\end{multline*}
\revise{Therefore we obtain the $o_{\eps}(1)$ on the right hand side of~\eqref{eq:part}, which is uniform in $y,z$. }
 Combining with~\eqref{eq:part2}, this concludes our proof. 
\end{proof}


\begin{lemma}\label{lem:E2}
We have \xin{$\lim_{\eps\to 0} \Xi^{\mathrm{uu}}_{2,\eps}+\hat\Xi^{\mathrm{uu}}_{2,\eps}+\Xi^{\mathrm{tau}}_{2,\eps}=\Xi_2$} where \begin{align*}
 \Xi_2  & \xin{\defeq} \frac{\chi^2\gamma^2}{2}s(s - 1)  \cW(q) e^{\pi \chi P u} \int_0^1  \wDelta(u, y) \cT(u, y) e^{\pi \gamma P y} \cV_1(u, y) dy; \\
 \wDelta(u, x) &\defeq  \frac{1}{2} \frac{\Theta''_\tau(u + x)}{\Theta_\tau(u + x)}
- \frac{\Theta_\tau'(u + x)}{\Theta_\tau(u + x)} \frac{\Theta_\tau'(u)}{\Theta_\tau(u)} + \frac{1}{2} \frac{\Theta'_\tau(u)^2}{\Theta_\tau(u)^2} - \frac{\pi^2}{6} - \frac{1}{6} \frac{\Theta_\tau'''(0)}{\Theta_\tau'(0)}.
\end{align*}
\end{lemma}
\begin{proof}
Combining~\eqref{eq:u-der}, Lemma~\ref{lem:partialtau}, and Lemma~\ref{lem:parts}, we have  \revise{
\begin{multline}\label{eq:E2}
\xin{\Xi^{\mathrm{uu}}_{2,\eps}+\hat\Xi^{\mathrm{uu}}_{2,\eps}+\Xi^{\mathrm{tau}}_{2,\eps}} =   \frac{\chi^2\gamma^2}{2} s(s - 1)  \cW(q) e^{\pi \chi P u}\times \\
 \int_0^1\int_0^1 (\Delta_{2}(y, z) +o_{\eps}(1) ) \cT(u, y) \cT(u, z) e^{\pi \gamma P y + \pi \gamma P z} \cV_{2, \eps}(u, y, z) dy dz
\end{multline} }
for
\begin{align*}
\Delta_{2}(y, z) &\defeq \left[
\frac12\frac{\Theta_{\tau}'(y - z)}{\Theta_{\tau}(y - z)} \Big(\frac{\Theta_\tau'(u + y)}{\Theta_\tau(u + y)} - \frac{\Theta_\tau'(u + z)}{\Theta_\tau(u + z)}\Big)\right.\\
&\phantom{=====} \left.+ \frac{1}{2} \Big(\frac{\Theta_\tau'(u + y)}{\Theta_\tau(u + y)} - \frac{\Theta_\tau'(u)}{\Theta_\tau(u)}\Big) \Big(\frac{\Theta_\tau'(u + z)}{\Theta_\tau(u + z)} - \frac{\Theta_\tau'(u)}{\Theta_\tau(u)}\Big)\right.\\
&\phantom{================} \left.- \frac{1}{4} \frac{\Theta_\tau''(y - z)}{\Theta_\tau(y - z)} - \frac{\pi^2}{6} + \frac{1}{12} \frac{\Theta'''_\tau(0)}{\Theta_\tau'(0)} 
\right],
\end{align*}
where we have used (\ref{eq:theta-heat}).  Applying  (\ref{eq:theta-iden}) with $(a,b)=(u + y, u + z)$, we get
\begin{equation}\label{eq:Delta2}
\Delta_{2}(y, z) = \frac{1}{2}\left(\wDelta(u, y) + \wDelta(u, z) \right).
\end{equation}
Therefore \xin{$\lim_{\eps\to 0} \Xi^{\mathrm{uu}}_{2,\eps}+\hat\Xi^{\mathrm{uu}}_{2,\eps}+\Xi^{\mathrm{tau}}_{2,\eps}=\Xi_2$ equals
$\frac{\chi^2\gamma^2}{2} s(s - 1)  \cW(q) e^{\pi \chi P u}$ times} 
\begin{equation}\label{eq:lim-V2}
\lim_{\eps\to 0}\int_0^1 \int_0^1 \frac{1}{2}  \Big(\wDelta(u, y) + \wDelta(u, z) \Big) \cT(u, y) \cT(u, z) e^{\pi \gamma P y + \pi \gamma P z} \cV_{2, \eps}(u, y, z) dy dz.    
\end{equation}
\xin{Recall \( \int_0^1    \cT(u, z) e^{ \pi \gamma P z} \cV_{2, \eps}(u, y, z) dz
=  \cV_{1,\eps}(u, y)\) from~\eqref{eq:V1-V2}.Therefore the limit in~\eqref{eq:lim-V2} equals 
$\int_0^1  \wDelta(u, y) \cT(u, y) e^{\pi \gamma P y} \cV_1(u, y) dy$ as desired.}
\end{proof}

\begin{proof}[Proof of Theorem~\ref{thm:bpz}]
\xin{Recall $\widetilde\Delta$ from Lemma~\ref{lem:E2}.  By~\eqref{eq:u-der}, Lemma~\ref{lem:partialtau}---Lemma~\ref{lem:E2}, we see that $\lim_{\eps\to0}\Big(\partial_{uu}+ 2\ii \pi \chi^2 \partial_\tau\Big) \psi^\alpha_{\chi,\eps}(u, q)=\Xi$} where 
\begin{equation}
    \xin{\Xi}   \defeq  s  \cW(q) e^{\pi \chi P u} \int_0^1 \Delta_1(u, y) \cT(u, y) e^{\pi \gamma P y} \cV_1(u, y) dy;   \label{eq:Xi-def}
\end{equation}
\begin{align}
\Delta_1(u, y) = -\frac{2 \chi}{\gamma} \frac{\partial_{uy} \cT(u, y)}{\cT(u, y)} + \frac{\partial_{uu} \cT(u, y)}{\cT(u, y)}+ 2\ii \pi \chi^2 \frac{\partial_\tau \cT(u, y)}{\cT(u, y)}  + (s - 1) \frac{\chi^2 \gamma^2}{2} \wDelta(u, y).\nonumber
\end{align}
\xin{We claim that \(\Xi=s  \Delta_1(u, y) \psi^\alpha_\chi(u, q)\).}
\guillaume{Then by} (\ref{eq:wp-theta}),\begin{multline*}
\xin{s  \Delta_1(u, y) - l_\chi(l_\chi + 1) \wp(u)} 
= - \frac{l_\chi}{3}(\chi - \frac{\gamma}{2})(\chi - \frac{2}{\gamma}) \frac{\Theta_\tau'''(0)}{\Theta_\tau'(0)}\\
+ l_\chi(l_\chi + 1) \Big(\frac{\Theta_\tau'(u)^2}{\Theta_\tau(u)^2} - \frac{\Theta_\tau''(u)}{\Theta_\tau(u)} + \frac{1}{3} \frac{\Theta_\tau'''(0)}{\Theta_\tau'(0)}\Big) - l_\chi(l_\chi + 1) \wp(u)= 0.
\end{multline*} 
\xin{So this claims yields  $\Big(\partial_{uu} - l_\chi(l_\chi + 1) \wp(u) + 2\ii \pi \chi^2 \partial_\tau\Big) \psi^\alpha_\chi(u, q) = 0$ as desired.}

\xin{It remains to prove \(\Xi=s  \Delta_1(u, y) \psi^\alpha_\chi(u, q)\).} We compute 
\begin{align*}
\frac{\partial_{uu}\cT(u, y)}{\cT(u, y)} &= \frac{\gamma \chi}{2} \left( \frac{\Theta_{\tau}''(u+y)}{\Theta_{\tau}(u+y)} - \frac{\Theta_{\tau}''(u)}{\Theta_{\tau}(u)} - \left(\frac{ \Theta_{\tau}'(u+y)}{\Theta_{\tau}(u+y)}\right)^2 + \left( \frac{\Theta_{\tau}'(u)}{\Theta_{\tau}(u)}\right)^2 \right)\\
&\phantom{=}+ \frac{\gamma^2 \chi^2}{4} \left( \frac{\Theta_{\tau}'(u+y)}{\Theta_{\tau}(u+y)} - \frac{\Theta_{\tau}'(u)}{\Theta_{\tau}(u)} \right)^2,\\
\frac{\partial_{uy}\cT(u, y)}{\cT(u, y)} &= \frac{\gamma \chi}{2} \left( \frac{\Theta''_{\tau}(u+y)}{\Theta_{\tau}(u+y)} - \left(\frac{\Theta'_{\tau}(u+y)}{\Theta_{\tau}(u+y)} \right)^2 \right)\\
&\phantom{=} + \frac{\gamma \chi}{2} \left( \frac{\Theta'_{\tau}(u+y)}{\Theta_{\tau}(u+y)} -  \frac{\Theta'_{\tau}(u)}{\Theta_{\tau}(u)}  \right) \left( \frac{\gamma \chi}{2}  \frac{\Theta'_{\tau}(u+y)}{\Theta_{\tau}(u+y)} - \frac{\alpha \gamma}{2}  \frac{\Theta'_{\tau}(y)}{\Theta_{\tau}(y)} \right),\\
\frac{\partial_{\tau}\cT(u, y)}{\cT(u, y)} &= \frac{1}{4 \pi \ii} \left( -\frac{\alpha \gamma}{2} \frac{\Theta''_{\tau}(y)}{\Theta_{\tau}(y)} + \frac{\gamma \chi}{2} \frac{\Theta''_{\tau}(u+y)}{\Theta_{\tau}(u+y)} - \frac{\gamma \chi}{2} \frac{\Theta''_{\tau}(u)}{\Theta_{\tau}(u)} \right).
\end{align*}
The total prefactor of $\frac{\Theta_\tau'(u + y)^2}{\Theta_\tau(u + y)^2}$ in $\Delta_1(u, y)$ is therefore
\[
-\frac{\gamma}{2} \chi + (1 + \frac{\gamma^2}{4})\chi^2 - \frac{\gamma}{2} \chi^3 = -\frac{\gamma}{2} \chi (\chi - \frac{\gamma}{2})(\chi - \frac{2}{\gamma}) = 0.
\]
Similarly, the total prefactor of $\frac{\Theta_\tau'(u)^2}{\Theta_\tau(u)^2}$ in $\Delta_1(u, y)$ is
$\frac{\gamma}{2} \chi - \frac{\alpha \gamma}{4} \chi^2 + \frac{\gamma}{4} \chi^3$. We may therefore write
\begin{align*}
\Delta_1(u, y) &= \frac{\gamma}{2}\Big(\chi - \frac{\alpha}{2} \chi^2 + \frac{1}{2} \chi^3\Big) \frac{\Theta_\tau'(u)^2}{\Theta_\tau(u)^2} + \chi \Delta_1^1(u, y) + \chi^2 \Delta_1^2(u, y) + \chi^3 \Delta_1^3(u, y)
\end{align*}
for
\begin{align*}
\Delta_1^1(u, y) &= \frac{\gamma}{2} \left(\frac{\Theta_{\tau}''(u+y)}{\Theta_{\tau}(u+y)} - \frac{\Theta_{\tau}''(u)}{\Theta_{\tau}(u)}\right),\\
\Delta_1^2(u, y)
&= - (1 + \frac{\gamma^2}{4}) \frac{\Theta''_{\tau}(u+y)}{\Theta_{\tau}(u+y)} + \frac{\alpha \gamma}{4} \frac{\Theta_\tau''(u)}{\Theta_\tau(u)} + \frac{\alpha \gamma \pi^2}{12} - \frac{\alpha \gamma}{6} \frac{\Theta_\tau'''(0)}{\Theta_\tau'(0)} + \frac{\pi^2 \gamma^2}{12} + \frac{\gamma^2}{12} \frac{\Theta_\tau'''(0)}{\Theta_\tau'(0)},\\ 
\Delta_1^3(u, y) &= \frac{\gamma}{2} \frac{\Theta''_\tau(u + y)}{\Theta_\tau(u + y)} - \frac{\gamma}{4} \frac{\Theta_\tau''(u)}{\Theta_\tau(u)} - \frac{\pi^2 \gamma}{12} - \frac{\gamma}{12} \frac{\Theta'''_\tau(0)}{\Theta_\tau'(0)}.
\end{align*}
\xin{Here to obtain $\Delta_1^2$  we used} (\ref{eq:theta-iden}) for $(a, b) = (u + y, y)$.  Adding
$0 = (-\frac{\gamma}{2} \chi + (1 + \frac{\gamma^2}{4})\chi^2 - \frac{\gamma}{2}\chi^3)\frac{\Theta_\tau''(u + y)}{\Theta_\tau(u + y)}$ \xin{to $\Delta_1(u, y)$}, we obtain
\begin{multline*}
\Delta_1(u, y) = \Big(\frac{\chi\gamma}{2} - \frac{\alpha\gamma}{4} \chi^2 + \frac{\gamma}{4} \chi^3\Big) \frac{\Theta_\tau'(u)^2}{\Theta_\tau(u)^2}
- \Big(\frac{\chi \gamma}{2} - \frac{\alpha \gamma}{4} \chi^2 + \frac{\chi^3 \gamma}{4}\Big) \frac{\Theta_\tau''(u)}{\Theta_\tau(u)}\\
+ \Big(-\frac{\chi^2 \alpha \gamma}{6} - \frac{\chi^3 \gamma}{12} + \frac{\chi^2\gamma^2}{12}\Big) \frac{\Theta'''_\tau(0)}{\Theta_\tau'(0)}
+ \Big(\frac{\pi^2 \alpha \gamma \chi^2}{12} - \frac{\pi^2 \gamma \chi^3}{12} + \frac{\pi^2 \chi^2 \gamma^2}{12}\Big).
\end{multline*}
Therefore $\Delta_1(u, y)$ does not depend on $y$.  \xin{Hence by~\eqref{eq:Xi-def} we have}
\[
\xin{\Xi} = s  \cW(q) e^{\pi \chi P u} \Delta_1(u, y) \revise{ \int_0^1  \cT(u, z) e^{\pi \gamma P z} \cV_1(u, z) dz }.
\]
\xin{Since 
\(\psi^\alpha_\chi(u, q)=  \cW(q) e^{\pi \chi P u} \int_0^1  \cT(u, z) e^{\pi \gamma P z} \cV_1(u, z) dz
\) by Lemma~\ref{lem:V1}, we get \(\Xi=s  \Delta_1(u, y) \psi^\alpha_\chi(u, q)\) as claimed.}
\end{proof}

\section{From the BPZ equation to hypergeometric
differential equations}\label{sec:ode}
In this section, we apply separation of variables to the BPZ equation in Theorem~\ref{thm:bpz} to show that,
up to a renormalization and change of variable, the coefficients of the $q$-series expansion of the $u$-deformed
block satisfy the system of hypergeometric differential equations (\ref{eq:hgf1}), giving access to
certain analytic properties of the $u$-deformed block beyond the scope of GMC.

\xin{Fix $\chi\in \{\frac{\gamma}{2},\frac{2}{\gamma}\}$ and $\alpha\in (-\frac{4}{\gamma}+\chi,Q)$.
Recall  $\Sigma^\alpha_{\chi}(u, q)$ from (\ref{eq:gen-block}). For $u $ on the upper half plane $\bbH$,  $\Sigma^\alpha_{\chi}(u, q)$  is analytic in $q$ for $|q|$ small enough. More precisely, this holds when $|q|<r_{\alpha-\chi}\wedge q_0$ and $-\frac{3}{4}\log |q|>\Im u$. 
For each $u\in\bbH$, let $\sigma^\alpha_{\chi, n}(u)$ be the $q$-series coefficients of $\Sigma^\alpha_{\chi}(u, q)$. Namely,}
\begin{equation} \label{eq:psin}
\Sigma^\alpha_{\chi}(u, q)= 
\sum_{n = 0}^\infty \sigma^\alpha_{\chi, n}(u) q^n \qquad \text{ for $u \in \HH$ and $|q|$ small enough}.
\end{equation}
\xin{By Theorem~\ref{thm:bpz}, $\psi^\alpha_\chi(u, \tau) = e^{\left(\frac{P^2}{2} + \frac{1}{6 \chi^2}l_\chi(l_\chi+1)  \right)\ii \pi \tau} \Sigma^\alpha_\chi(u,e^{\ii \pi \tau})$ satisfies the BPZ equation, where $l_\chi= \frac{\chi^2}{2} - \frac{\alpha \chi}{2}$. This yields a system of differential equations for $\sigma^\alpha_{\chi, n}(u)$. To get the desired hypergeometric equations, we need the change of variable from the following lemma.}
\begin{lemma}\label{lem:sin2}
The map $u\mapsto \sin^2(\pi u)$ is a holomorphic bijection from
$(0,1)\times(0,\infty)$ to $\CC\setminus(-\infty,1]$ which maps $\{u: \Re u =1/2,\Im u>0\}$  to $(1,\infty)$.
\end{lemma}
\xin{In light of Lemma~\ref{lem:sin2}, for $w\in \CC\setminus (-\infty,1]$, let}
\begin{equation} \label{eq:phin-def}
\phi_{\chi, n}^\alpha(w) \defeq \xin{\sin(\pi u)^{l_\chi}  \sigma_{\chi,n}^\alpha(u)}
\text{ for } w = \sin^2(\pi u) \textrm{ and } u \in (0,1)\times (0,\infty).
\end{equation}
\xin{Here $\sin(\pi u)^{l_\chi}=e^{l_\chi\log \sin(\pi u)}$ where  $\arg(\sin(\pi u))$ is specified by requiring $\arg(\sin(\pi u))=0$ when $\Re u=\frac12$.
We now describe the hypergeometric equations which will be satisfied by  $\phi_{\chi, n}^\alpha(w)$.}
Define the differential operator
\begin{equation}\label{eq:hyper-op}
\hpg_\chi \defeq w(1 - w)\partial_{ww} + (1/2 - l_\chi - (1 - l_\chi)w)\partial_w.
\end{equation}
For $n\ge 1$, recall the coefficients $\wp_n(u)$ in the $q$-series expansion of 
Weierstrass's elliptic function $\wp(u)$ from~\eqref{eq:wpn} and the polynomials  $\wwp_n(w)$ such that
$\wwp_{n}(w) = \wp_n(u)$ for $w = \sin^2(\pi u)$.
Consider the system of differential equations on the sequence of functions $ \{\phi_{n}(w)\}_{n\ge 0}$
\begin{align} \label{eq:hgf1}
\Big(\hpg_\chi  - \Big(\frac{1}{4} l_\chi^2 + \frac{1}{4} \chi^2 (P^2 + 2n)\Big) \Big) \phi_{n}(w)
= \frac{l_\chi(l_\chi+1)}{4\pi^2} \sum_{l = 1}^n  \wwp_{l}(w) \phi_{n - l}(w).
\end{align}
Here we adopt the convention that the empty summation $\sum_{l = 1}^{0}$ is $0$ so that~\eqref{eq:hgf1} is
homogeneous for $n=0$.  For $n\in\NN_0$, (\ref{eq:hgf1}) is a  \emph{hypergeometric
differential equation} with parameters $(A_{\chi, n}, B_{\chi, n}, C_\chi)$ defined by
\begin{equation} \label{eq:hgf-params}
A_{\chi, n} = -\frac{l_\chi}{2} + \ii \frac{\chi}{2} \sqrt{P^2 + 2n},
\qquad  B_{\chi, n} = -\frac{l_\chi}{2} - \ii \frac{\chi}{2}\sqrt{P^2 + 2n}, \qquad C_\chi = \frac{1}{2} - l_\chi.
\end{equation}
\begin{prop} \label{prop:bpz-to-hgf}\xin{Fix $\chi\in \{\frac{\gamma}{2},\frac{2}{\gamma}\}$ and $\alpha\in (-\frac{4}{\gamma}+\chi,Q)$.}
Equations~\eqref{eq:hgf1} hold for $\{\phi_{\chi, n}^\alpha(w)\}_{n\ge 0}$ on \xin{$\CC\setminus (-\infty,1]$.}
\end{prop}

\xin{Proposition \ref{prop:bpz-to-hgf} follows from a straightforward calculation starting from the BPZ equation, which we give in Section~\ref{subsec:sep-bpz}. It allows us to understand analytic properties $\phi^\alpha_{\chi,n}$ around 0 using the well-known solution theory for hypergeometric equations;
see Appendix \ref{sec:hgf} for a summary of needed facts  on hypergeometric equations. Since in general solutions of such equations have branch cut on the real line, it is instrumental to consider the restriction of $\phi^\alpha_{\chi,n}$ to the upper and lower half planes separately. We have Corollary~\ref{cor:decompose} below of Proposition~\ref{prop:bpz-to-hgf}, which relies on the following property.} 
 \begin{define}[Property (R)]\label{def:property-R}
A function $f$ on the closed unit disk  $\overline\D$ satisfies Property (R) if $f$ is of the form $f(w)=\sum_{n = 0}^{\infty} a_n w^n$
for $|w|\le 1$ with $\sum_{n = 0}^{\infty} |a_n|<\infty$.
\end{define}
  
\begin{corr}\label{cor:decompose}
Suppose \xin{$\gamma,\chi,\alpha$ in Proposition~\ref{prop:bpz-to-hgf} are such that}
$C_\chi$ in~\eqref{eq:hgf-params} is not an integer.
\xin{Let $\D_1=\{w\in \CC: |w|<1, \Im w>0 \}$ and $\D_2=\{w\in \CC: |w|<1, \Im w<0 \}$.} Then for $i \in \{1, 2\}$, there exist   functions $\{\phi^{\alpha, 1}_{\chi, n, i}(w)\}_{n\ge 0}$ and $\{\phi^{\alpha, 2}_{\chi, n, i}(w)\}_{n\ge 0}$  on the closed unit disk  $\overline \D$ satisfying Property (R) such that $\phi^{\alpha, 1}_{\chi, n, i}(w)$ and $w^{1 - C_\chi} \phi^{\alpha, 2}_{\chi, n, i}(w)$ are solutions to
equations~\eqref{eq:hgf1} and 
\[
\phi^\alpha_{\chi, n}(w) = \phi^{\alpha, 1}_{\chi, n, i}(w)
+ w^{1 - C_\chi} \phi^{\alpha, 2}_{\chi, n, i}(w) \quad \textrm{for }  \xin{w\in \D_i},
\]
where $w^{1-C_\chi} := e^{(1-C_\chi) \log w}$ is specified by requiring $\arg w\in (-\pi, \pi)$.
\end{corr}

We prove Corollary~\ref{cor:decompose} in Section~\ref{subsec:phi-analytic}.
\xin{Moreover, we will prove the following linear relation between these $\phi^{\alpha, j}_{\chi, n, i}$'s, which will be used several times.}
\begin{lemma}\label{lem:phase-phi}
In the setting of Corollary~\ref{cor:decompose}, for $w\in \overline\D$, we have 
\begin{equation} \label{eq:phi-prop}
\phi^{\alpha, 1}_{\chi, n, 2}(w) = e^{\pi \chi P - \ii \pi l_\chi} \phi^{\alpha, 1}_{\chi, n, 1}(w) \textrm{ and }
\phi^{\alpha, 2}_{\chi, n, 2}(w) = -e^{\pi \chi P + \ii \pi l_\chi} \phi^{\alpha, 2}_{\chi, n, 1}(w).
\end{equation}
\end{lemma}
\xin{The hypergeometric equations allow
us to analytically extend $\phi^{\alpha, j}_{\chi, n, i}(w)$  jointly  in $\alpha$ and $w$ beyond the range achieved  in Lemma~\ref{lem:psi-props} by GMC arguments. For this} we need the following generalization of Property (R).
\begin{define}\label{def:walpha}
Suppose $U\subset\CC$ is an open set. We say that a function $g(w,\alpha)$  is $(w,\alpha)$-regular on $\overline \D \times U$ if
$g$ can be written as $g(w,\alpha)=\sum_{n=0}^{\infty} a_n(\alpha) w^n$ satisfying two properties: (1) $a_n(\alpha)$ are analytic functions on $U$;
and (2) $\sum_{n=0}^{\infty}|a_n(\alpha)|<\infty$ where the convergence holds uniformly \xin{for $\alpha$ in} each compact subset of $U$.
\end{define}

\begin{lemma} \label{lem:phi-alpha-analytic}
\xin{Recall $C_\chi=\frac12-l_\chi=1-\frac{(Q-\alpha) \chi}{2}$.}
For each $i, j \in \{1, 2\}$ and  $n\in \NN_0$, there exists an open complex neighborhood $U$ of 
 $\{\alpha\in (-\frac{4}{\gamma} + \chi, Q): C_\chi\notin \ZZ \}$  such that   the function \xin{$(w,\alpha) \mapsto \phi^{\alpha, j}_{\chi, n, i}(w)$ in Corollary~\ref{cor:decompose}} have an extension to $\overline \D\times U$
which is $(w,\alpha)$-regular in the sense of Definition~\ref{def:walpha}.
\end{lemma}
\xin{Lemma~\ref{lem:phi-alpha-analytic}   will be used in Section \ref{sec:block-proof} to prove the shift equations (Theorem \ref{thm:shift}) for the probabilistic
conformal block.  We now construct a particular solution to~(\ref{eq:hgf1}), which will be used to in the proof of  Lemma~\ref{lem:phi-alpha-analytic} and Theorem \ref{thm:shift}.
To define it,} for $i=1,2$ and $n\ge 1$,  let
\begin{equation}\label{eq:defg}
g_{\chi, n, i}^{\alpha, j}(w) = \frac{l_\chi(l_\chi + 1)}{4 \pi^2} \sum_{l = 1}^n \wwp_{l}(w) \phi_{\chi, n - l, i}^{\alpha, j}(w)\quad \textrm{for }w\in \overline \D_i.
\end{equation}
Then $g_{\chi, n, i}^{\alpha, j}(w)$  satisfies Property (R) by Corollary~\ref{cor:decompose}.
\xin{By the solution theory of homogeneous hypergeometric equations (see Lemma~\ref{lem:property-R} and Section~\ref{sec:particular}),} we can define the following functions in terms of $g_{\chi, n, i}^{\alpha, j}(w)$.
\begin{define}\label{def:particular}
For $n\ge 1$, and $i=1,2$, let $G_{\chi, n, i}^{\alpha, 1}(w)$ be the unique function \xin{on $\overline\D$} satisfying Property (R) such that  $G_{\chi, n, i}^{\alpha, 1}(1)=0$ and
\[
\left(\hpg_\chi  - \left(\frac{1}{4} l_\chi^2 + \frac{1}{4} \chi^2 (P^2 + 2n)\right) \right)G_{\chi, n, i}^{\alpha, 1}(w)
= g^{\alpha, 1}_{\chi, n, i}(w) \textrm{ for }w\in \overline{\D}_i.
\]
  Let $G_{\chi, n, i}^{\alpha, 2}(w)$ be the unique function \xin{on $\overline{\D}$} satisfying Property (R) such that 
and $G_{\chi, n, i}^{\alpha, 2}(1)=0$ and 
\[
\left(\hpg_\chi  - \left(\frac{1}{4} l_\chi^2 + \frac{1}{4} \chi^2 (P^2 + 2n)\right) \right)  w^{1 - C_\chi}G_{\chi, n, i}^{\alpha, 2}(w)
= w^{1 - C_\chi}g^{\alpha, 2}_{\chi, n, i}(w)  \textrm{ for }w\in \overline{\D}_i.
\] Define $G^\alpha_{\chi, 0, i}(w) = G_{\chi, 0, i}^{\alpha, 1}(w) =G_{\chi, 0, i}^{\alpha, 2}(w)=0$ and
\[
G^\alpha_{\chi, n, i}(w) := G_{\chi, n, i}^{\alpha, 1}(w) + w^{1 - C_\chi} G_{\chi, n, i}^{\alpha, 2}(w)\qquad \textrm{for } w\in \overline\D_i \textrm{ and }  n\ge 1.
\]
\end{define}
\xin{We need the following two facts on  $G^{\alpha, j}_{\chi, n, i}$ which we prove in Section~\ref{sec:particular}.}
\begin{prop} \label{prop:phi-G}
For $i=1,2$, the functions $\{G^\alpha_{\chi, n, i}(w)\}_{n\ge 0}$ are  a solution to (\ref{eq:hgf1}) satisfying the following:
\begin{itemize}
    \item[{\bf (a).}] $G_{\chi, n, 1}^\alpha(1) = G_{\chi, n, 2}^\alpha(1) = 0$, $G_{\chi, n, 2}^{\alpha, 1}(0) = e^{\pi \chi P - \ii \pi l_\chi} G_{\chi, n, 1}^{\alpha, 1}(0)$,
and $G_{\chi, n, 2}^{\alpha, 2}(0) = -e^{\pi \chi P + \ii \pi l_\chi} G_{\chi, n, 1}^{\alpha, 2}(0)$ for each $n\ge 0$.
\item[{\bf (b).}] \xin{Lemma~\ref{lem:phi-alpha-analytic} holds with $G^{\alpha, j}_{\chi, n, i}(w)$ in place of  $\phi^{\alpha, j}_{\chi, n, i}(w)$. Namely for each $i, j \in \{1, 2\}$ and  $n\in \NN_0$, there exists an open complex neighborhood $U$ of 
 $\{\alpha\in (-\frac{4}{\gamma} + \chi, Q): C_\chi\notin \ZZ \}$  such that $(w,\alpha) \mapsto G^{\alpha, j}_{\chi, n, i}(w)$ has an extension to $\overline \D\times U$
which is $(w,\alpha)$-regular.}
\end{itemize}
\end{prop}

We now give the proofs of all the statements claimed above. 
\subsection{\xin{Proof of Proposition~\ref{prop:bpz-to-hgf}}}\label{subsec:sep-bpz}
 The  BPZ equation (\ref{eq:bpz1}) for $\psi^\alpha_\chi(u, \tau)= e^{\left(\frac{P^2}{2} + \frac{1}{6 \chi^2}l_\chi(l_\chi+1)  \right)\ii \pi \tau} \Sigma^\alpha_\chi(u,e^{\ii \pi \tau})$ implies that for $u \in \HH$, we have
\begin{multline*}
\sum_{n \geq 0} \Big[\partial_{uu}  \sigma_{\chi, n}^\alpha(u) - l_\chi(l_\chi+1)\sum_{l = 0}^n \wp_l(u)  \sigma_{\chi, n - l}^\alpha(u) - 2\pi^2\chi^2 n  \sigma_{\chi, n}^\alpha(u)\\
- 2 \pi^2 \chi^2 \Big(\frac{P^2}{2} + \frac{1}{6 \chi^2} l_\chi(l_\chi+1)\Big)  \sigma_{\chi, n}^\alpha(u)\Big] q^n = 0.
\end{multline*}
Therefore for each $n=0,1,2..$, we have 
\begin{multline*}
\left(\partial_{uu} - l_\chi(l_\chi+1) \frac{\pi^2}{\sin^2(\pi u)} - \pi^2 \chi^2 (P^2 + 2n) \right)  \sigma_{\chi, n}^\alpha(u)\\
= l_\chi(l_\chi+1) \sum_{l = 1}^n \wp_l(u)  \sigma_{\chi, n - l}^\alpha(u).
\end{multline*}
\xin{Setting $\hat \phi_{\chi, n}^\alpha(u) \defeq \phi_{\chi, n}^\alpha(\sin^2(\pi u))= \sin(\pi u)^{l_\chi}  \sigma_{\chi,n}^\alpha(u)$,}  we get
\begin{multline} \label{eq:pre-hgf}
\left(\partial_{uu} - 2 \pi l_\chi \cot(\pi u) \partial_u - \pi^2 l_\chi^2 - \pi^2\chi^2(P^2 + 2n)\right) \hat{\phi}_{\chi, n}^\alpha(u)\\
= l_\chi(l_\chi+1) \sum_{l = 1}^n \wp_l(u) \hat{\phi}_{\chi, n - l}^\alpha(u).
\end{multline}
Notice that $2 \pi \sqrt{w(1 - w)} \partial_w = \partial_u$, hence for $n\ge 0$, checking
\[
\partial_{uu} - 2 \pi l_\chi \cot(\pi u) \partial_u - \pi^2 l_\chi^2 - \pi^2\chi^2(P^2 + 2n)
= 4\pi^2 \Big( \cH_\chi 
- \Big(\frac{l_\chi^2}{4} + \frac{\chi^2}{4} (P^2 + 2n)\Big)\Big)
\]
implies that the equations~\eqref{eq:hgf1} hold for
$\{\phi_{\chi, n}^\alpha(w)\}_{n\ge 0}$. \qed

\subsection{Analytic properties via hypergeometric equations}\label{subsec:phi-analytic}
We now prove Corollary~\ref{cor:decompose} and Lemma~\ref{lem:phase-phi}.
\xin{With $\hpg_\chi$  from~\eqref{eq:hyper-op} and $A_{\chi,n}$, $B_{\chi,n}$, and $C_\chi$ from~\eqref{eq:hgf-params}, we have
\[
\hpg_\chi  - \left(\frac{1}{4} l_\chi^2 + \frac{1}{4} \chi^2 (P^2 + 2n) \right)
= (w(1 - w) \partial_{ww} + (C_\chi - (1 +A_{\chi,n} + B_{\chi,n}) w)\partial_w - A_{\chi,n}B_{\chi,n}) 
\]
This makes~\eqref{eq:hgf1} into hypergeometric equations 
of the form~\eqref{eq:gauss-hgf} which we reviewed in Appendix~\ref{sec:hgf}.} 
\begin{proof}[Proof of Corollary~\ref{cor:decompose}]
Note that  $C_\chi - A_{\chi,n}- B_{\chi,n}=\frac12$, and
\begin{multline*}
|\Im (C_\chi - A_{\chi,n})|=|\Im (C_\chi - B_{\chi,n})|\\
=|\Im (1- A_{\chi,n})|=|\Im(1- B_{\chi,n})|= \frac{\chi}{2} \sqrt{P^2+2n}.
\end{multline*}
Therefore if $C_\chi\notin\ZZ$, then $(A_{\chi,n}, B_{\chi,n},C_\chi)$ satisfy the condition in Lemma~\ref{lem:property-R} on inhomogeneous solutions.
\xin{Following Appendix~\ref{subsec:homo}, for each $n\in \NN_0$, let
\begin{align}
v^\alpha_{1, \chi, n}(w) &:= {_2F_1}(A_{\chi, n}, B_{\chi, n}, C_\chi, w),\label{eq:defv1}\\
v^\alpha_{2, \chi, n}(w) &:= {_2F_1}(1 + A_{\chi, n} - C_\chi, 1 + B_{\chi, n} - C_{\chi}, 2 - C_\chi, w),\label{eq:defv2}
\end{align}
where $ {_2F_1}$ is the Gauss hypergeometric function. Then $v^\alpha_{1, \chi, n}(w)$ and $v^\alpha_{2, \chi, n}(w)$ satisfy Property (R), and
 $v^\alpha_{1, \chi, n}(w)$ and $w^{1 - C_\chi} v^\alpha_{2, \chi, n}(w)$ form a linear  basis of the solutions to 
 $\hpg_\chi  - \left(\frac{1}{4} l_\chi^2 + \frac{1}{4} \chi^2 (P^2 + 2n) \right)f=0$.}

Fix $n\in \NN_0$ and $i\in \{1,2\}$.	Choose $w_1\neq w_2$ \xin{in $\D_i$} such that  
\begin{equation}\label{eq:linear}
v^{\alpha}_{1, \chi, n}(w_1)w_2^{1 - C_\chi} v^{\alpha}_{2, \chi, n}(w_2)-  w_1^{1 - C_\chi} v^{\alpha}_{2, \chi, n}(w_1) v^{\alpha}_{1, \chi, n}(w_2)\neq 0.
\end{equation}
\xin{For $n=0$},  $\phi^{\alpha}_{\chi, n, i}$ is linear combination of $v^\alpha_{1, \chi, n}(w)$  and  $w^{1 - C_\chi} v^\alpha_{2, \chi, n}(w)$ uniquely specified by the values of 
$\phi^{\alpha}_{\chi, n, i}(w_1)$ and $\phi^{\alpha}_{\chi, n, i}(w_2)$, \xin{since $\phi^{\alpha}_{\chi, 0, i}$ is a solution to the homogeneous hypergeometric equation.}   This gives the existence and uniqueness of 	$\phi^{\alpha, 1}_{\chi, 0, i}$ and $\phi^{\alpha, 2}_{\chi, 0, i}$.  
Note that $\wwp_{l}(w)$ are polynomials and hence entire functions for all $l\in \NN$. Now for $n=1$,
the existence and uniqueness of 	$\phi^{\alpha, 1}_{\chi, n, i}$ and $\phi^{\alpha, 2}_{\chi, n, i}$ follows from Lemma~\ref{lem:property-R}. 
Furthermore, the result for general $n$ follows from inductively applying Lemma~\ref{lem:property-R}.
\end{proof}
In light of Corollary~\ref{cor:decompose}, let  $\phi^\alpha_{\chi, n,i}$ be the continuous extension of $\phi^\alpha_{\chi,n}$ from $\D_i$ to $\overline\D_i\defeq \D_i\cup \partial \D_i$.
Namely,  
\begin{equation}\label{eq:phi-ext}
\phi^\alpha_{\chi, n,i}(w)= \phi^{\alpha, 1}_{\chi, n, i}(w)
+ w^{1 - C_\chi} \phi^{\alpha, 2}_{\chi, n, i}(w) \quad \textrm{for } w\in \overline\D_i  \textrm{ and }i\in \{1,2\}.
\end{equation} 
\xin{We will use the following lemma to prove Lemma~\ref{lem:phase-phi}.}
\begin{lemma}\label{lem:ext-phi}
For each $n\in \NN_0$ and 
\xin{$w\in [-1,0]$} we have 
\begin{equation}\label{eq:phase-phi}
\phi^\alpha_{\chi,n,1}(1)=\phi^\alpha_{\chi,n,2}(1)  \textrm{ and }\phi^\alpha_{\chi,n,2}(w)= e^{\pi \chi P-\pi \ii l_\chi} \phi^\alpha_{\chi,n,1}(w).
\end{equation}
\end{lemma}
\begin{proof}
Let $f(w)\defeq \phi^\alpha_{\chi,0}(1-w)$.  
Then $f$ solves the hypergeometric equation
for parameters $(A,B,C) = (A_{\chi,0},B_{\chi,0},1+A_{\chi,0}+B_{\chi,0}-C_{\chi})$. Since $C_\chi - A_{\chi,0}- B_{\chi,0}=\frac12$,
applying Lemma~\ref{lem:continuity} with $U\defeq \{z\in \D : 1-z\in \D_i \}$, and $D\defeq \D \setminus [0,1]$, 
we see that  as $w\in D$ tends to 0, $f(w)$ tends to a finite number $c_0$.
\xin{Therefore as $1 - w \in \xin{\CC\setminus (-\infty, 1]}$ tends to 1, $\phi^\alpha_{\chi,0}(1 - w)=f(w)$ tends to $c_0$.}  
Inductively applying Lemmas~\ref{lem:weak} and~\ref{lem:continuity}, we see that
\xin{as $1 - w \in  \CC\setminus (-\infty, 1]$ tends to 1, $\phi^\alpha_{\chi,n}(1 - w)=f(w)$ tends to a finite constant $c_n$.}
On the other hand, we have that $\phi^\alpha_{\chi,n}(w)=\phi^\alpha_{\chi,n,1}(w)$ tends to $\phi^\alpha_{\chi,n,1}(1)$ as $w\to 1$ within $\D_1$.
Therefore $\phi^\alpha_{\chi,n,1}(1)=\xin{c_n}$. The same argument gives $\phi^\alpha_{\chi,n,2}(1)=c_n$, hence $\phi^\alpha_{\chi,n,1}(1)=\phi^\alpha_{\chi,n,2}(1)$ as desired.

For the second identity, \xin{let $\hat{\phi}^\alpha_{\chi,n}(u) = \sin(\pi u)^{l_\chi}  \sigma_{\chi,n}^\alpha(u)$.} We claim that 
\begin{equation}\label{eq:shift-phase}
\yi{\hat{\phi}^\alpha_{\chi,n}(u+1)= e^{\pi \chi P-\pi \ii l_\chi} \hat{\phi}^\alpha_{\chi,n}(u)} \quad \textrm{for} \quad \yi{u \in \HH}.
\end{equation} 
Since $\frac{\gamma\chi}{2}(-\frac{\alpha}{\gamma} + \frac{\chi}{\gamma})=l_\chi$, Lemma~\ref{lem:f-nu} implies that 
$\EE\left[ f(u + 1,q)^{-\frac{\alpha}{\gamma} + \frac{\chi}{\gamma}} \right]= e^{-\pi \ii l_\chi} \EE\left[ \xin{f(u,q)}^{-\frac{\alpha}{\gamma} + \frac{\chi}{\gamma}} \right]$
\yi{for $|q|$ chosen sufficiently small for both sides to be defined}.
\yi{Because of the $e^{\pi \chi P u}$ factor in $\Sigma^\alpha_\chi(u, q)$ from~\eqref{eq:gen-block},} we get~\eqref{eq:shift-phase}.

Now note that $\phi^\alpha_{\chi,n,1}(\sin^2(\pi \ii t)) = \yi{\hat{\phi}^\alpha_{\chi,n}(\ii t)}$ and $\phi^\alpha_{\chi,n,2}(\sin^2(\pi(1+\ii t))) = \yi{\hat{\phi}^\alpha_{\chi,n}(1+\ii t)}$ \yi{for $t>0$.}  By~\eqref{eq:shift-phase}, we have $\yi{\hat{\phi}^\alpha_{\chi,n}(\ii t + 1)}= e^{\pi \chi P-\pi \ii l_\chi} \yi{\hat{\phi}^\alpha_{\chi,n}(\ii t)}$, 
hence $\phi^\alpha_{\chi,n,2}(w)= e^{\pi \chi P-\pi \ii l_\chi} \phi^\alpha_{\chi,n,1}(w)$ for  \xin{$w\in [-1,0)$.}
Taking the limit $w \to 0$ yields $\phi^\alpha_{\chi,n,2}(0)= e^{\pi \chi P-\pi \ii l_\chi} \phi^\alpha_{\chi,n,1}(0)$.
\end{proof}

\begin{proof}[Proof of Lemma~\ref{lem:phase-phi}]
By Lemma~\ref{lem:ext-phi}, we have $\phi^{\alpha, 1}_{\chi, n, 2}(0) = e^{\pi \chi P - \ii \pi l_\chi} \phi^{\alpha, 1}_{\chi, n, 1}(0)$. 
Set $\phi_n=\phi^{\alpha, 1}_{\chi, n, 2}- e^{\pi \chi P - \ii \pi l_\chi} \phi^{\alpha, 1}_{\chi, n, 1}$. 
Then $\phi_0$ is a solution to the $n=0$ case of~\eqref{eq:hgf1} which satisfies Property (R) and has the
value $\phi_0(0)=0$. By \xin{Lemma~\ref{lem:1D-R} on} the structure of the solution space of the hypergeometric equation, we must have $\phi_0\equiv 0$.
Since $\phi_1$ is a solution to~\eqref{eq:hgf1} with $n=1$, we similarly get $\phi_1\equiv 0$.  Continuing via induction on $n$, we get $\phi_n\equiv 0$, hence
$\phi^{\alpha, 1}_{\chi, n, 2}(w) = e^{\pi \chi P - \ii \pi l_\chi} \phi^{\alpha, 1}_{\chi, n, 1}(w)$ for all $n$. 

By Lemma~\ref{lem:ext-phi}, for $w \in \xin{[-1,0]}$  we have
\[
\phi^\alpha_{\chi, n, 2}(w) -\phi^{\alpha, 1}_{\chi, n, 2}(w)
=e^{\pi \chi P- \ii \pi l_\chi}(\phi^\alpha_{\chi, n, 1}(w) -  \phi^{\alpha, 1}_{\chi, n, 1}(w)).
\]
On the other hand, by~\eqref{eq:phi-ext},  since  $w^{1-C_\chi}$ has branch cut at $(-\infty, 0)$, we have on \xin{$(-1, 0)$} that
\begin{align*}
\phi^\alpha_{\chi, n, 1}(w) &= \phi^{\alpha, 1}_{\chi, n, 1}(w)
+ e^{\pi(1 - C_\chi) \ii}|w|^{1 - C_\chi} \phi^{\alpha, 2}_{\chi, n, 1}(w); \\
\phi^\alpha_{\chi, n, 2}(w) &= \phi^{\alpha, 1}_{\chi, n, 2}(w)
+ e^{-\pi(1 - C_\chi) \ii}|w|^{1 - C_\chi} \phi^{\alpha, 2}_{\chi, n, 2}(w). 
\end{align*}
Putting these together, we have $\phi^{\alpha, 2}_{\chi, n, 2}(w) =e^{\pi \chi P - \ii \pi l_\chi}  e^{2(1-C_\chi)\pi \ii}  \phi^{\alpha, 2}_{\chi, n, 1}(w) = -e^{\pi \chi P + \ii \pi l_\chi} \phi^{\alpha, 2}_{\chi, n, 1}(w)$ on \xin{$(-1,0)$.} Therefore, $\phi^{\alpha, 2}_{\chi, n, 2} = -e^{\pi \chi P + \ii \pi l_\chi} \phi^{\alpha, 2}_{\chi, n, 1}$ on $\overline \D$ by their analyticity.
\end{proof}

\subsection{The particular solution and $(w,\alpha)$-regularity} \label{sec:particular} \xin{By Definition~\ref{def:particular}, $\{G^\alpha_{\chi, n, i}(w)\}_{n\ge 0}$ solves (\ref{eq:hgf1}) as asserted in Proposition~\ref{prop:phi-G}. In this section we will first prove Proposition~\ref{prop:phi-G} (a) and then prove Lemma~\ref{lem:phi-alpha-analytic} and Proposition~\ref{prop:phi-G} (b) together.}

\begin{proof}[Proof of Proposition~\ref{prop:phi-G} (a)]
By Definition~\ref{def:particular},  \(G_{\chi, n, 1}^\alpha(1) = G_{\chi, n, 2}^\alpha(1) = 0\).
By Lemma~\ref{lem:phase-phi},  
we have $g_{\chi, n, 2}^{\alpha, 1}(w) = e^{\pi \chi P - \ii \pi l_\chi} g_{\chi, n, 1}^{\alpha, 1}(w)$ and $g_{\chi, n, 2}^{\alpha, 2}(w) = -e^{\pi \chi P + \ii \pi l_\chi} g_{\chi, n, 1}^{\alpha, 2}(w)$, 
which implies that $G_{\chi, n, 2}^{\alpha, 1}(0) = e^{\pi \chi P - \ii \pi l_\chi} G_{\chi, n, 1}^{\alpha, 1}(0)$ and
$G_{\chi, n, 2}^{\alpha, 2}(0) = -e^{\pi \chi P + \ii \pi l_\chi} G_{\chi, n, 1}^{\alpha, 2}(0)$. \qedhere
\end{proof}

\begin{proof}[\xin{Proof of Lemma~\ref{lem:phi-alpha-analytic} and 
Proposition~\ref{prop:phi-G} (b)}]
We will repeatedly use two key facts about $(w,\alpha)$-regularity. Firstly, if $f(\alpha)$ is analytic on a domain $U\subset\CC$ and $g(w)$ satisfies Property (R) on $\overline{\D}$, then $f(\alpha)g(w)$
is $(w,\alpha)$-regular on $\overline{\D}\times U$. Moreover, $(w,\alpha)$-regularity is preserved by the solution to hypergeometric differential equation as stated in Lemma~\ref{lem:alpha-ext2}.

\xin{We now inductively prove the following statement index by integer $m\ge 0$:
\begin{equation}\label{eq:ind-state}
\textrm{Proposition~\ref{prop:phi-G} (b) holds for $G^{\alpha, j}_{\chi, m, i}$ and Lemma~\ref{lem:phi-alpha-analytic} holds $\phi^{\alpha, j}_{\chi, m-1, i}$.}
\end{equation}
Since $G^{\alpha, j}_{\chi, 0, i}(w)=0$ and Statement~\eqref{eq:ind-state} on  $\phi^{\alpha, j}_{\chi, -1, i}$ is vacuous, Statement~\eqref{eq:ind-state} holds for $m=0$. Now we fix $n\ge 0$ and assume that Statement~\eqref{eq:ind-state} holds for any $m\le n$. We aim at proving  Statement~\eqref{eq:ind-state} for $m=n+1$.}

\xin{By the definition of $G^\alpha_{\chi, n, i}$ , we see that $\phi^\alpha_{\chi, n, i}(w)-G^\alpha_{\chi, n, i}(w)$ is a solution to the homogeneous version of the  hypergeometric equation~\eqref{eq:hgf1}. Recall $ v^\alpha_{1, \chi, n}(w)$ and $ v^\alpha_{2, \chi, n}(w)$ from~\eqref{eq:defv1} and~\eqref{eq:defv2}. By Lemma~\ref{lem:1D-R}}
we may write
\begin{equation} \label{eq:alpha-decomp}
\phi^\alpha_{\chi, n, i}(w) = G^\alpha_{\chi, n, i}(w) + X^1_{\chi, n, i}(\alpha) v^\alpha_{1, \chi, n}(w) +
X^2_{\chi, n, i}(\alpha) w^{1 - C_\chi} v^\alpha_{2, \chi, n}(w),
\end{equation}
for $X^1_{\chi, n, i}(\alpha)$ and $X^2_{\chi, n, i}(\alpha)$ independent of $w$. Thus  
\begin{equation} \label{eq:phi-decomp}
\phi^{\alpha, j}_{\chi, n, i}(w) = G^{\alpha, j}_{\chi, n, i}(w) + X^j_{\chi, n, i}(\alpha) v^\alpha_{j, \chi, n}(w) \quad \textrm{for }i,j\in \{1,2\}.
\end{equation}
For $i=1,2$ recall from Corollary~\ref{cor:decompose} that $ \phi^{\alpha, 1}_{\chi, n, i}(w)
+ w^{1 - C_\chi} \phi^{\alpha, 2}_{\chi, n, i}(w)=\phi^\alpha_{\chi, n}(w)$ on $w\in \D_i$. By Lemma~\ref{lem:psi-props} \xin{and Equations~\eqref{eq:psin} and~\eqref{eq:phin-def}},  $\phi^\alpha_{\chi, n}(w)$ is analytic  in $\alpha$ on a complex neighborhood of $(-\frac{4}{\gamma} + \chi, Q)$.
Due to the analyticity of $G^{\alpha, j}_{\chi, n, i}(w)$ in $\alpha$ by induction hypothesis, 
we see that 
$F^\alpha(w)\defeq X^1_{\chi, n, i}(\alpha) v^\alpha_{1, \chi, n}(w) + X^2_{\chi, n, i}(\alpha) w^{1 - C_\chi} v^\alpha_{2, \chi, n}(w)$ is analytic  in $\alpha$ on a complex neighborhood of  $\{\alpha\in (-\frac{4}{\gamma} + \chi, Q):  C_\chi\notin \ZZ \}$ for $w\in \D_i$.

For each $\alpha_0\in \CC$, there exist $w_1, w_2\in \D_i$ such that equation~\eqref{eq:linear} holds with $\alpha=\alpha_0$
and $v^{\alpha}_{j, \chi, n}(w_k)$ is analytic at $\alpha_0$ for $1\le j,k\le 2$. By solving linear systems, $X^1_{\chi, n, i}(\alpha)$ and $X^2_{\chi, n, i}(\alpha)$ can be expressed in terms of  $F^\alpha(w_k)$ and $ \{v^\alpha_{j, \chi, n}(w_k)\}_{1\le j,k\le 2}$. Therefore $X^1_{\chi, n, i}(\alpha)$ and $X^2_{\chi, n, i}(\alpha)$  are  analytic   in $\alpha$
on a complex neighborhood of  $\{\alpha\in (-\frac{4}{\gamma} + \chi, Q):  C_\chi\notin \ZZ \}$. 
By Lemmas~\ref{lem:v-R} and~\ref{lem:v-analytic} on the regularity of $v^\alpha_{j,\chi,n}(\alpha)$, equation~\eqref{eq:phi-decomp} yields that  $\phi^{\alpha, j}_{\chi, n, i}$ is $(w,\alpha)$-regular on $\overline \D \times U$.
Recall $g_{\chi, n, i}^{\alpha, j}(w)$ from~\eqref{eq:defg}. By the induction hypothesis, we see that $g_{\chi, n, i}^{\alpha, j}(w)$ is $(w,\alpha)$-regular on $\overline \D \times U$.
By Lemmas~\ref{lem:property-R} and~\ref{lem:alpha-ext2}, we see that $G^{\alpha, j}_{\chi, n+1, i}$ is $(w,\alpha)$-regular on $\overline \D \times U$. This concludes our induction.
\end{proof}

\section{Operator product expansions for conformal blocks} \label{sec:opes}
In this section, \xin{we prove that the $q$-series coefficients
$\{\cA_{\gamma, P, n}(\alpha)\}_{n \in \NN_0}$ of $\cA^q_{\gamma, P}(\alpha)$ from~\eqref{eq:ca-def},
 which were originally defined for $\alpha$ on a
complex neighborhood of $(-\frac{4}{\gamma}, Q)$ may be analytically continued to a complex neighborhood
of $(-\frac{4}{\gamma}, 2Q)$. Moreover, they are linearly related to  $\phi_{\chi, n, i}^{\alpha, j}$ from Corollary~\ref{cor:decompose}. To state the result, we introduce the following quantities.}
Fix $\chi\in\{\frac{\gamma}{2},\frac{2}{\gamma}\}$ and recall $l_\chi= \frac{\chi^2}{2} - \frac{\alpha \chi}{2}$  from~\eqref{eq:lchi}. 
Define the functions
\begin{align} \label{eq:w1}
W_\chi^-&(\alpha, \gamma) \defeq  \pi^{l_\chi} (2\pi e^{\ii \pi})^{-\frac{1}{3}\Big(2 + \frac{2 \gamma l_\chi }{\chi}+ \frac{4 l_{\chi}}{\chi \gamma} + \frac{ 6 l_\chi^2}{\chi^2}\Big)}; \\ \label{eq:w2}
W_\chi^+&(\alpha, \gamma) \defeq  -e^{2\ii \pi l_\chi - 2\ii \pi \chi^2} (2\pi e^{\ii \pi})^{-\frac{1}{3} \left( \frac{\gamma l_{\chi}}{\chi} + \frac{2 l_\chi}{\chi^2} - 8 l_{\chi} + \frac{6 l_{\chi}^2}{\chi^2} \right) } \pi^{-l_\chi - 1}  \\ \nonumber
& \frac{1 - e^{2\pi \chi P - 2\ii \pi l_\chi}}{\chi(Q - \alpha)} (\frac{4}{\gamma^2})^{\mathbf{1}_{\chi = \frac{2}{\gamma}}}\frac{\Gamma(\frac{\alpha \chi}{2} - \frac{\chi^2}{2} + \frac{2 \chi}{\gamma})\Gamma(1 - \alpha \chi) \Gamma(\alpha \chi - \chi^2)}{\Gamma(\frac{\alpha \chi}{2} - \frac{\chi^2}{2})\Gamma(1 -\frac{\gamma^2}{4})^{\frac{2 \chi}{\gamma}}}.
\end{align}
\xin{Recall from \eqref{eq:Theta'0} that} $\Theta'_{\tau}(0)$ is equal to $q^{1/4}$ times a power series in $q$ with radius of convergence $1$. We define the quantities $\eta^\pm_{\chi,  n}(\alpha)$ as coefficients of the following $q$-series expansions:
\begin{align} \label{eq:eta-m}
\Theta_{\tau}'(0)^{\frac{4}{3} \frac{l_\chi(l_\chi + 1)}{\chi^2} + \frac{2}{3} l_\chi + \frac{2}{3}} &= 
q^{\frac{1}{3} \frac{l_\chi(l_\chi + 1)}{\chi^2} + \frac{1}{6}l_\chi + \frac{1}{6}} \sum_{n = 0}^\infty \eta^-_{\chi, n}(\alpha) q^n;\\ \label{eq:eta-p}
\Theta'_\tau(0)^{\frac{4}{3} \frac{l_\chi(l_\chi + 1)}{\chi^2} - \frac{2}{3} l_\chi} &= q^{\frac{1}{3} \frac{l_\chi(l_\chi + 1)}{\chi^2} - \frac{1}{6}l_\chi} \sum_{n = 0}^\infty \eta^+_{\chi, n}(\alpha) q^n.
\end{align}
\xin{We are now ready to state the main theorem of this section.}
\begin{theorem} \label{thm:series-opes}
For each $n\in \NN_0$, the function
$\cA_{\gamma, P, n}(\alpha)$ can be analytically extended to a complex neighborhood of $(-\frac{4}{\gamma}, 2Q)$. Recall $\phi_{\chi, n, i}^{\alpha, j}$ from Corollary~\ref{cor:decompose}. 
For $\chi\in \{ \frac{\gamma}{2},\frac{2}{\gamma}\}$ and $\alpha\in (\chi,Q)$ we have
\begin{align}
\phi_{\chi, n, 1}^{\alpha, 1}(0)  &= W^-_\chi(\alpha, \gamma) \Big[\eta_{\chi, 0}^-(\alpha) \cA_{\gamma, P, n}(\alpha - \chi)
+ \sum_{m = 0}^{n - 1} \eta_{\chi, n - m}^-(\alpha) \cA_{\gamma, P, m}(\alpha - \chi)\Big];  \label{eq:opeq-0a}\\ 
\phi^{\alpha, 2}_{\chi, n, 1}(0)&= W^+_\chi(\alpha, \gamma)\Big[\eta_{\chi, 0}^+(\alpha) \cA_{\gamma, P, n}(\alpha + \chi)
+ \sum_{m = 0}^{n - 1} \eta_{\chi, n - m}^+(\alpha) \cA_{\gamma, P, m}(\alpha + \chi)\Big]\label{eq:wope-0b},
\end{align}
where we interpret $\cA_{\gamma, P, n}(\alpha)$ via the analytic extension above and use the convention that
the value of the empty summation $\sum_{m = 0}^{- 1}$ is $0$. \xin{(Note that $\phi_{\chi, n, 1}^{\alpha, 1}(0)$ is well defined for $\chi\in \{ \frac{\gamma}{2},\frac{2}{\gamma}\}$ and $\alpha\in (\chi,Q)$ since $C_\chi=\frac{1}{2}-l_\chi=1-\frac{(Q-\alpha) \chi}{2}\in (\frac12,1)$ and thus the condition of Corollary \ref{cor:decompose} is satisfied.)} 
\end{theorem}
Our proof of Theorem~\ref{thm:series-opes} relies on the operator product expansion (OPE) for the $u$-deformed
conformal blocks.  \xin{We will state the OPE results in Section~\ref{subsec:ope1}. Since these results  
follow from a direct adaptation of the methods of \cite{KRV19b,RZ20}, we will give brief proofs in Appendix \ref{sec:ope-proof}.} We complete the proof of Theorem~\ref{thm:series-opes}
in Section~\ref{subsec:proof-ope}.

\xin{Theorem~\ref{thm:series-opes} allows} us to strengthen Lemma~\ref{lem:phi-alpha-analytic} and  the last assertion of Proposition~\ref{prop:phi-G} into  the following 
Corollary \ref{corr:gap-extend}, which gives analytic extensions of $G^{\alpha, j}_{\chi, n, i}(w)$ and
$\phi^{\alpha, j}_{\chi, n, i}(w)$.

\begin{corr} \label{corr:gap-extend}
Fix $\gamma\in(0,2)$ and $\chi\in\{\frac{\gamma}{2},\frac{2}{\gamma} \}$.	There exists a complex neighborhood $V$ of  $(-\frac{4}{\gamma} + \chi, 2Q  - \chi)$ such that 
$\phi^{\alpha, j}_{\chi, n, i}(w)$  in Lemma~\ref{lem:phi-alpha-analytic} and \xin{$G^{\alpha, j}_{\chi, n, i}(w)$  in Proposition~\ref{prop:phi-G}} admit an extension on $\overline{\D}\times V$ which is $(w,\alpha)$-regular (see Definition~\ref{def:walpha}) for $i,j=1,2$ and $n\in \NN_0$.
\end{corr}
\begin{proof}
\xin{Since $\cA_{\gamma, P, n}(\alpha)$ can be analytically extended to a complex neighborhood of $(-\frac{4}{\gamma}, 2Q)$, by~\eqref{eq:opeq-0a} and~\eqref{eq:wope-0b},} there exists an open complex neighborhood $V$ of $(-\frac{4}{\gamma} + \chi, 2Q  - \chi)$ on which
$\phi_{\chi, n, 1}^{\alpha, 1}(0)$ and $\phi^{\alpha, 2}_{\chi, n, 1}(0)$ admit analytic extension in $\alpha$.  By Lemma~\ref{lem:phase-phi}, the same holds for 
$\phi_{\chi, n, 1}^{\alpha, 2}(0)$ and $\phi^{\alpha, 2}_{\chi, n, 2}(0)$. 

\xin{Recall~\eqref{eq:phi-decomp} that \(\phi^{\alpha, j}_{\chi, n, i}(w) = G^{\alpha, j}_{\chi, n, i}(w) + X^j_{\chi, n, i}(\alpha) v^\alpha_{j, \chi, n}(w)\) for $i,j\in \{1,2\}$.}
Setting $w = 0$ yields
\begin{equation}\label{eq:X-alpha}
X^j_{\chi, n, i}(\alpha) = \phi^{\alpha, j}_{\chi, n, i}(0) - G^{\alpha, j}_{\chi, n, i}(0).
\end{equation}
For $n=0$, equation~\eqref{eq:X-alpha} implies that $X^j_{\chi, n, i}(\alpha)$ admits an analytic extension to $V$, hence $\phi^{\alpha, j}_{\chi, 0, i}$ is $(w,\alpha)$-regular on $\overline{\D}\times V$
by~\eqref{eq:alpha-decomp}. Now by Lemma~\ref{lem:alpha-ext2} we see that $G^{\alpha, j}_{\chi, 1, i}$ admits a $(w,\alpha)$-regular extension on $\overline{\D}\times V$. This further implies that $X^j_{\chi, 1, i}(\alpha)$ admits an analytic extension to $V$, hence $\phi^{\alpha, j}_{\chi, 1, i}$ and  $G^{\alpha, j}_{\chi, 2, i}$ admit $(w,\alpha)$-regular extension on $\overline{\D}\times V$. Now by induction in $n$ we get that 
$\phi^{\alpha, j}_{\chi, n, i}$ and  $G^{\alpha, j}_{\chi, n, i}$ admit $(w,\alpha)$-regular extension on $\overline{\D}\times V$.
\end{proof}

Theorem~\ref{thm:series-opes} and Corollary~\ref{corr:gap-extend} are the only ingredients used in the proof of Theorem~\ref{thm:nek-block} in Section~\ref{sec:block-proof}. \xin{The reader may skip the rest of this section in the first reading.}


\subsection{Operator product expansion}\label{subsec:ope1}
Recall the $u$-deformed block $\psi_{\chi}^\alpha(u, \tau)$ from Definition~\ref{def:u-block}.
We define the renormalized deformed block by
\begin{equation} \label{eq:renorm-block-def}
\phi^\alpha_{\chi}(u, q) \defeq \sin(\pi u)^{l_\chi} \psi_{\chi}^\alpha(u, \tau).
\end{equation}
Throughout this section we assume that \xin{$\gamma\in (0,2)$, $P\in \RR$, $\tau\in \ii \RR $ and $q=e^{\ii \pi \tau }\in (0,1)$. Moreover,  $q$ and $u$ are small enough such that $(u,q)\in D^\alpha_\chi$ as in equation~\eqref{eq:d-def}.
We will provide three operator product expansion (OPE) results for $\phi^\alpha_{\chi}(u, q)$
which describe its asymptotic behavior  
as $u$ tends to $0$.}  The first one corresponds to the direct evaluation of
$\phi^\alpha_\chi(u, q)$ at $u = 0$.

\begin{lemma} \label{lem:ope-0}
\xin{Consider $u = \ii t$ with  $t \in (0, \frac{1}{2} \Im(\tau))$.}
For $\alpha \in (-\frac{4}{\gamma} + \chi , Q)$,
\begin{equation} \label{eq:ope-0a}
\xin{\lim_{u \to 0} \phi_\chi^\alpha(u, q)} = W^-_\chi(\alpha, \gamma) q^{\frac{P^2}{2} - \frac{1}{6} \frac{l_\chi(l_\chi + 1)}{\chi^2} - \frac{1}{6} l_\chi - \frac{1}{6}} \Theta_{\tau}'(0)^{\frac{4}{3} \frac{l_\chi(l_\chi + 1)}{\chi^2} + \frac{2}{3} l_\chi + \frac{2}{3}} \cA^{q}_{\gamma, P}(\alpha - \chi).
\end{equation}
\end{lemma}
\begin{proof}
By direct substitution and using \eqref{eq:Theta'0} we have
\begin{align*}
&\xin{\lim_{u \to 0} \phi_\chi^\alpha(u, q)} = q^{\frac{P^2}{2} + \frac{\gamma l_\chi}{12 \chi} - \frac{1}{6} \frac{l_\chi^2}{\chi^2} }\pi^{l_{\chi}} \Theta'_{\tau}(0)^{- \frac{2 l_\chi^2}{3 \chi^2} - \frac{2 l_\chi}{3}  + \frac{4 l_{\chi}}{3 \gamma \chi}}\\
&\phantom{======} \times \EE\left[\Big(\int_0^1 e^{\frac{\gamma}{2} Y_\tau(x)} \Theta_\tau(x)^{-\frac{\alpha \gamma}{2} + \frac{\gamma \chi}{2}} e^{\pi \gamma Px} dx\Big)^{-\frac{\alpha}{\gamma} + \frac{\chi}{\gamma}}\right]\\
& = W^-_\chi(\alpha, \gamma) q^{\frac{P^2}{2} - \frac{1}{6} \frac{l_\chi(l_\chi + 1)}{\chi^2} - \frac{1}{6} l_\chi - \frac{1}{6}} \Theta_{\tau}'(0)^{\frac{4}{3} \frac{l_\chi(l_\chi + 1)}{\chi^2} + \frac{2}{3} l_\chi + \frac{2}{3}} \cA^{q}_{\gamma, P}(\alpha - \chi). \qedhere
\end{align*}
\end{proof}
\xin{Our second OPE result concerns the next order expansion when $\chi =\frac{\gamma}{2}$ and $\alpha \in (\frac{\gamma}{2}, \frac{2}{\gamma})$. We gives its proof in Appendix \ref{sec:ope-proof}. }
 
\begin{lemma} \label{thm:opes}
Consider $u = \ii t$ with  $t \in (0, \frac{1}{2} \Im(\tau))$. Let $\chi = \frac{\gamma}{2}$ and $\alpha \in (\frac{\gamma}{2}, \frac{2}{\gamma})$. 
We have
\begin{multline} \label{eq:ope-0b}
\lim_{u \to 0} \sin(\pi u)^{- 2\xin{l_{\chi}} - 1} \Big(\phi^\alpha_{\frac{\gamma}{2}}(u, q) - \phi^\alpha_{\frac{\gamma}{2}}(0, q)\Big)\\
= W^+_\chi(\alpha, \gamma) q^{\frac{P^2}{2} + \frac{\xin{l_{\chi}}}{6} - \frac{2}{3}\frac{\xin{l_{\chi}}(1+ \xin{l_{\chi}})}{\gamma^2}}  \Theta_\tau'(0)^{\frac{16}{3} \frac{\xin{l_{\chi}}(\xin{l_{\chi}} + 1)}{\gamma^2} - \frac{2}{3}\xin{l_{\chi}}} \cA^{q}_{\gamma, P}(\alpha + \frac{\gamma}{2}).
\end{multline}
\end{lemma}

\xin{Our third and most intricate OPE result is the counterpart of Lemma~\ref{thm:opes} with $\chi = \frac{\gamma}{2}$ or $\frac{2}{\gamma}$ and $\alpha$ close to $Q$.
To state the result we need the \emph{reflection coefficient} of boundary Liouville CFT. We will recall its probabilistic meaning in Appendix~\ref{sec:ope-proof}. However, for the rest of this section we just define  it  as the explicit formula obtained in
\cite[Theorem~1.8]{RZ20}. Namely define}
\begin{multline}\label{eq:exp-R}
\xin{R(\alpha,\chi,P)\defeq} \frac{ (2 \pi)^{ \frac{2}{\gamma}(Q -\alpha ) -\frac{1}{2}} (\frac{2}{\gamma})^{ \frac{\gamma}{2}(Q - \alpha ) -\frac{1}{2} }  }{(Q-\alpha) \Gamma(1 -\frac{\gamma^2}{4}  )^{ \frac{2}{\gamma}(Q -\alpha ) } } \\
\times \frac{ \Gamma_{\frac{\gamma}{2}}(\alpha - \frac{\gamma}{2}  ) e^{- \ii \pi(\frac{\chi}{2} + \ii P)(Q-\alpha)}}{\Gamma_{\frac{\gamma}{2}}(Q- \alpha ) S_{\frac{\gamma}{2}}(\frac{\alpha}{2} + \frac{\chi}{2} + \ii P)S_{\frac{\gamma}{2}}(\frac{\alpha}{2} - \frac{\chi}{2} - \ii P) },
\end{multline}
where $\Gamma_{\frac{\gamma}{2}}(x)$ and $S_{\frac{\gamma}{2}}(x)$ are special functions  introduced in Appendix~\ref{subsec:gamma}. 

For $\alpha \in (Q, 2Q + \frac{4}{\gamma})$, we use the reflection coefficient to  define
\begin{align} \label{eq:a-extend}
&\mathcal{R}^{q}_{\gamma, P}(\alpha)\defeq  -q^{\frac{1}{6}(1 - \frac{\alpha}{\gamma} - Q(Q  +\frac{\gamma}{2} - \alpha))} \eta(q)^{ \frac{3 \alpha \gamma}{2}  +\frac{2\alpha}{\gamma} -2 -\frac{3 \alpha^2}{2} + (Q + \frac{\gamma}{2} -\alpha)(3 \alpha - 4Q)} \Theta'_{\tau}(0)^{(Q-\alpha)(\gamma - \alpha)}  \\ \nonumber
&\phantom{=}\times \frac{e^{\ii \pi (\frac{\alpha \gamma}{2} - (\alpha - \frac{\gamma}{2} -Q)(\alpha - 2Q))} (2 \pi)^{(\alpha - \frac{\gamma}{2} - Q)(Q -\alpha)}}{  (-\frac{\alpha}{\gamma} + 1)(1-e^{\pi \gamma P - \ii \pi \frac{\gamma^2}{2} +
\ii \pi \frac{\alpha \gamma}{2}}) }  \frac{\Gamma(-\frac{\gamma^2}{4}) \Gamma(\frac{2 \alpha}{\gamma} -1 - \frac{4}{\gamma^2} ) \Gamma(1 + \frac{4}{\gamma^2} - \frac{\alpha}{\gamma})}{\Gamma(\frac{\alpha \gamma}{2} -1 - \frac{\gamma^2}{2})\Gamma(1 + \frac{\gamma^2}{4} - \frac{\alpha \gamma}{2} ) \Gamma(\frac{\alpha}{\gamma} - 1)} \\ \nonumber
&\phantom{=} \times  \xin{R(\alpha - \frac{\gamma}{2},\chi,P)}  \mathbb{E} \left[  \left( \int_0^1 e^{\frac{\gamma}{2} Y_{\tau}(x)} \Theta_{\tau}( x)^{-\frac{ \gamma}{2}(2Q -\alpha )} e^{\pi  \gamma P x} dx \right)^{ \frac{\alpha}{\gamma} - \frac{4}{\gamma^2} - 1} \right].
\end{align}
\begin{lemma}\label{lem:R-analytic}
\xin{The function $\alpha\mapsto \mathcal{R}^{q}_{\gamma, P}(\alpha)$ admits an analytic continuation
in a complex neighborhood of $(Q, 2Q)$.}
\end{lemma}
\begin{proof}
By the moment bounds given by Lemma \ref{lem:GMC-moment} and the analyticity
provided by Lemma \ref{lem:psi-props}, the GMC expectation in (\ref{eq:a-extend}) is well-defined and analytic in $\alpha$ in a complex neighborhood of
$(Q, 2Q +\frac{4}{\gamma})$. The prefactor in front of the GMC expectation is an explicit meromorphic function
of $\alpha$ with known poles; the exact formula \eqref{eq:exp-R} shows that it is analytic in $\alpha$ in a complex
neighborhood of $\alpha \in (Q, 2Q)$, making the entire expression of $\mathcal{R}^{q}_{\gamma, P}(\alpha)$ analytic
in a complex neighborhood of $(Q, 2Q)$.
\end{proof}
\xin{We are now ready to state the last OPE result. The proof is also given in Appendix~\ref{sec:ope-proof}.}
\begin{lemma} \label{thm:opes2}
Consider $u = \ii t$ with  $t \in (0, \frac{1}{2} \Im(\tau))$. Let $\chi = \frac{\gamma}{2}$ or $\frac{2}{\gamma}$.  There exists
a small $\alpha_0 > 0$ such that for $\alpha \in (Q - \alpha_0, Q)$ we have
\begin{multline} \label{eq:ope-0b2}
\lim_{u \to 0} \sin(\pi u)^{- 2l_\chi - 1} \Big(\phi^\alpha_{\chi}(u, q) - \phi^\alpha_{\chi}(0, q)\Big)\\
= W^+_\chi(\alpha, \gamma) q^{\frac{P^2}{2} + \frac{l_{\chi}}{6} - \frac{1}{6}\frac{l_{\chi}(1+ \l_{\chi})}{\chi^2}}  \Theta_\tau'(0)^{\frac{4}{3} \frac{l_\chi(l_\chi + 1)}{\chi^2} - \frac{2}{3}l_\chi} \mathcal{R}^{q}_{\gamma, P}(\alpha + \chi).
\end{multline}
\end{lemma}

\xin{Note that the $\chi =\frac{\gamma}{2}$ case of Lemma~\ref{thm:opes2}}
takes the same form as Lemma \ref{thm:opes}, except with $\mathcal{R}^{q}_{\gamma, P}(\alpha + \frac{\gamma}{2})$
in place of $\mathcal{A}^{q}_{\gamma, P}(\alpha + \frac{\gamma}{2})$. This suggests that
$\mathcal{R}^{q}_{\gamma, P}(\alpha )$ gives the analytic extension of
$\mathcal{A}^{q}_{\gamma, P}(\alpha)$ beyond $\alpha = Q$.  We will prove this at
the level of $q$-series in the proof of Theorem~\ref{thm:series-opes} below.

\subsection{Proof of Theorem~\ref{thm:series-opes}}\label{subsec:proof-ope}
We \xin{first prove the easiest part of Theorem~\ref{thm:series-opes}, which does not require analytic continuation of $\cA^{q}_{\gamma, P}(\alpha)$.
\begin{lemma}\label{lem:easy-shift}
The equation \eqref{eq:opeq-0a} holds for $\chi\in \{\frac{\gamma}{2},\frac{2}{\gamma} \}$ and $\alpha\in (\chi,Q)$.
\end{lemma}
\begin{proof}

 Recall $\psi^\alpha_\chi(u, \tau) = q^{\frac{P^2}{2} + \frac{1}{6 \chi^2}l_\chi(l_\chi+1)} \Sigma^\alpha_\chi(u,q)$ and 
\(\Sigma^\alpha_{\chi}(u, q)= 
\sum_{n = 0}^\infty \sigma^\alpha_{\chi, n}(u) q^n\) from~\eqref{eq:u-block} and~\eqref{eq:psin}.
Recall $\phi^\alpha_{\chi,n}(w)= \sin(\pi u)^{l_\chi}  \sigma^\alpha_{\chi, n}(u)$ with $w=\sin^2(\pi u)$
from~\eqref{eq:phin-def}. 
Since
\(\phi^\alpha_{\chi}(u, q)=\sin(\pi u)^{l_\chi} \psi_{\chi}^\alpha(u, \tau)\), we have
\[
\phi^\alpha_{\chi}(u, q)= q^{\frac{P^2}{2} + \frac{1}{6 \chi^2}l_\chi(l_\chi+1)}\sum_{n = 0}^\infty \phi^\alpha_{\chi, n}(w) q^n \quad \textrm{for }w=\sin^2(\pi u).
\]
Recall Proposition~\ref{prop:bpz-to-hgf} and Corollary~\ref{cor:decompose} that  for $\alpha \in (-\frac{4}{\gamma} + \chi, Q)$ and $C_\chi\notin \ZZ$, we have \(\phi^\alpha_{\chi, n}(w) = \phi^{\alpha, 1}_{\chi, n, i}(w)+ w^{1 - C_\chi} \phi^{\alpha, 2}_{\chi, n, i}(w)\)  for $w\in \D_i$.
This gives \( \lim_{u \to 0} \phi_\chi^\alpha(u, q) = q^{\frac{P^2}{2} + \frac{1}{6 \chi^2}l_\chi(l_\chi+1)}\sum_{n = 0}^\infty \phi^{\alpha, 1}_{\chi, n, 1}(0) q^n\).
For $\alpha \in (-\frac{4}{\gamma} + \chi, Q)$, by Lemma~\ref{lem:ope-0} $q^{\frac{P^2}{2} + \frac{1}{6 \chi^2}l_\chi(l_\chi+1)}\sum_{n = 0}^\infty \phi^{\alpha, 1}_{\chi, n, 1}(0) q^n$ equals
\begin{equation}\label{eq:eta-series}
  W^-_\chi(\alpha, \gamma) q^{\frac{P^2}{2} - \frac{1}{6} \frac{l_\chi(l_\chi + 1)}{\chi^2} - \frac{1}{6} l_\chi - \frac{1}{6}} \Theta_{\tau}'(0)^{\frac{4}{3} \frac{l_\chi(l_\chi + 1)}{\chi^2} + \frac{2}{3} l_\chi + \frac{2}{3}} \cA^{q}_{\gamma, P}(\alpha - \chi).
\end{equation}
Since $C_\chi=1-\frac{(Q-\alpha)\chi}{2}$, for $\alpha\in (\chi,Q)$, we have $C_\chi\in (\frac12,1)$ hence $C_\chi\notin\ZZ$.  Expanding~\eqref{eq:eta-series} into a $q$ power series using the definition of $\eta^-_{\chi,n}$ from~\eqref{eq:eta-m}, we get  \eqref{eq:opeq-0a} for $\alpha\in (\chi,Q)$ as desired.
\end{proof}

We have seen that Lemma~\ref{lem:easy-shift} follows from comparing series coefficients based on the OPE result Lemma~\ref{lem:ope-0}. Using the same argument with the OPE results Lemmas~\ref{thm:opes} and~\ref{thm:opes2} instead, we get the following.
\begin{lemma}\label{lem:hard-shift}
Equation  \eqref{eq:wope-0b} holds for $\chi=\frac{\gamma}{2}$, $\alpha\in (\frac{\gamma}{2},\frac2\gamma)$.
Define $\mathcal{R}_{\gamma, P, n}(\alpha)$ by the series expansion  $\mathcal{R}^q_{\gamma, P}(\alpha)=\sum_{n=0}^\infty \mathcal{R}_{\gamma, P, n}(\alpha)q^n$. 
Then equation  \eqref{eq:wope-0b} holds with $\mathcal{R}_{\gamma, P, n}(\alpha+\chi)$ in place of $\cA_{\gamma, P, n}(\alpha + \chi)$ for $\chi\in \{\frac{\gamma}{2},\frac{2}{\gamma}\}$ and $\alpha\in (Q - \alpha_0, Q)$ for some small $\alpha_0 > 0$.
\end{lemma}
\begin{proof}
By Corollary~\ref{cor:decompose}, for $\alpha \in (-\frac{4}{\gamma} + \chi, Q)$ and $C_\chi\notin \ZZ$, we have
\begin{equation}\label{eq:lim-phi}
\phi^{\alpha, 2}_{\chi, n, 1}(0)
= \lim_{t \to 0^{+}} \sin(\pi \ii t)^{-2l_\chi - 1} \Big(\phi_{\chi, n}^\alpha(\ii t) - \phi_{\chi, n}^\alpha(0)\Big).
\end{equation}
Similarly as in Lemma~\ref{lem:easy-shift},
by Lemma~\ref{thm:opes}, for $\chi = \frac{\gamma}{2}$ and $\alpha \in (\frac{\gamma}{2}, \frac{2}{\gamma})$ we have that $q^{\frac{P^2}{2} + \frac{1}{6 \chi^2}l_\chi(l_\chi+1)}\sum_{n = 0}^\infty \phi^{\alpha, 1}_{\chi, n, 1}(0) q^n$  equals
\begin{equation}\label{eq:eta-series2}
    W^+_\chi(\alpha, \gamma) q^{\frac{P^2}{2} + \frac{\xin{l_{\chi}}}{6} - \frac{2}{3}\frac{\xin{l_{\chi}}(1+ \xin{l_{\chi}})}{\gamma^2}}  \Theta_\tau'(0)^{\frac{16}{3} \frac{\xin{l_{\chi}}(\xin{l_{\chi}} + 1)}{\gamma^2} - \frac{2}{3}\xin{l_{\chi}}} \cA^{q}_{\gamma, P}(\alpha + \frac{\gamma}{2}).
\end{equation}
Since  $C_{\frac{\gamma}{2}}=1-\frac{(Q-\alpha)\gamma}{4}\in (\frac1{2},1-\frac{\gamma^2}{8})$ for $\alpha \in (\frac{\gamma}{2}, \frac{2}{\gamma})$, we have $C_\chi\notin\ZZ$ for $\chi = \frac{\gamma}{2}$ and $\alpha \in (\frac{\gamma}{2}, \frac{2}{\gamma})$.
Expanding~\eqref{eq:eta-series2} into a $q$ power series using the definition of $\eta^+_{\chi,n}$ from~\eqref{eq:eta-p}, we get  \eqref{eq:wope-0b}.

Using Lemma~\ref{thm:opes2} instead of Lemma~\ref{thm:opes}, the identical argument
gives \eqref{eq:wope-0b} with $\mathcal{R}_{\gamma, P, n}(\alpha+\chi)$ in place of $\cA_{\gamma, P, n}(\alpha + \chi)$ for $\chi\in \{\frac{\gamma}{2},\frac{2}{\gamma}\}$ and $\alpha\in (Q - \alpha_0, Q)$ for some small $\alpha_0 > 0$.  (Note that by choosing $\alpha_0$ small enough $C_\chi=1-\frac{(Q-\alpha)\chi}{2}\notin\ZZ$   as before.) 
\end{proof} 
\begin{proof}[Proof of Theorem~\ref{thm:series-opes}]
We prove the following claim by induction:
\begin{align}\label{eq:claim-ope}
&\cA_{\gamma, P, n}(\alpha) \textrm{ admits an analytic
extension to a complex neighborhood of}\\ 
& (-\frac{4}{\gamma}, 2Q), \textrm{under which } \cA_{\gamma, P, n}(\alpha)=\mathcal{R}_{\gamma, P,n}(\alpha)
\textrm{ for }\alpha \in (Q, 2Q).\nonumber
\end{align}
 Given~\eqref{eq:claim-ope}, by Lemma~\ref{lem:hard-shift} and the analyticity of  $\phi^{\alpha, 2}_{\chi, n, 1}(0)$ from Lemma~\ref{lem:phi-alpha-analytic}, we immediately complete the proof of Theorem~\ref{thm:series-opes}.

We  first prove~\eqref{eq:claim-ope} for $n=0$. By Lemma~\ref{lem:hard-shift}, setting $n=0$ and $\chi=\frac{\gamma}{2}$ in~\eqref{eq:wope-0b} we see that for $\alpha\in (\frac{\gamma}{2},\frac2\gamma)$
\begin{equation}\label{eq:A0-ext1}
    \cA_{\gamma, P, 0}(\alpha + \frac{\gamma}{2})= [W^+_{\frac{\gamma}{2}}(\alpha, \gamma) \eta^+_{\frac{\gamma}{2}, 0}(\alpha)]^{-1} \phi^{\alpha, 2}_{\frac{\gamma}{2}, 0, 1}(0).
\end{equation}
By Lemma \ref{lem:a-props}(b), we know $\cA_{\gamma, P, 0}(\alpha)$ is analytic in a complex neighborhood of $(-\frac{4}{\gamma},Q)$. 
Since $C_{\frac{\gamma}2}=1-\frac{(Q-\alpha)\gamma}{4}\in (\frac{1}{2},1)$ for $\alpha\in (\frac{\gamma}{2},Q)$. By Lemma~\ref{lem:phi-alpha-analytic}, $\phi^{\alpha, 2}_{\frac{\gamma}{2}, n, 1}(0)$ is analytic  in $\alpha$ in a complex neighborhood of  $(\frac{\gamma}{2},Q)$. By the explicit expression for $W^+_\chi(\alpha, \gamma)$ from~\eqref{eq:w2} and 
the fact that $\eta_{\chi, 0}^+(\alpha)= (2 \pi e^{\ii \pi})^{\frac{4}{3} \frac{l_\chi(l_\chi + 1)}{\chi^2} - \frac{2}{3} l_\chi}$, the right hand side of~\eqref{eq:A0-ext1} is analytic  in $\alpha$ in a complex neighborhood of  $(\frac{\gamma}{2},Q)$.
We now define $\cA_{\gamma, P, 0}(\alpha + \frac{\gamma}{2})$
as the right hand side of~\eqref{eq:A0-ext1} on this neighborhood.
This gives a definition of $\cA_{\gamma, P, 0}(\alpha)$ on a complex neighborhood of $(\gamma,Q+\frac{\gamma}{2})$, which is consistent with its original definition on $(-\frac{4}{\gamma},Q)$. 
Therefore we have  analytically extended $\cA_{\gamma, P, 0}(\alpha)$ in $\alpha$ to a complex neighborhood $(-\frac{4}{\gamma},Q+\frac{\gamma}{2})$, under which~\eqref{eq:A0-ext1} holds
for $\alpha\in (\frac{\gamma}2, Q)$. 

On the other hand,  by Lemma~\ref{lem:hard-shift}  for some small $\alpha_0 > 0$ we have
\begin{equation}\label{eq:A0-ext}
    \cR_{\gamma, P, 0}(\alpha + \frac{\gamma}{2})= [W^+_{\frac{\gamma}{2}}(\alpha, \gamma) \eta^+_{\frac{\gamma}{2}, 0}(\alpha)]^{-1} \phi^{\alpha, 2}_{\frac{\gamma}{2}, 0, 1}(0)\quad \textrm{for } \alpha\in (Q - \alpha_0, Q).
\end{equation}
By Lemma~\ref{lem:R-analytic}, $\mathcal{R}_{\gamma, P,0}(\alpha)$ is analytic in a complex neighborhood of $(Q,2Q)$. Set $\cA_{\gamma, P, 0}(\alpha)=\mathcal{R}_{\gamma, P,0}(\alpha)$ in this neighborhood. Since 
$(-\frac{4}{\gamma},Q+\frac{\gamma}{2})\cap(Q,2Q)\neq \emptyset$, this defines an analytic extension of $\cA_{\gamma, P, 0}(\alpha)$ on a complex neighborhood of $(-\frac{4}{\gamma},2Q)$. This proves~\eqref{eq:claim-ope} for $n=0$.

For $n\ge 1$, suppose~\eqref{eq:claim-ope} holds for $m < n$. By~\eqref{eq:wope-0b} and Lemma~\ref{lem:hard-shift},  
\begin{equation}  \nonumber
\cA_{\gamma, P, n}(\alpha + \frac{\gamma}{2}) = [W^+_{\frac{\gamma}{2}}(\alpha, \gamma) \eta^+_{\frac{\gamma}{2}, 0}(\alpha)]^{-1} \phi^{\alpha, 2}_{\frac{\gamma}{2}, n, 1}(0)
- \sum_{m = 0}^{n - 1} \frac{\eta^+_{\frac{\gamma}{2}, n - m}(\alpha)}{\eta^+_{\frac{\gamma}{2}, 0}(\alpha)} \cA_{\gamma, P, m}(\alpha + \frac{\gamma}{2});
\end{equation}
\begin{equation}  \nonumber
\mathcal{R}_{\gamma, P, n}(\alpha + \frac{\gamma}{2}) = [W^+_{\frac{\gamma}{2}}(\alpha, \gamma) \eta^+_{\frac{\gamma}{2}, 0}(\alpha)]^{-1} \phi^{\alpha, 2}_{\frac{\gamma}{2}, n, 1}(0)
- \sum_{m = 0}^{n - 1} \frac{\eta^+_{\frac{\gamma}{2}, n - m}(\alpha)}{\eta^+_{\frac{\gamma}{2}, 0}(\alpha)} \mathcal{R}_{\gamma, P, m}(\alpha + \frac{\gamma}{2}).
\end{equation} 
As in the $n=0$ case, the first equation together with the induction hypothesis for~\eqref{eq:claim-ope} gives a definition of $\cA_{\gamma, P, n}(\alpha + \frac{\gamma}{2})$ for $\alpha$ in a complex neighborhood of $(\frac{\gamma}{2},Q)$. This provides an analytic extension of $\cA_{\gamma, P, n}(\alpha)$ on a complex neighborhood of $(\gamma,Q+\frac{\gamma}{2})$. Setting $\cA_{\gamma, P, n}(\alpha)=\mathcal{R}_{\gamma, P,n}(\alpha)$ on a complex neighborhood of $(Q,2Q)$, we arrive at the desired analytic extension of $\cA_{\gamma, P, n}(\alpha)$ claimed in~\eqref{eq:claim-ope}.
 \end{proof}
 }

\section{Equivalence of the probabilistic conformal block and Nekrasov partition function} \label{sec:block-proof}
Recall \xin{$\wcA^q_{\gamma, P}(\alpha)$ from~\eqref{eq:Atitlde} and $\cZ^\alpha_{\gamma, P}(q)$ from~\eqref{eq:nek-def}.  Theorem \ref{thm:nek-block} asserts that they agree as $q$-power series.
In this section, we prove Theorem \ref{thm:nek-block} by showing  that their $q$-series
coefficients satisfy the same characterizing shift equations. 
The proof is divided into six steps.  We now present these steps,  deferring the proof of a few statements to later subsections.}

Our first step is to establish the shift equations for $\cA_{\gamma, P, 0}(\alpha)$
and $\{\wcA_{\gamma, P, n}(\alpha)\}_{n \in \NN}$
defined in \eqref{eq:a-exp} and \eqref{eq:Atitlde}.
The shift equations uses  $A_{\chi,n}$, $B_{\chi,n}$, $C_\chi$ from~\eqref{eq:hgf-params} and $G_{\chi, n, i}^{\alpha, j}$ from Definition~\ref{def:particular}, which are related to the hypergeometric differential equations from
Proposition \ref{prop:bpz-to-hgf}. In particular, we need
\[
V_{\chi, n}^{\alpha, 1} \defeq G_{\chi, n, 1}^{\alpha, 1}(0) \qquad \text{and} \quad V_{\chi, n}^{\alpha, 2}\defeq G_{\chi, n, 1}^{\alpha, 2}(0)\qquad \textrm{for }n\in \NN_0.
\]
Moreover, we need the
\emph{connection coefficients} (see~\eqref{eq:app-ccoeff})
\begin{equation} \label{eq:ccoef}
\Gamma_{n, 1} := \frac{\Gamma(C_\chi) \Gamma(C_\chi - A_{\chi, n} - B_{\chi, n})}{\Gamma(C_\chi - A_{\chi, n}) \Gamma(C_\chi - B_{\chi, n})} \textrm{ and } \Gamma_{n, 2} := \frac{\Gamma(2 - C_\chi) \Gamma(C_\chi - A_{\chi, n} - B_{\chi, n})}{\Gamma(1 - A_{\chi, n}) \Gamma(1 - B_{\chi, n})}.
\end{equation}

\xin{Recall that 
$\cA_{\gamma, P, n}(\alpha)$ and $G_{\chi, n, i}^{\alpha, j}$ are analytically extended in  $\alpha$ via Theorem~\ref{thm:series-opes},  and  Corollary \ref{corr:gap-extend}, respectively. This allows us to view $\wcA^q_{\gamma, P,n}= \frac{\cA_{\gamma, P,n}}{\cA_{\gamma, P, 0}}$ from ~\eqref{eq:Atitlde} as an analytic function in $\alpha$ on a complex neighborhood of $(-\frac{4}{\gamma},2Q)$, and to view  $V_{\chi, n}^{\alpha, 1}$  and $V_{\chi, n}^{\alpha, 2}$ as analytic functions in $\alpha$ on a complex neighborhood of $(-\frac{4}{\gamma} + \chi, 2Q  - \chi)$.}
We are now ready to state the shift equations for $\cA_{\gamma, P, 0}(\alpha)$ and $\{\wcA_{\gamma, P, n}(\alpha)\}_{n \in \NN}$, which we prove in Section~\ref{sec:shift} by combining the hypergeometric differential equations and the operator product expansions from Theorem \ref{thm:series-opes}.
\begin{theorem} \label{thm:shift}
Recall explicit functions $W_\chi^\pm(\alpha, \gamma)$ from (\ref{eq:w1}) and (\ref{eq:w2})  and $\eta_{\chi, m}^\pm(\alpha)$ from
(\ref{eq:eta-m}) and (\ref{eq:eta-p}). Fix $\gamma\in (0,2)$ and $\chi \in \{\frac{\gamma}{2}, \frac{2}{\gamma}\}$. For $\alpha$
in a complex neighborhood of $(-\frac{4}{\gamma}+\chi, 2Q - \chi)$, we have
\begin{equation} \label{eq:shift-0}
\cA_{\gamma, P, 0}(\alpha - \chi) = - \frac{W^+_\chi(\alpha, \gamma)}{W^-_\chi(\alpha, \gamma)} \frac{\Gamma_{0, 2}}{\Gamma_{0, 1}} \frac{1 + e^{\pi \chi P + \ii \pi l_\chi}}{1 - e^{\pi \chi P - \ii \pi l_\chi}} \frac{\eta_{\chi, 0}^+(\alpha)}{\eta_{\chi, 0}^-(\alpha)} \cA_{\gamma, P, 0}(\alpha + \chi).
\end{equation}
Setting $\wV_{\chi, n}^{\alpha, j} =  V_{\chi, n}^{\alpha, j}W^-_\chi(\alpha, \gamma)^{-1} \eta_{\chi, 0}^-(\alpha)^{-1} \cA_{\gamma, P, 0}(\alpha - \chi)^{-1}$, we have 
\begin{multline} \label{eq:shift-n}
\wcA_{\gamma, P, n}(\alpha - \chi) + \sum_{m = 0}^{n - 1} \frac{\eta_{\chi, n - m}^-(\alpha)}{\eta_{\chi, 0}^-(\alpha)} \wcA_{\gamma, P, m}(\alpha - \chi)\\
= \frac{\Gamma_{n, 2}}{\Gamma_{n, 1}} \frac{\Gamma_{0, 1}}{\Gamma_{0, 2}} \wcA_{\gamma, P, n}(\alpha + \chi) + \frac{\Gamma_{n, 2}}{\Gamma_{n, 1}} \frac{\Gamma_{0, 1}}{\Gamma_{0, 2}} \sum_{m = 0}^{n - 1} \frac{\eta_{\chi, n - m}^+(\alpha)}{\eta_{\chi, 0}^+(\alpha)} \wcA_{\gamma, P, m}(\alpha + \chi)\\
+ \frac{\Gamma_{n, 2}}{\Gamma_{n, 1}} \frac{1 + e^{\pi \chi P + \ii \pi l_\chi}}{1 - e^{\pi \chi P - \ii \pi l_\chi}} \wV_{\chi, n}^{\alpha, 2} + \wV_{\chi, n}^{\alpha, 1}.
\end{multline}
\end{theorem} 

\xin{Our second step is to get the explicit expression~\eqref{eq:norm-answer} below for $\cA_{\gamma, P, 0}(\alpha)$. We prove it in Section~\ref{sec:denom-comp} by checking that the right side of~\eqref{eq:norm-answer} satisfies the shift equation~\eqref{eq:shift-0} as well. Proposition~\ref{prop:DenominatorBlocks} allows us to explicitly compute the normalization $Z_{\gamma, P}^{\alpha}(q)$ in Definition \ref{def:conf-block} thanks to \eqref{eq:link-Z-A}.}
\begin{prop} \label{prop:DenominatorBlocks}
For $\gamma \in (0,2)$, $\alpha \in (-\frac{4}{\gamma}, Q)$, and $P \in \RR$, we have
\begin{multline} \label{eq:norm-answer}
\cA_{\gamma, P, 0}(\alpha) = e^{\frac{\ii \pi \alpha^2}{2}} \left(\frac{\gamma}{2}\right)^{\frac{\gamma \alpha}{4}} e^{-\frac{\pi \alpha P}{2}}
\Gamma(1 - \frac{\gamma^2}{4})^{\frac{\alpha}{\gamma}}\\
\frac{\Gamma_{\frac{\gamma}{2}}(Q - \frac{\alpha}{2}) \Gamma_{\frac{\gamma}{2}}(\frac{2}{\gamma} + \frac{\alpha}{2}) \Gamma_{\frac{\gamma}{2}}(Q - \frac{\alpha}{2} - \ii P) \Gamma_{\frac{\gamma}{2}}(Q - \frac{\alpha}{2} + \ii P)}
     {\Gamma_{\frac{\gamma}{2}}(\frac{2}{\gamma})\Gamma_{\frac{\gamma}{2}}(Q - \ii P)\Gamma_{\frac{\gamma}{2}}(Q + \ii P)\Gamma_{\frac{\gamma}{2}}(Q - \alpha)}.
\end{multline}
\end{prop}

\xin{Our third step is to prove a uniqueness result for a shift equation closely related to~\eqref{eq:shift-n}.}  For $n\in\NN$, consider the shift equation
\begin{equation} \label{eq:abs-shift}
X_n(\alpha - \chi) = Y_n(\chi, \alpha) X_n(\alpha + \chi) + Z_n(\chi, \alpha)
\end{equation}
on unknown function $X_n(\alpha)$, where we set
$Y_n(\chi, \alpha)\defeq \frac{\Gamma_{n, 2} \Gamma_{0, 1}}{\Gamma_{n, 1} \Gamma_{0, 2}}$ and 
\begin{multline} \label{eq:zn-def}
Z_n(\chi, \alpha) \defeq - \sum_{m = 0}^{n - 1} \frac{\eta_{\chi, n - m}^-(\alpha)}{\eta_{\chi, 0}^-(\alpha)} \wcA_{\gamma, P, m}(\alpha - \chi)\\
+ \frac{\Gamma_{n, 2} \Gamma_{0, 1}}{\Gamma_{n, 1} \Gamma_{0, 2}} \sum_{m = 0}^{n - 1} \frac{\eta_{\chi, n - m}^+(\alpha)}{\eta_{\chi, 0}^+(\alpha)} \wcA_{\gamma, P, m}(\alpha + \chi)
+ \frac{\Gamma_{n, 2}}{\Gamma_{n, 1}} \frac{1 + e^{\pi \chi P + \ii \pi l_\chi}}{1 - e^{\pi \chi P - \ii \pi l_\chi}} \wV_{\chi, n}^{\alpha, 2} +  \wV_{\chi, n}^{\alpha, 1}.
\end{multline}
By~\eqref{eq:shift-n},  we see that \eqref{eq:abs-shift} holds
with $\wcA_{\gamma, P, n}$ in place of $X_n$ for each $n\in\NN$. \xin{However,
it is crucial that in~\eqref{eq:abs-shift} we view $X_n$ as the only unknown function while $Y_n$ and $Z_n$ are viewed as given. It is the uniqueness of this shift equation that is most convenient  for  the proof of Theorem~\ref{thm:nek-block}. We state this uniqueness as Proposition~\ref{prop:shift-unique} and prove it in Section~\ref{sec:denom-comp} along with Proposition~\ref{prop:DenominatorBlocks}.}
\begin{prop} \label{prop:shift-unique}
Fix $\gamma\in(0,2)$ with $\gamma^2$ irrational and $P\in \RR$. For $n\in\NN$, let
$X^1_n(\alpha)$ and $X^2_n(\alpha)$ be meromorphic functions on a complex neighborhood $V$ of
$(-\frac{4}{\gamma}, 2Q)$.  Suppose \eqref{eq:abs-shift} for $\chi \in \{\frac{\gamma}{2}, \frac{2}{\gamma}\}$
and $\alpha\in(-\frac{4}{\gamma}+\chi, 2Q - \chi)$ holds with $X^i_n$ in place of $X_n$ for $i=1,2$, and moreover,  $X^1_n(\alpha_0) = X^2_n(\alpha_0)$ for some $\alpha_0 \in (-\frac{4}{\gamma}+\chi, 2Q - \chi)$.  Then
$X^1_n(\alpha) = X^2_n(\alpha)$ for all $\alpha\in V$.
\end{prop}

Our fourth step is to establish Zamolodchikov's recursion when $-\frac{\alpha}{\gamma}\in \NN$: 
\begin{theorem}\label{thm:int-RecursionThm} 
Fix $\gamma \in (0,2)$, $\alpha \in (-\frac{4}{\gamma}, Q)$, and $q\in (0,1)$. If $-\frac{\alpha}{\gamma}\in \NN$, then $P\mapsto \wcA^q_{\gamma, P}(\alpha)$ admits a meromorphic extension to  all of $\CC$ under which
\begin{equation} \label{eq:int-ZRecur}
\wcA^q_{\gamma, P}(\alpha) =
\sum_{n,m = 1}^\infty  q^{2nm}\frac{R_{\gamma, m, n}(\alpha)}{P^{2} - P^2_{m,n}}\wcA^q_{\gamma, P_{-m, n}}(\alpha)
+ [q^{-\frac{1}{12}} \eta(q)]^{\alpha(Q - \frac{\alpha}{2}) - 2},
\end{equation}
where $R_{\gamma, m, n}(\alpha)$ and $P_{m, n}$ are explicitly defined in  (\ref{eq:rmn-def}) and (\ref{eq:pmn-def}).
\end{theorem}
\xin{The key to the proof of Theorem~\ref{thm:int-RecursionThm} is that when $-\frac{\alpha}{\gamma}\in \NN$ we have a  Dotsenko-Fateev integral representation of   $\cA^q_{\gamma, P}(\alpha)$, which we analyze in detail in Section~\ref{sec:pre-Z}. From this representation and Proposition~\ref{prop:DenominatorBlocks},  we will show  in Section~\ref{sec:int-Z} that for $-\frac{\alpha}{\gamma}\in \NN$ all poles of $\wcA^q_{\gamma, P}(\alpha)$ are simple, lie in $\{\pm P_{m, n}: n\in \NN \textrm{ and }1 \leq m \leq N\}$, and have residues given by $R_{\gamma,m,n}(\alpha)$ in~\eqref{eq:int-ZRecur}. To conclude the proof, we then show in Section~\ref{sec:int-Z}  that $\lim_{P\to\infty} \wcA^q_{\gamma, P}(\alpha)= [q^{-\frac{1}{12}} \eta(q)]^{\alpha(Q - \frac{\alpha}{2}) - 2}$, giving the final term in~\eqref{eq:int-ZRecur}. }
 
Our fifth step is to show Theorem \ref{thm:nek-block} for $-\frac{\alpha}{\gamma}\in \NN$
using Theorem \ref{thm:int-RecursionThm}. \xin{This argument  is known in physics, but we provide it below for completeness.}
\begin{theorem} \label{thm:int-nek-block}
Suppose $-\frac{\alpha}{\gamma}\in \NN$ for $\gamma \in (0,2)$ and $\alpha \in (-\frac{4}{\gamma}, Q)$, and $q \in (0, 1)$.
Let $\wcA^q_{\gamma, P}(\alpha)$ be defined under the meromorphic extension to $P \in \CC$ from Theorem~\ref{thm:int-RecursionThm}.
Then  $\cZ^\alpha_{\gamma, P}(q) = \wcA^q_{\gamma, P}(\alpha)$ as formal $q$-series.
\end{theorem}
\begin{proof}
By Theorem \ref{thm:int-RecursionThm}, (\ref{eq:block-def-nek}),
and (\ref{eq:z-recur}), when $N \in \NN$, the formal $q$-series expansions
for both $\cZ^\alpha_{\gamma, P}(q)$ and the meromorphic continuation of $\wcA^q_{\gamma, P}(\alpha)$ solve
the recursion (\ref{eq:int-ZRecur}).  Denoting their difference by
$\Delta^q_{\gamma, P}(\alpha) = \sum_{n = 0}^\infty \Delta_{\gamma, P, n}(\alpha) q^n$,
we find by subtraction that
\[
\sum_{n = 0}^\infty \Delta_{\gamma, P, n}(\alpha) q^n =\sum_{n,m = 1}^\infty  q^{2nm} \frac{R_{\gamma, m, n}(\alpha)}{P^2 - P_{m, n}^2}  \sum_{k = 0}^\infty \Delta_{\gamma, P_{-m,n}, k}(\alpha) q^k.
\] 
Equating $q$-series coefficients of both sides expresses $\Delta_{\gamma, P, n}(\alpha)$ as a linear
combination of $\Delta_{\gamma, P, m}(\alpha)$ with $m < n$.  By the form of the right hand side,
we find $\Delta_{\gamma, P, 0}(\alpha) = 0$, and an induction shows $\Delta_{\gamma, P, n}(\alpha) = 0$
as needed.
\end{proof}

Our final step is to put everything together and prove Theorem \ref{thm:nek-block}. 
\xin{This is done by  combining Theorem~\ref{thm:int-nek-block} with a detailed
analysis of the shift equation from Theorem~\ref{thm:shift} and
Proposition~\ref{prop:shift-unique}; see Section~\ref{sec:nek-block-proof}.} In the rest of the section we give detailed proofs for Steps 1,2,3,4,6 above.


\subsection{Proof of Theorem~\ref{thm:shift}} \label{sec:shift}
For $i,j\in\{1,2\}$ and $n\in \NN_0$, recall $\phi_{\chi, n, i}^{\alpha, j}(w)$  and $\phi^\alpha_{\chi, n, i}(w) $ from Corollary~\ref{cor:decompose}, where
we use the analytic extension of $\phi_{\chi, n, i}^{\alpha, j}(w)$ in $\alpha$ from Corollary \ref{corr:gap-extend}.
Since $\phi^\alpha_{\chi, n, i}(w) - G_{\chi, n, i}^\alpha(w)$ is a special solution to the homogeneous variant of~\eqref{eq:hyper-op}, 
the discussion of the linear solution space around $0$ and $1$ of such differential equations in Appendix~\ref{subsec:homo}
implies that for some $X_{\chi, n, i}^j(\alpha), Y_{\chi, n, i}^j(\alpha)$ we have
\begin{align*}
&\phi^\alpha_{\chi, n, i}(w) =  G_{\chi, n, i}^\alpha(w) + X_{\chi, n, i}^1(\alpha)\, {_2F_1}(A_{\chi, n}, B_{\chi, n}, C_\chi; w) \\
&+ X_{\chi, n, i}^2(\alpha)\, w^{1 - C_\chi} {_2F_1}(1 + A_{\chi, n} - C_\chi, 1 + B_{\chi, n} - C_\chi, 2 - C_\chi; w); \\
&\phi^\alpha_{\chi, n, i}(w) = G_{\chi, n, i}^\alpha(w) + Y_{\chi, n, i}^1(\alpha)\, {_2F_1}(A_{\chi, n}, B_{\chi, n}, 1 + A_{\chi, n} + B_{\chi, n} - C_\chi; 1 - w)\\
&+ Y_{\chi, n, i}^2(\alpha)\, (1 - w)^{C_\chi - A_{\chi, n} - B_{\chi, n}} {_2F_1}(C_\chi - A_{\chi, n}, C_\chi - B_{\chi, n}, 1 + C_\chi - A_{\chi, n} - B_{\chi, n}; 1 - w).
\end{align*}
Together, these equations imply for $i \in \{1, 2\}$ that
\begin{align*}
\phi_{\chi, n, i}^{\alpha, 1}(w) &= G_{\chi, n, i}^{\alpha, 1}(w) + X_{\chi, n, i}^1(\alpha) \, {_2F_1}(A_{\chi, n}, B_{\chi, n}, C_{\chi}; w);\\
\phi_{\chi, n, i}^{\alpha, 2}(w) &= G_{\chi, n, i}^{\alpha, 2}(w)
+ X_{\chi, n, i}^2(\alpha)\, {_2F_1}(1 + A_{\chi, n} - C_\chi, 1 + B_{\chi, n} - C_\chi, 2 - C_\chi; w).
\end{align*}
 By the connection equation (\ref{eq:app-ccoeff}) \xin{with the coefficients $\Gamma_{n,i}$ in~\eqref{eq:ccoef},} we have
\[
Y_{\chi, n, i}^1(\alpha) = \Gamma_{n, 1} X_{\chi, n, i}^1(\alpha) + \Gamma_{n, 2} X_{\chi, n, i}^2(\alpha) \quad \textrm{for } i \in \{1, 2\}.
\]
Because $\phi^\alpha_{\chi, n, 1}(1) = \phi^\alpha_{\chi, n, 2}(1)$,
$G_{\chi, n, 1}^\alpha(1) = G_{\chi, n, 2}^\alpha(1) = 0$,
and $C_{\chi} - A_{\chi, n} - B_{\chi, n} = \frac{1}{2}$, this implies that
\begin{equation} \label{eq:connect-eq}
X_{\chi, n, 1}^1(\alpha) - X_{\chi, n, 2}^1(\alpha) = - \frac{\Gamma_{n, 2}}{\Gamma_{n, 1}} (X_{\chi, n, 1}^2(\alpha) - X_{\chi, n, 2}^2(\alpha)).
\end{equation}
\xin{By $\phi^{\alpha, 1}_{\chi, n, 2}(0) = e^{\pi \chi P - \ii \pi l_\chi} \phi^{\alpha, 1}_{\chi, n, 1}(0) $ and $ 
\phi^{\alpha, 2}_{\chi, n, 2}(0) = -e^{\pi \chi P + \ii \pi l_\chi} \phi^{\alpha, 2}_{\chi, n, 1}(0)$ from Lemma~\ref{lem:phase-phi}, we have}
\begin{align*}
X_{\chi, n, 2}^1&(\alpha) + G_{\chi, n, 2}^{\alpha, 1}(0) = \phi_{\chi, n, 2}^{\alpha, 1}(0) = e^{\pi \chi P - \ii \pi l_\chi} (X_{\chi, n, 1}^1(\alpha) + G_{\chi, n, 1}^{\alpha, 1}(0)) \\
X_{\chi, n, 2}^2&(\alpha) + G_{\chi, n, 2}^{\alpha, 2}(0) = \phi_{\chi, n, 2}^{\alpha, 2}(0) = -e^{\pi \chi P + \ii \pi l_\chi} (X_{\chi, n, 1}^2(\alpha) + G_{\chi, n, 1}^{\alpha, 2}(0))
\end{align*}
\xin{Combined with (\ref{eq:connect-eq})  and Proposition \ref{prop:phi-G} on $G_{\chi, n, i}^{\alpha, j}$, this gives}
\begin{equation}\label{eq:rel}
  (1 - e^{\pi \chi P - \ii \pi l_\chi}) X_{\chi, n, 1}^1(\alpha) = - \frac{\Gamma_{n, 2}}{\Gamma_{n, 1}} (1 + e^{\pi \chi P + \ii \pi l_\chi}) X_{\chi, n, 1}^2(\alpha).  
\end{equation} 
\xin{Recall $W^\pm_\chi(\alpha, \gamma)$ defined in (\ref{eq:w1}) and (\ref{eq:w2}). By Theorem~\ref{thm:series-opes} we have}
\begin{align*}
X_{\chi, n, 1}^1&(\alpha) + G_{\chi, n, 1}^{\alpha, 1}(0) = \phi_{\chi, n, 1}^{\alpha, 1}(0)\\
&= W_\chi^-(\alpha, \gamma) \Big[\eta_{\chi, 0}^-(\alpha) \cA_{\gamma, P, n}(\alpha - \chi) + \sum_{m = 0}^{n - 1} \eta_{\chi, n - m}^-(\alpha) \cA_{\gamma, P, m}(\alpha - \chi)\Big]\\
X_{\chi, n, 1}^2&(\alpha) + G_{\chi, n, 1}^{\alpha, 2}(0) = \phi_{\chi, n, 1}^{\alpha, 2}(0)\\
&= W_\chi^+(\alpha, \gamma)
\Big[\eta_{\chi, 0}^+(\alpha) \cA_{\gamma, P, n}(\alpha + \chi) + \sum_{m = 0}^{n - 1} \eta_{\chi, n - m}^+(\alpha) \cA_{\gamma, P, m}(\alpha + \chi)\Big].
\end{align*}
\xin{Recall our definition $V_{\chi, n}^{\alpha, j} =G_{\chi, n, 1}^{\alpha, j}(0)$.  By ~\eqref{eq:rel} we have} 
\begin{align*}
&\eta^-_{\chi, 0}(\alpha) \cA_{\gamma, P, n}(\alpha - \chi)\\
&= W^-_\chi(\alpha, \gamma)^{-1} (X_{\chi, n, 1}^1(\alpha) + G_{\chi, n, 1}^{\alpha, 1}(0)) - \sum_{m = 0}^{n - 1} \eta_{\chi, n - m}^-(\alpha) \cA_{\gamma, P, m}(\alpha - \chi)\nonumber\\
&= - W^+_\chi(\alpha, \gamma) W^-_\chi(\alpha, \gamma)^{-1} \frac{\Gamma_{n, 2}}{\Gamma_{n, 1}} \frac{1 + e^{\pi \chi P + \ii \pi l_\chi}}{1 - e^{\pi \chi P - \ii \pi l_\chi}} \eta_{\chi, 0}^+(\alpha) \cA_{\gamma, P, n}(\alpha + \chi)\nonumber\\
&\phantom{=} - W^+_\chi(\alpha, \gamma) W^-_\chi(\alpha, \gamma)^{-1} \frac{\Gamma_{n, 2}}{\Gamma_{n, 1}} \frac{1 + e^{\pi \chi P + \ii \pi l_\chi}}{1 - e^{\pi \chi P - \ii \pi l_\chi}} \sum_{m = 0}^{n - 1}\eta_{\chi, n - m}^+(\alpha) \cA_{\gamma, P, m}(\alpha + \chi)\nonumber\\
&\phantom{=} + W^-_\chi(\alpha, \gamma)^{-1} \frac{\Gamma_{n, 2}}{\Gamma_{n, 1}} \frac{1 + e^{\pi \chi P + \ii \pi l_\chi}}{1 - e^{\pi \chi P - \ii \pi l_\chi}} V_{\chi, n}^{\alpha, 2} + W^-_\chi(\alpha, \gamma)^{-1} V_{\chi, n}^{\alpha, 1} \\
&\phantom{=}- \sum_{m = 0}^{n - 1} \eta^-_{\chi, n - m}(\alpha) \cA_{\gamma, P, m}(\alpha - \chi).\nonumber
\end{align*}
Specializing the above equation to $n = 0$ yields~\eqref{eq:shift-0}. For $n \geq 1$, dividing both sides of 
the equation by $W^-_\chi(\alpha, \gamma)\eta_{\chi, 0}^-(\alpha) \cA_{\gamma, P, 0}(\alpha - \chi)$ and
applying (\ref{eq:shift-0}), we get ~\eqref{eq:shift-n}.  This concludes  the proof. \qed

\subsection{Proof of \xin{Propositions \ref{prop:DenominatorBlocks} and~\ref{prop:shift-unique}}} \label{sec:denom-comp}Let $A(\alpha)$ be the claimed expression for $\cA_{\gamma, P, 0}(\alpha)$ given by the right-hand side of \eqref{eq:norm-answer}. \xin{We first show that $A(\alpha)$  and $\cA_{\gamma, P, 0}(\alpha)$ satisfy the same shift equation.}
\begin{lemma}\label{lem:YW}
 \xin{We have
$A(\alpha - \chi) = Y_0(\alpha,\chi) A(\alpha +\chi)$}
where $Y_0(\alpha,\chi)$ equals 
\begin{equation*} 
e^{4\ii \pi l_\chi - 2\ii \pi \chi^2} e^{\pi \chi P} \Gamma(1 -\frac{\gamma^2}{4})^{-\frac{2 \chi}{\gamma}} 
\frac{\Gamma(\frac{2 \chi}{\gamma} - l_\chi)\Gamma(1 + 2 l_\chi - \chi^2)\Gamma(1 + 2 l_\chi)}{\Gamma(1 + l_\chi)\Gamma(1 + l_\chi - \ii \chi P) \Gamma(1 + l_\chi + \ii \chi P)} (\frac{4}{\gamma^2})^{ \mathbf{1}_{\chi = \frac{2}{\gamma}}}.    
\end{equation*} 
Moreover, \xin{\(\cA_{\gamma, P, 0}(\alpha - \chi)=Y_0(\alpha,\chi)  \cA_{\gamma, P, 0}(\alpha +\chi),\)
where $\cA_{\gamma, P, 0}(\alpha)$ is   extended to a complex neighborhood of $(-\frac{4}{\gamma},2Q)$  as in Theorem~\ref{thm:series-opes}.}
\end{lemma}
\begin{proof}
By  \eqref{eq:sh_double_gamma}, we have $A(\alpha - \chi) = Y_0(\alpha,\chi) A(\alpha +\chi)$.
 In light of \eqref{eq:shift-0}, to prove \(\cA_{\gamma, P, 0}(\alpha - \chi)=Y_0(\alpha,\chi)  \cA_{\gamma, P, 0}(\alpha +\chi),\) it suffices to show
\begin{equation}\label{eq:desired}
Y_0(\alpha,\chi) = - \frac{W^+_\chi(\alpha, \gamma)}{W^-_\chi(\alpha, \gamma)} \frac{\Gamma_{0, 2}}{\Gamma_{0, 1}}
\frac{1 + e^{\pi \chi P + \ii \pi l_\chi}}{1 - e^{\pi\chi P - \ii \pi l_\chi}}
\frac{\eta_{\chi, 0}^+(\alpha)}{\eta_{\chi, 0}^-(\alpha)}.
\end{equation}
where $W_\chi^\pm(\alpha, \gamma)$ and $\eta_{\chi, m}^\pm(\alpha)$ are as in (\ref{eq:w1})---(\ref{eq:eta-p}).
By (\ref{eq:w1}) and (\ref{eq:w2})   
\begin{multline*}
-\frac{W^+_\chi(\alpha, \gamma)}{W^-_\chi(\alpha, \gamma)}
= \frac{e^{2\ii \pi l_{\chi} - 2\ii \pi \chi^2} (2\pi e^{\ii \pi})^{-\frac{1}{3}( \frac{\gamma l_{\chi}}{\chi} + \frac{2 l_\chi}{\chi^2} - 8 l_{\chi} + \frac{6 l_{\chi}^2}{\chi^2} )}}
{ (2\pi e^{\ii \pi})^{-\frac{1}{3}(2 + \frac{2 \gamma l_\chi }{\chi}+ \frac{4 l_{\chi}}{\chi \gamma} + \frac{ 6 l_\chi^2}{\chi^2})}} \pi^{-2l_\chi - 1}\\
\revise{ \frac{1 - e^{2\pi \chi P - 2\ii \pi l_\chi}}{\chi(Q - \alpha)}\frac{\Gamma(\frac{2 \chi}{\gamma} - l_\chi)\Gamma(1 + 2 l_\chi - \chi^2) \Gamma(-2l_\chi)}{\Gamma(-l_\chi)\Gamma(1 -\frac{\gamma^2}{4})^{\frac{2 \chi}{\gamma}}}
	(\frac{4}{\gamma^2})^{\mathbf{1}_{\chi = \frac{2}{\gamma}}}.}
\end{multline*}
By the reflection and duplication formulas for the gamma function (see~\eqref{eq:reflection-sin} and~\eqref{eq:duplication}), we have
\begin{multline*}
\frac{\Gamma_{0, 2}}{\Gamma_{0, 1}}
= \frac{\Gamma(2 - C_\chi) \Gamma(C_\chi - A_{\chi, 0}) \Gamma(C_\chi - B_{\chi, 0})}{\Gamma(C_\chi) \Gamma(1 - A_{\chi, 0}) \Gamma(1 - B_{\chi, 0})}\\
\revise{=  \frac{2^{2l_\chi} \pi \Gamma(\frac{3}{2} + l_\chi) }
	{\Gamma(\frac{1}{2} - l_\chi) \Gamma(1 + l_\chi - \ii \chi P) \Gamma(1 + l_\chi + \ii \chi P) \cos(\frac{\pi}{2} l_\chi - \ii \pi \frac{\chi P}{2}) \cos(\frac{\pi}{2} l_\chi + \ii \pi \frac{\chi P}{2})}.}
\end{multline*}
By (\ref{eq:eta-m}) and (\ref{eq:eta-p}) we have
$\frac{\eta^+_{\chi, 0}(\alpha)}{\eta^-_{\chi, 0}(\alpha)} = (2\pi e^{\ii \pi})^{- \frac{4}{3}l_\chi - \frac{2}{3}}$.
The desired identity~\eqref{eq:desired} follows by combining these identities and simplifying.
\end{proof}
\begin{proof}[Proof of Proposition~\ref{prop:DenominatorBlocks}]

Recalling that $l_\chi= \frac{\chi^2}{2} - \frac{\alpha \chi}{2}$ from (\ref{eq:lchi}), we observe
that $Y_0(\chi, \alpha)$ is meromorphic in $\alpha\in\CC$ with countably many zeros and poles, so
we can find $\alpha_0\in (-\frac{4}{\gamma}+\chi, 2Q-\chi)$ such that $A(\alpha_0) \neq 0$ and $Y_0(\chi, \alpha)$ has no zeros
or poles in the set $\alpha_0 + \ZZ \gamma + \ZZ \frac{4}{\gamma}$.
Let $\Delta_0(\alpha)\defeq \cA_{\gamma, P, 0}(\alpha)-cA(\alpha)$ where $c$ is such that $\Delta_0(\alpha_0)=0$. \xin{By Lemma~\ref{lem:YW}, $\Delta_0(\alpha - \chi) = Y_0(\chi, \alpha) \Delta_0(\alpha + \chi)$ for $\chi\in \{\frac{\gamma}{2},\frac{2}{\gamma}\}$ and  $\alpha$ in a complex neighborhood of $(-\frac{4}{\gamma},2Q)$.}
Because the interval $(-\frac{4}{\gamma}, 2Q)$ has length bigger than $\gamma$, the function $\Delta_0(\alpha)$ admits a meromorphic extension to a complex neighborhood $U$ of $\RR$, which we still denote by $\Delta_0(\alpha)$, such that $\Delta_0(\alpha - \frac{\gamma}{2}) = \xin{Y_0(\frac{\gamma}{2}, \alpha)} \Delta_0(\alpha + \frac{\gamma}{2})$ for each $\alpha\in U$. \xin{By the meromorphicity of $\Delta_0(\alpha)$, we also have $\Delta_0(\alpha - \frac{2}{\gamma}) = Y_0(\frac{2}{\gamma}, \alpha) \Delta_0(\alpha + \frac{2}{\gamma})$ for each $\alpha\in U$. For  $\alpha\in \alpha_0 + \ZZ \gamma + \ZZ \frac{4}{\gamma}$, since $\Delta_0(\alpha_0)=0$ and $Y_0(\chi, \alpha)\neq 0$,  we have  $\Delta_0(\alpha) = 0$.} Note that $\alpha_0 + \ZZ \gamma + \ZZ \frac{4}{\gamma}$ is dense in $\RR$ when $\gamma^2 \notin \QQ$.  Since $\Delta_0$ is meromorphic on $U$, we must have $\Delta_0(\alpha) = 0$ for all $\alpha\in U$.
Now the continuity in $\gamma$ implies that $\cA_{\gamma, P, 0}(\alpha)-cA(\alpha)=\Delta_0(\alpha) = 0$ for all $\gamma \in (0, 2)$.
Since $\cA_{\gamma, P, 0}(0)=A(0)$ by direct computation, we must have $c=1$, meaning that  $\cA_{\gamma, P, 0}(\alpha) = A(\alpha)$
for all $\gamma\in(0,2)$  as desired. 
\end{proof}

\begin{proof}[Proof of Proposition~\ref{prop:shift-unique}]
Define $\Delta_n(\alpha) := X^1_n(\alpha) - X^2_n(\alpha)$.
Subtracting the given equations for $i = 1, 2$, we obtain for
$\chi \in \{\frac{\gamma}{2}, \frac{2}{\gamma}\}$ and $\alpha\in(-\frac{4}{\gamma}+\chi, 2Q - \chi)$
that $\Delta_n(\alpha - \chi) = Y_n(\chi, \alpha) \Delta_n(\alpha + \chi)$. 
Since $\Gamma$ has poles at $\{0,-1,-2,...\}$ and no zeros, for $P\in\RR$ and $n\in\NN$ the explicit expression
\begin{multline*}
Y_n(\chi, \alpha) = \frac{\Gamma(\frac{1}{2} - \frac{1}{2} l_\chi + \ii \frac{\chi}{2} \sqrt{P^2 + 2n}) \Gamma(\frac{1}{2} - \frac{1}{2} l_\chi - \ii \frac{\chi}{2} \sqrt{P^2 + 2n})}{\Gamma(1 + \frac{1}{2} l_\chi + \ii \frac{\chi}{2} \sqrt{P^2 + 2n}) \Gamma(1 + \frac{1}{2} l_\chi - \ii \frac{\chi}{2} \sqrt{P^2 + 2n})}\\ \frac{\Gamma(1 + \frac{1}{2} l_\chi + \ii \frac{\chi}{2} P) \Gamma(1 + \frac{1}{2} l_\chi - \ii \frac{\chi}{2} P)}{\Gamma(\frac{1}{2} - \frac{1}{2} l_\chi + \ii \frac{\chi}{2} P) \Gamma(\frac{1}{2} - \frac{1}{2} l_\chi - \ii \frac{\chi}{2} P)}
\end{multline*}
yields that $Y_n(\chi, \alpha)$ is meromorphic in $\alpha\in\CC$ without real zeros or poles. \xin{Since $\Delta_n(\alpha_0)=0$ for some $\alpha_0$,  the same argument as in the proof of  Proposition~\ref{prop:DenominatorBlocks} gives that $\Delta_n=0$ for irrational $\gamma^2$  as desired.}
\end{proof}

\subsection{The Dotsenko-Fateev integral}\label{sec:pre-Z}
The starting point for proving  Theorem \ref{thm:int-RecursionThm} is the following  integral representation of $\cA^q_{\gamma, P}(\alpha)$ when  $-\frac{\alpha}{\gamma}\in \NN$. 
\begin{lemma}\label{lem:DF}
	If $N = -\frac{\alpha}{\gamma}\in \NN$ for  $\gamma\in(0,2)$, $\alpha \in (-\frac{4}{\gamma}, Q)$,and
	$q \in (0, 1)$, the function $\cA^q_{\gamma, P}(\alpha)$ is given by
	\begin{multline} \label{eq:df-int}
	\cA^q_{\gamma, P}(\alpha) := q^{\frac{\alpha^2}{24} - \frac{\alpha}{12} Q + \frac{1}{6}}
	\eta(q)^{\frac{5}{4} \alpha \gamma + \frac{2\alpha}{\gamma} - \frac{5}{4} \alpha^2 - 2}\\
	\left(\int_0^1\right)^N \prod_{1 \leq i < j \leq N} |\Theta_\tau(x_i - x_j)|^{-\frac{\gamma^2}{2}}
	\prod_{i = 1}^N \Theta_\tau(x_i)^{-\frac{\alpha\gamma}{2}} e^{\pi \gamma P x_i} \prod_{i = 1}^N dx_i.
	\end{multline}
\end{lemma}
\begin{proof}
	Since $-\frac{\alpha}{\gamma}=N\in\NN$, \revise{by Theorem \ref{thm:Girsanov}} we can write
	\begin{multline*}
	\EE\Big[\Big(\int_0^1 e^{\frac{\gamma}{2} Y_\tau(x)} \Theta_\tau(x)^{-\frac{\alpha\gamma}{2}} e^{\pi \gamma P x} dx \Big)^{-\frac{\alpha}{\gamma}} \Big]\\
	= \left(\int_0^1\right)^N  \prod_{i = 1}^N \Theta_\tau(x_i)^{-\frac{\alpha\gamma}{2}} e^{\pi \gamma P x_i} \prod_{1\le i< j\le N} e^{\frac{\gamma^2}{4} \mathbb{E}[Y_\tau(x_i) Y_\tau(x_j)]} \prod_{i = 1}^N dx_i.
	\end{multline*}
	The explicit formula for $\EE[Y_\tau(x_i)Y_\tau(x_j)]$ in~\eqref{eq:cov2} now gives Lemma~\ref{lem:DF}.
\end{proof}

\xin{We now apply a \emph{rotation-of-contour} trick  to the integral~\eqref{eq:df-int} and use the quasi-periodicity	of the theta function to relate the values of $\cA^q_{\gamma, P}(\alpha)$ at $P_{m, n}$ and $P_{m - 2, n}$. This is crucial to the proof of Theorem \ref{thm:int-RecursionThm}.}
\begin{prop} \label{prop:ZRecId}
\xin{Recall  $P_{m, n}= \frac{2 \ii n}{\gamma} + \frac{\ii \gamma m}{2}$ from~\eqref{eq:pmn-def}.}
If $N = -\frac{\alpha}{\gamma} \in \NN$, for $\gamma\in(0,2)$ and $\alpha \in (-\frac{4}{\gamma}, Q)$,
viewing $P\mapsto\cA^q_{\gamma, P}(\alpha)$ as an entire function,  
\begin{equation}\label{eq:Id2}
\cA^q_{\gamma, P_{m, n}}(\alpha)
= q^{2n + (m - 1) \frac{\gamma^2}{2}} e^{-\frac{\ii \pi \alpha \gamma}{2}} \cA^q_{\gamma, P_{m - 2, n}}(\alpha) \qquad \textrm{for } (m,n) \in \ZZ^2.
\end{equation}
\end{prop}
\begin{proof} 
Consider the domain \(D=\{x+\tau y: x\in(0,1) \textrm{ or } y\in (0,1), \revise{x, y \in \RR} \}\)
from Definition~\ref{def:log} and define on $D$ the functions
\begin{align*}
g_{m, n}(u) &:= q^{\frac{\alpha^2}{24} - \frac{\alpha}{12} Q + \frac{1}{6}}
\eta(q)^{\frac{5}{4} \alpha \gamma + \frac{2\alpha}{\gamma} - \frac{5}{4} \alpha^2 - 2}\\
&\phantom{====}
\Theta_\tau(u)^{-\frac{\alpha \gamma}{2} + \frac{m\gamma^2}{4} + \frac{\alpha \gamma}{4}}
\Theta_\tau(1 - u)^{-\frac{m\gamma^2}{4} - \frac{\alpha\gamma}{4}} e^{\pi \gamma P_{m, n} u};\\
f(P, u) &:= \left(\int_0^1\right)^{N - 1} \prod_{1 \leq i < j \leq N - 1} |\Theta_\tau(x_i - x_j)|^{-\frac{\gamma^2}{2}}\\
&\phantom{====}
\prod_{i = 1}^{N - 1}  \Theta_\tau(x_i - u)^{-\frac{\gamma^2}{4}} \Theta_\tau(u - x_i)^{-\frac{\gamma^2}{4}} \Theta_\tau(x_i)^{-\frac{\alpha\gamma}{2}} e^{\pi \gamma P x_i} \prod_{i = 1}^{N - 1} dx_i,
\end{align*}
where we interpret fractional powers of $\Theta_{\tau}$ via Appendix~\ref{subsec:fractional}.

By the  integral expression (\ref{eq:df-int}) \revise{and (\ref{eq:phase})}, for $N\in \NN$ 
\begin{equation} \label{eq:ca-exp}
\cA^q_{\gamma, P_{m, n}}(\alpha) = e^{\ii \pi(N - 1) \frac{\gamma^2}{4}} \int_0^1 g_{m, n}(u) f(P_{m, n}, u) du.
\end{equation}
Let the fundamental domain $\TT_0$ be the parallelogram bounded by $0, 1, \tau, 1 + \tau$.
We see that $f(P, u)$ and $g_{m, n}(u)$ are holomorphic in $u$ on the interior of $\TT_0$, so integrating along
a contour limiting to the boundary of $\TT_0$, we find 
\begin{multline} \label{eq:int-identity}
\int_0^1 g_{m, n}(u) f(P, u) du + \int_1^{1 + \tau} g_{m, n}(u) f(P, u) du\\
- \int_0^\tau g_{m, n}(u) f(P, u) du - \int_\tau^{1 + \tau} g_{m, n}(u) f(P, u) = 0.
\end{multline}
By~\eqref{eq:theta-shift0},  we have 
$f(P,u+1)=f(P,u)$ if $u\in \{ x+y\tau: x \in \RR, y\in (0,1)\}$.
Moreover, since $\pi \gamma P_{m, n} - \frac{\ii \pi m \gamma^2}{2} = 2 \pi \ii n$, we find
\[
g_{m, n}(u + 1)
= e^{\pi \gamma P_{m, n} + \ii \pi(\frac{\alpha \gamma}{2} - \frac{m\gamma^2}{2} - \frac{\alpha \gamma}{2})} g_{m, n}(u)= g_{m, n}(u),   \textrm{ if } u\in \{ x+y\tau: y\in (0,1)\}.
\]
Thus $\int_0^\tau g_{m, n}(u) f(P, u) du = \int_1^{1 + \tau} g_{m, n}(u) f(P, u) du$,
and (\ref{eq:int-identity}) implies
\[
\int_0^1 g_{m, n}(u) f(P, u) du
= \int_\tau^{1 + \tau} g_{m, n}(u) f(P, u) du
= \int_0^1 g_{m, n}(u + \tau) f(P, u + \tau) du.
\]
By computing with~\eqref{eq:theta-shift1}, we find for $u\in \{ x+y\tau: x\in (0,1), y \in \RR\}$ that
\begin{align*}
g_{m, n}(u + \tau)
&= e^{\pi P_{m, n} \gamma \tau} e^{\ii \pi \alpha \gamma (u - \frac{1}{2} + \frac{\tau}{2})} g_{m, n}(u),\\
f(P, u + \tau) &= e^{(N - 1)(\ii \pi \gamma^2 u + \frac{\ii \pi \gamma^2 \tau}{2})} f(P - \ii \gamma, u).
\end{align*}
Combining these, we find that if  $u\in \{ x+y\tau: x\in (0,1), y \in \RR\}$, then
\begin{align*}
g_{m, n}&(u + \tau) f(P_{m, n}, u + \tau)\\
&= e^{\pi P_{m, n} \gamma \tau} e^{\ii \pi \alpha \gamma (u - \frac{1}{2} + \frac{\tau}{2})} 
e^{(N - 1)(\ii \pi \gamma^2 u + \frac{\ii \pi \gamma^2 \tau}{2})} g_{m, n}(u) f(P_{m, n} - \ii \gamma, u)\\
&= q^{2 n + (m - 1) \frac{\gamma^2}{2}} e^{- \frac{\ii \pi \alpha \gamma}{2}} g_{m - 2, n}(u) f(P_{m - 2, n}, u).
\end{align*}
\revise{In the last line, we have used the fact that $\Theta_{\tau}(u)= \Theta_{\tau}(1-u)$ for any $u\in \{ x+y\tau: x\in (0,1), y \in \RR\}$.} Integrating both sides over $[0, 1]$  and recalling (\ref{eq:ca-exp}), we obtain~\eqref{eq:Id2}.
\end{proof}
\noindent We will use the following corollary of Proposition~\ref{prop:ZRecId} to prove Theorem \ref{thm:int-RecursionThm}. 

\begin{prop}\label{corr:many-shifts}
	\xin{If $\gamma\in(0,2)$, $\alpha \in (-\frac{4}{\gamma}, Q)$,   $q\in (0,1)$ and $-\frac{\alpha}{\gamma}\in \NN$, then}
	\begin{equation}\label{eq:Id3}
	\cA^q_{\gamma, P_{m, n}}(\alpha) = q^{2nm} e^{-\frac{\ii \pi \alpha \gamma m}{2}} \cA^q_{\gamma, P_{-m, n}}(\alpha) \quad\textrm{for } (m, n)\in  \ZZ^2. 
	\end{equation}
\end{prop}  
\begin{proof} 
\xin{This follows from Proposition~\ref{prop:ZRecId}  and the fact that
\begin{equation*}
\cA^q_{\gamma, P_{m, n}}(\alpha) =\prod_{k=1}^{m} \frac{\cA^q_{\gamma, P_{m-2k+2, n}}(\alpha)}{\cA^q_{\gamma, P_{m - 2k, n}}(\alpha)}
\cA^q_{\gamma, P_{-m, n}}(\alpha). \qedhere
\end{equation*}}
\end{proof}

\subsection{Proof of Theorem \ref{thm:int-RecursionThm}}\label{sec:int-Z}
\xin{By Lemma~\ref{lem:DF},  when $-\frac{\alpha}{\gamma}\in\NN$, the function $P\mapsto\cA^q_{\gamma, P}(\alpha)$ can be analytically extended to $P \in \CC$ via the  
integral (\ref{eq:df-int}). Therefore
$P\to \wcA^q_{\gamma, P}(\alpha)=\frac{\cA^q_{\gamma, P}(\alpha)}{\cA_{\gamma, P, 0}(\alpha)}$ meromorphically extends to $P\in\CC$ as well. To establish Theorem~\ref{thm:int-RecursionThm}, it remains to prove~\eqref{eq:int-ZRecur}.}
We now state four lemmas on the pole structure of $\wcA^q_{\gamma, P}(\alpha)$
and prove ~\eqref{eq:int-ZRecur} with these lemmas.  
We then devote most of the subsection
to the proof of  these lemmas. \xin{In all four lemmas we assume
$\gamma\in(0,2)$, $\alpha\in (-\frac{4}{\gamma},0)$, $q\in (0,1)$, and $N= -\frac{\alpha}{\gamma}\in \NN$. Moreover, we recall $P_{m,n}=\frac{2\ii n }{\gamma}+\frac{\ii m\gamma}{2}$ from~\eqref{eq:pmn-def} and $R_{m,n}$ from~\eqref{eq:rmn-def}.}
\begin{lemma}\label{lem:even}
We have $\wcA^q_{\gamma, P}(\alpha)=\wcA^q_{\gamma, -P}(\alpha)$
for all $P\in \CC$.   
\end{lemma}
\begin{lemma}\label{lem:residue}
The function $P\mapsto \cA_{\gamma, P, 0}(\alpha)^{-1}$ is meromorphic
with poles only at $P = \pm P_{m, n}$ for $n \in \NN$ and $1 \leq m \leq N$.  Moreover, the poles are simple.
Finally, the residues   $\Res_{P = P_{m,n}}\wcA^q_{\gamma,P}(\alpha)$ \xin{of $\wcA^q_{\gamma,P}(\alpha)$ satisfy}
\begin{equation}\label{eq:resPmn}
\Res_{P = P_{m,n}}\wcA^q_{\gamma,P}(\alpha)
= q^{2nm} \frac{1}{2 P_{m, n}} R_{\gamma, m, n}(\alpha) \wcA^q_{\gamma,P_{-m,n}}(\alpha).
\end{equation}
\end{lemma}
\begin{lemma}\label{lem:ration}
The function $\wcA_{\gamma, P, n}(\alpha)$ is rational in $P$ for each $n \in \NN$.
\end{lemma}
\begin{lemma} \label{lem:block-p-lim}
\xin{Define $a_k$ by $[q^{-\frac{1}{12}} \eta(q)]^{\alpha (Q - \frac{\alpha}{2}) - 2}=\sum_{k=0}^{\infty} a_kq^k$ near $q=0$.
Then for $k\ge 0$ the limit $\lim\limits_{\RR \ni P\to -\infty} \wcA_{\gamma, P,k}(\alpha)$ exists and equals $a_k$}.
\end{lemma}
\begin{proof}[Proof of  Theorem \ref{thm:int-RecursionThm} based on Lemmas~\ref{lem:even}--\ref{lem:block-p-lim}]
By Lemma~\ref{lem:ration},
$\wcA^q_{\gamma, P, k}(\alpha)$ is a rational function in $P$.  By Lemma \ref{lem:residue},
its poles are located at $P = \pm P_{m, n}$. By Lemmas~\ref{lem:even} and~\ref{lem:block-p-lim}, \xin{with $a_k$ defined in Lemma~\ref{lem:block-p-lim}} we have 
\begin{align*} 
\wcA_{\gamma,P,k}(\alpha) = \sum_{n,m = 1}^{\infty} \frac{2P_{m,n} \Res_{P=P_{m,n}}\wcA_{\gamma,P,k}(\alpha)}{P^2-P^2_{m,n}} + a_k,
\end{align*}
where the series has finitely many non-zero summands.

By Cauchy's theorem around $P=P_{m,n}$, we have
$\Res_{P = P_{m,n}}\wcA^q_{\gamma,P}(\alpha)=\sum_{k=0}^{\infty}\Res_{P=P_{m,n}}\wcA_{\gamma,P,k}(\alpha)q^k$ as
convergent series in $q$. \xin{By~\eqref{eq:resPmn} we have $\Res_{P=P_{m,n}}\wcA_{\gamma,P,k}(\alpha)=0$ for  $k<2mn$, and moreover,}
\begin{equation*} 
\wcA_{\gamma,P,k}(\alpha) =  \sum_{n,m \in\NN, 2mn\le k} \frac{R_{\gamma, m, n}(\alpha)}{P^{2} - P^2_{m,n}}\wcA_{\gamma, P_{-m, n}, k-2mn}(\alpha)+a_k.
\end{equation*}
Therefore  the $q^k$-coefficients of both sides of \eqref{eq:int-ZRecur} are equal, namely 
\eqref{eq:int-ZRecur} holds as an equality of formal $q$-power series, yielding Theorem~\ref{thm:int-RecursionThm}.
\end{proof}

In the rest of this subsection, we prove Lemmas~\ref{lem:even}--\ref{lem:block-p-lim} in order.

\begin{proof}[Proof of Lemma~\ref{lem:even}]
For $P\in \RR$, we have
\begin{multline*}
\EE\Big[\Big(\int_0^1 e^{\frac{\gamma}{2} Y_\tau(x)} \Theta_\tau(x)^{-\frac{\alpha\gamma}{2}} e^{\pi \gamma P x} dx \Big)^{-\frac{\alpha}{\gamma}} \Big] \\
= \EE\Big[\Big(\int_0^1 e^{\frac{\gamma}{2} Y_\tau(1-x)} \Theta_\tau(1-x)^{-\frac{\alpha\gamma}{2}} e^{\pi \gamma P(1- x)} dx \Big)^{-\frac{\alpha}{\gamma}}\Big].
\end{multline*}
Since $\{Y_\tau(1-x)\}_{0\le x\le 1}$ and $\{Y_\tau(x)\}_{0\le x\le 1}$ are equal in law and $\Theta_\tau(1-x)=\Theta_\tau(x)$, for $P\in\RR$ we have
$\cA^q_{\gamma, P}(\alpha)= e^{-\pi P \alpha} \cA^q_{\gamma, -P}(\alpha)$.  
\xin{Sending $q\to 0$} we have $\cA_{\gamma, P,0}(\alpha)= e^{-\pi P \alpha} \cA_{\gamma, -P,0}(\alpha)$. Therefore $\wcA^q_{\gamma, P}(\alpha)=\wcA^q_{\gamma, -P}(\alpha)$ for $P\in\RR$.
By meromorphicity in $P$ of $\cA^q_{\gamma, P}(\alpha)$  for $-\frac{\alpha}{\gamma}=N\in\NN$, the same holds for all $P \in \CC$.
\end{proof}
\xin{Lemma~\ref{lem:residue} follows from a straightforward (but lengthy) calculation based on the explicit formula of $\cA_{\gamma, P, 0}(\alpha)$ from Proposition \ref{prop:DenominatorBlocks}.}
\begin{proof}[Proof of Lemma~\ref{lem:residue}]
Specializing Proposition \ref{prop:DenominatorBlocks} to $N= -\frac{\alpha}{\gamma}\in \NN$ and using \xin{the shift equation \eqref{eq:sh_double_gamma} for $\Gamma_{\frac{\gamma}{2}}$,} we get
\begin{equation}\label{corr:int-denom}
\cA_{\gamma, P, 0}(\alpha) = e^{\frac{\ii \pi \gamma^2 N^2}{2}} \frac{e^{\frac{\pi \gamma P N}{2}}}{\Gamma(1 - \frac{\gamma^2}{4})^N} \prod_{j = 1}^N \frac{\Gamma(1 - \frac{j \gamma^2}{4})\Gamma(1 + \frac{(2N - j +1)\gamma^2}{4})}{\Gamma(1 + \frac{j \gamma^2}{4} + \frac{\ii \gamma P}{2})\Gamma(1 + \frac{j \gamma^2}{4} - \frac{\ii \gamma P}{2})}.
\end{equation}
Recall $\Gamma(z)$ has simple poles at $\{0,-1,-2,\ldots\}$ and no zeros.
Thus \eqref{corr:int-denom} implies $\cA_{\gamma, P, 0}(\alpha)$ has
simple zeros at $P = \pm P_{m, n}$ for $n \in \NN$ and $1 \leq m \leq N$.  
\xin{This yields the pole structure of $\cA_{\gamma, P, 0}(\alpha)^{-1}$ asserted in Lemma~\ref{lem:residue}.}
We now compute the residues of $\cA_{\gamma, P, 0}(\alpha)^{-1}$.  Define 
\begin{equation}\label{eq:def-fP}
f(P) := \prod_{j = 1}^N \Gamma(1 +\frac{j\gamma^2}{4}  + \frac{\ii \gamma P}{2})\Gamma(1 +\frac{j \gamma^2}{4}  - \frac{\ii \gamma P}{2}) 
\end{equation}
so that $\cA_{\gamma, P, 0}(\alpha)^{-1}=C e^{-\frac{\pi \gamma P N}{2}}f(P)$
for some $C := C(\alpha, \gamma)$ independent of $P$.  \xin{Recall $R_{\gamma, m, n}(\alpha)$ from~\eqref{eq:rmn-def}. We claim that}
\begin{equation}\label{eq:f-res}
    \xin{\frac{\Res_{P = P_{m, n}} f(P)}{f(P_{-m, n})} = \frac{1}{2 P_{m, n}} R_{\gamma, m, n}(\alpha).}
\end{equation}By (\ref{eq:gamma-poles}), we have $\Res_{P= P_{m,n}} \Gamma(1 + \frac{m\gamma^2}{4} + \frac{\ii \gamma P}{2}) = \frac{2}{\ii \gamma} \Res_{x = 1 - n} \Gamma(x) = \frac{2}{\ii \gamma} \frac{(-1)^{n-1}}{(n-1)!}$ 
for $n \in \NN$ and $1 \leq m \leq N$. Therefore
\begin{align*}
\Res_{P = P_{m,n}} f(P)
&= \frac{2}{\ii \gamma} \frac{(-1)^{n - 1}}{(n - 1)!} \prod_{j = 1}^N \Gamma(1 +\frac{j\gamma^2}{4}  + m\frac{\gamma^2}{4}+ n)\!\!\!\!
 \prod_{j = 1, j \neq m}^N \Gamma(1 +\frac{j\gamma^2}{4}  - m\frac{\gamma^2}{4}- n).
\end{align*}By $\Gamma(x+n)=\Gamma(x-n)\prod_{l=-n}^{n-1}(x+l)$, we have
\begin{align}
\frac{\Res_{P = P_{m, n}} f(P)}{f(P_{-m, n})} 
&= \frac{2}{\ii \gamma} \frac{(-1)^{n - 1}}{(n - 1)!}
\frac{\prod_{j= - m}^{ m-1} \prod_{l =-n  }^{n-1} ( 1 + \frac{\gamma^2}{4} + \frac{N \gamma^2}{4} +  \frac{j \gamma^2}{4}  +l ) }{n! \prod_{j=1-m, j \neq 0}^{m} \prod_{l =-n}^{n-1} ( 1+\frac{j \gamma^2}{4}  +l ) },\label{eq:Rf}
\end{align}
where we used $\frac{\prod_{j=1}^N  (1 + \frac{j \gamma^2}{4} + \frac{m \gamma^2}{4} +l ) }{\prod_{j=1,j\neq m}^N 
( 1+\frac{j \gamma^2}{4} - \frac{m \gamma^2}{4} +l) }=\frac{\prod_{j= - m}^{m-1}  ( 1 + \frac{\gamma^2}{4} + \frac{N \gamma^2}{4} +  \frac{j \gamma^2}{4}  +l ) }{\prod_{j=1-m,j\neq 0}^{m} ( 1+\frac{j \gamma^2}{4}  +l ) }$.
Note that 
\begin{align}\label{eq:Rnum}
\prod_{j= - m}^{ m-1} \prod_{l =-n  }^{n-1} ( 1 + \frac{\gamma^2}{4} + \frac{N \gamma^2}{4} +  \frac{j \gamma^2}{4}  +l )=(\frac{\gamma}{2})^{4mn}  \prod_{j = -m}^{m - 1} \prod\limits_{l = -n}^{n - 1} (Q - \frac{\alpha}{2} + \frac{j \gamma}{2} + \frac{2l}{\gamma}).
\end{align} 
Since $(-1)^{n-1}(n-1)!n! = \prod_{l =-n,l\neq -1}^{n-1} ( 1+l)$, we have
\begin{equation}\label{eq:fdom}
(-1)^{n-1}(n-1)!n! \prod_{j=1-m, j \neq 0}^{m} \prod_{l =-n}^{n-1} ( 1+\frac{j \gamma^2}{4}  +l)=(\frac{m \gamma^2}{4}+n) \prod\limits_{(j, l) \in S_{m, n}}\revise{ (\frac{j \gamma^2}{4} + l)},
\end{equation}
where we recall $S_{m, n}$ defined below~\eqref{eq:rmn-def}. 
Combining~\eqref{eq:rmn-def},~\eqref{eq:Rf},~\eqref{eq:Rnum}, and~\eqref{eq:fdom} yields~\eqref{eq:f-res}.
Since $e^{-\frac{\pi \gamma (P_{m,n}-P_{-m,n} ) N}{2}} = e^{\frac{\ii \pi \alpha \gamma m}{2}}$ for $N=-\frac{\alpha}{\gamma}$,
\xin{and $\cA_{\gamma, P, 0}(\alpha)^{-1}=C e^{-\frac{\pi \gamma P N}{2}}f(P)$,} for $n \in \NN$ and $1 \leq m \leq N$ we have
\[
\Res_{P = P_{m, n}} \cA_{\gamma, P, 0}(\alpha)^{-1}
= e^{\frac{\ii \pi \alpha \gamma m}{2}} \frac{1}{2 P_{m, n}} R_{\gamma, m, n}(\alpha) \cA_{\gamma, P_{-m, n}, 0}(\alpha)^{-1}.
\]
Combining this with  $\cA^q_{\gamma, P_{m, n}}(\alpha) = q^{2nm} e^{-\frac{\ii \pi \alpha \gamma m}{2}} \cA^q_{\gamma, P_{-m, n}}(\alpha) $
from Proposition~\ref{corr:many-shifts}, we obtain~\eqref{eq:resPmn}.
\end{proof}

\xin{For Lemma~\ref{lem:ration}, it is easy to see that $P\mapsto\wcA_{\gamma,P,k}(\alpha)$ has finitely many poles, the main effort is to obtain the polynomial growth of $\wcA_{\gamma,P,k}(\alpha)$ at $\infty$ which  yields its rationality. This requires a fine  analysis of $f(P)$ in~\eqref{eq:def-fP}.}
\begin{proof}[Proof of Lemma~\ref{lem:ration}]
\xin{By Cauchy's theorem around $P=P_{m,n}$, we have
$\Res_{P = P_{m,n}}\wcA^q_{\gamma,P}(\alpha)=\sum_{k=0}^{\infty}\Res_{P=P_{m,n}}\wcA_{\gamma,P,k}(\alpha)q^k$ as
convergent series in $q$. By~\eqref{eq:resPmn} we have $\Res_{P=P_{m,n}}\wcA_{\gamma,P,k}(\alpha)=0$ for  $k<2mn$.}
By Lemma~\ref{lem:even}, $\Res_{P = -P_{m,n}}\wcA_{\gamma,P,k}(\alpha)=0$ as well.  
Because all poles are located at $P = \pm P_{m, n}$ for some $m, n$ by Lemma~\ref{lem:residue}, \xin{for each $k$} the meromorphic function $P\mapsto \wcA_{\gamma,P,k}$ has finitely many poles.  
Let 
\[
r\defeq \min\{|P_{m,n}-P_{m',n'}| :1\le  m,m' \le N, \quad n,n'\in\NN, \quad \textrm{ and  }(m,n)\neq (m',n')\},
\] 
which is positive. For $1\leq m\leq N$ and $n\in \NN$, let $\mathcal{B^+}_{m,n}$ (resp. $\mathcal B^{-}_{m,n}$) be the ball around $P_{m,n}$ (resp. $-P_{m,n}$) with radius
$\frac{r}{3}$, so that $\mathcal{B}_{m,n}^\pm\cap \mathcal{B}_{m',n'}^\pm= \emptyset$ for $(m,n)\neq (m',n')$.
Define $\CC^{\circ} := \CC\setminus \cup_{1\leq m\leq N,n\geq 1} (\mathcal{B^+}_{m,n}\cup \mathcal B^-_{m,n})$.  
Recall that \(\mathcal{A}_{\gamma, P,0}(\alpha)^{-1}=C(\alpha,\gamma)e^{-\pi N\gamma P/2}f(P)\)  
for $f(P)$ in~\eqref{eq:def-fP} and some explicit $C(\alpha,\gamma)$ not depending on $P$.  We claim that there exists $K\in \NN$ with
\begin{align}\label{eq:GammaBd}
M\defeq \sup_{P \in \CC^{\circ}}|P|^{-K}e^{\pi N\gamma |\Re(P)|/2} |f(P)|<\infty.
\end{align}
We first prove Lemma~\ref{lem:ration} given~\eqref{eq:GammaBd}. For $\Re P\le 0$ and $P\in \CC^\circ$, we have
\[|P|^{-K}|\mathcal{A}_{\gamma, P,0}(\alpha)^{-1}|=C(\alpha,\gamma)|P|^{-K}e^{-\pi N\gamma \mathrm{Re}(P)/2}|f(P)|\le MC(\alpha,\gamma). \]
On the other hand, since $|e^{\pi \gamma Px}|\le 1$ for $\Re P\le 0$, we have $|\mathcal{A}_{\gamma,P}^q(\alpha)|<\infty$, 
where the bound only depends on $\alpha,\gamma, |q|$ and is uniform in $\Re P\le 0$. Applying Cauchy's theorem in
$q$ to extract $q$-series coefficients, for $k\in \NN$,  we get $C_k(\alpha,\gamma)\defeq \sup_{\Re P\le 0}|\mathcal{A}_{\gamma,P,k}(\alpha)|<\infty$.
Since $\wcA_{\gamma,P,k}(\alpha)=\wcA_{\gamma,-P,k}(\alpha)$, we get $ \sup_{P\in \CC^{\circ}} |P|^{-K}|\wcA_{\gamma,P,k}(\alpha)|<C_k$ with $C_k=MC(\alpha,\gamma)C_k(\alpha,\gamma)$.
Since it has finitely many poles, $P^{-K}\wcA_{\gamma,P,k}(\alpha)$ is analytic for large enough $|P|$. By the
maximal modulus theorem, $|P^{-K}\wcA_{\gamma,P,k}(\alpha)|\le C_k$ for large enough $|P|$. 
Thus $P^{-K}\wcA_{\gamma,P,k}(\alpha)$ is a rational function and so is $\wcA_{\gamma,P,k}(\alpha)$.

It remains to prove~\eqref{eq:GammaBd}.  Note that $\Gamma(1 +\frac{j\gamma^2}{4}  + \frac{\ii \gamma P}{2})=\pi \sin(\pi (1 +\frac{j\gamma^2}{4}  + \frac{\ii \gamma P}{2}))^{-1} \Gamma(-\frac{j\gamma^2}{4} - \frac{\ii \gamma P}{2})^{-1}$ by~\eqref{eq:reflection-sin}. 
Moreover, by our choice of $\CC^\circ$ we have $\max_{P\in \CC^\circ}e^{\pi \gamma |\Re( P)|/2} |\sin(\pi (1 +\frac{j\gamma^2}{4}  + \frac{\ii \gamma P}{2}))^{-1}|<\infty$.
Since $f(P)=f(-P)$, 
it suffices to show for $1\le j\le N$ there exist some $C_j>0$ and $K_j\in \NN$ with
\begin{equation}\label{eq:ratio-Gamma}
\left|\Gamma(1 +\frac{j \gamma^2}{4}  - \frac{\ii \gamma P}{2}) \Gamma(-\frac{j\gamma^2}{4} -\frac{\ii \gamma P}{2})^{-1}\right| \le C_j \xin{(|P|^{K_j} \vee 1)}  \quad \textrm{ for }\Im P\ge 0.
\end{equation}
Because $\frac{\Gamma(1 +\frac{j \gamma^2}{4}  - \frac{\ii \gamma P}{2}) }{\Gamma(-\frac{j\gamma^2}{4} -\frac{\ii \gamma P}{2})}$ is analytic
in $P$ for $\Im P \geq 0$, it suffices to check that it is polynomially bounded for $|P|$ large.
By Stirling's approximation~\eqref{eq:stiring}, $\Gamma(z) \sim \sqrt{\frac{2\pi}{z}}e^{-z}z^z(1+O(|z|^{-1}))$ as $|z| \to \infty$ with $\Re z \ge  0$. 
For $\Im P\ge 0$, this applies to $\Gamma(1 +\frac{j \gamma^2}{4}  - \frac{\ii \gamma P}{2})$ and $\Gamma(-\frac{j\gamma^2}{4} -\frac{\ii \gamma P}{2})$ as $|P|$ grows large. Therefore~\eqref{eq:ratio-Gamma} holds if $K_j>1 +\frac{j \gamma^2}{4}-(-\frac{j \gamma^2}{4})=  \frac{j\gamma^2}{2}+1$, which implies (\ref{eq:GammaBd}).
\end{proof}
\xin{We now prove Lemma~\ref{lem:block-p-lim} using the expression of $\cA^q_{\gamma, P}(\alpha)$ from Lemma~\ref{lem:reduce}.}
\begin{proof}[Proof of Lemma~\ref{lem:block-p-lim}]
We assume that $P<0$ and $N=\frac{-\alpha}{\gamma}\in\NN$. \xin{Recall  that
$\cA^q_{\gamma, P}(\alpha)=[q^{-\frac1{12}}\eta(q)]^{\alpha (Q-\frac{\alpha}{2}) -2}e^{\ii\alpha^2\pi/2} \hat\cA^q_{\gamma, P}(\alpha)$ from Lemma~\ref{lem:reduce} where}
\begin{equation}\label{eq:hatA-int}
\hat\cA^q_{\gamma, P}(\alpha)= \mathbb{E} \Big[\Big( \int_0^1 e^{\frac{\gamma}{2}(F_\tau(x) -F_\tau(0)) } (2\sin(\pi x))^{-\alpha\gamma/2}  e^{ \pi \gamma P x } e^{\frac{\gamma}{2} Y_{\infty}(x)}dx  \Big)^{- \frac{\alpha}{\gamma}} \Big].
\end{equation}
Under our assumption that $-\frac{\alpha}{\gamma}\in\NN$, the right side of~\eqref{eq:hatA-int} is analytic in $q$ for \xin{$|q|<1$}. 
We claim that \xin{for a fixed  $r\in(0,1)$},
\begin{equation}\label{eq:P-lim}
\lim_{\RR \ni P\to-\infty}   \frac{\hat\cA^q_{\gamma,P}(\alpha)}{\hat\cA^0_{\gamma,P}(\alpha)}=1 \textrm{ uniformly for } q\in \CC \textrm{ with }|q|=r.  
\end{equation}
\xin{Since \(\wcA^q_{\gamma,P}(\alpha)=A^q_{\gamma,P}(\alpha)A^0_{\gamma,P}(\alpha)^{-1}=(q^{-\frac1{12}}\eta(q))^{\alpha (Q-\frac{\alpha}{2}) - 2} \hat\cA^q_{\gamma,P}(\alpha) \hat\cA^0_{\gamma,P}(\alpha)^{-1}\),  we get Lemma~\ref{lem:block-p-lim} from~\eqref{eq:P-lim}.}

To prove (\ref{eq:P-lim}), we introduce the (complex) measure
\[
\mu_q(dx):=e^{\frac{\gamma}{2}(F_\tau(x) -F_\tau(0)) } (2\sin(\pi x))^{-\alpha\gamma/2}  e^{ \pi \gamma P x } e^{\frac{\gamma}{2} Y_{\infty}(x)}dx 
\]
so that  $\hat\cA^q_{\gamma, P}(\alpha)=\E[\mu_q([0,1])^N]$ for \xin{$|q|<1$}. 
For $\eps \in (0,1)$, we have
\begin{equation} \label{eq:binomial-exp}
\E[\mu_q([0,1])^N]=\sum_{i=0}^{N}\binom{N}{i} \EE[\mu_q([0,\eps])^i\mu_q([\eps,1 ])^{N-i}].
\end{equation}
For $I\subset  [0,1]$, let $M_r(I):=\sup_{x\in I, |q|=r} e^{\frac{\gamma}{2}|F_\tau(x) -F_\tau(0)| }$. Then \(|\mu_q(I)|\le M_r(I) \mu_0(I).\)
Holder's inequality and the independence of $F_\tau$ and $Y_\infty$ imply
\begin{multline} \label{eq:holder-ratio}
\left|\EE[\mu_q([0,\eps])^i \mu_q([\eps,1 ])^{N-i}]\right| \\ 
\le \EE[M_r([0,1])^N]\EE[\mu_0([0,\eps])^{N}] \cdot \Big( \frac{\EE[\mu_0([\eps,1])^N]}{\EE[\mu_0([0,\eps])^N]} \Big)^{1-\frac{i}{N}}.
\end{multline}
For $P<0$, note that
\[
\EE[\mu_0([\eps,1])^N]\le e^{ -\pi \gamma N |P| \eps } \EE\Big[\Big(\int_{\eps}^{1}(2\sin(\pi x))^{-\alpha\gamma/2}  e^{\frac{\gamma}{2} Y_{\infty}(x)}dx\Big)^N\Big]
\]
and 
\[
\EE[\mu_0([0,\eps])^N]\ge  e^{ -\pi \gamma N |P| \eps/2 }
\EE\Big[\Big(\int_{0}^{\eps/2}(2\sin(\pi x))^{-\alpha\gamma/2}  e^{\frac{\gamma}{2} Y_{\infty}(x)}dx\Big)^N\Big].
\]
Since $\EE[M_r([0,1])^N]<\infty$ and $\EE[\mu_0([0,1])^{N}]<\infty$, by~\eqref{eq:holder-ratio}
there exists a constant $C=C(\eps,r)$ such that for $i<N$, $|q|=r$, and $P < 0$, we have
\begin{equation} \label{eq:ratio-bound}
\left|\EE[\mu_q([0,\eps])^i\mu_q([\eps,1 ])^{N-i}]\right|\le \E[\mu_0([0,1])^N]C(\eps,r) e^{-\frac12\pi \eps \gamma|P|}.
\end{equation}
Applying (\ref{eq:binomial-exp}) and (\ref{eq:ratio-bound}), we have for a possibly enlarged $C(\eps, r)$ that
\begin{equation}\label{eq:eps-app}
\E[\mu_q([0,1])^N] = \E[\mu_q([0,\eps])^N]+\E[\mu_0([0,1])^N]C(\eps,r) e^{-\frac12\pi \eps \gamma|P|}
\end{equation}
 for $|q|=r$ and $P<0$. A similar argument shows that for some $C(\eps)>0$ we have $\E[\mu_0([0,1])^N] = \E[\mu_0([0,\eps])^N] (1+C(\eps) e^{-\frac12\pi \eps \gamma|P|})$.

Set $m_\eps:=\sup_{x\in I, |q|=r} |e^{\frac{\gamma}{2}(F_\tau(x) -F_\tau(0)) }-1 |$ so that
\(\left|\mu_q([0,\eps]) - \mu_0([0,\eps]) \right| \le m_\eps \mu_0([0,\eps])\). Then
\begin{align*}
\left|\mu_q([0,\eps])^N - \mu_0([0,\eps])^N \right| &\le \left|\mu_q([0,\eps]) - \mu_0([0,\eps]) \right| \times N M_r([0,1])^N \mu_0([0,\eps])^{N-1}\\
&\le m_\eps  \times N M_r([0,1])^N \mu_0([0,\eps])^{N}.
\end{align*} 
By the dominated convergence and continuity of $F_\tau(x)$ at $x=0$, we have $\lim_{\eps\to 0}\EE[m_\eps  M_r([0,1])^N]=0$,
hence $\lim_{\eps \to 0}\frac{\EE[\mu_q([0,\eps])^N]}{\EE[\mu_0([0,\eps])^N]}=1$ uniformly in $|q|=r$ and $P<0$.
Combined with~\eqref{eq:eps-app}, we get~\eqref{eq:P-lim}.
\end{proof}

\subsection{Proof of Theorem \ref{thm:nek-block}} \label{sec:nek-block-proof}
We \xin{need to prove $\widetilde \cA_{\gamma, P, n}(\alpha)=\cZ_{\gamma, P, n}(\alpha)$ for $\cZ_{\gamma, P, n}$ in \eqref{eq:nek-exp-def}. 
Let us outline the proof. 
First,   from explicit expression  we see that as a meromorphic function in $\gamma$ 
\begin{equation}\label{eq:dual}
    \cZ_{\gamma, P, n}(\alpha)=\cZ_{4/\gamma, P, n}(\alpha).
\end{equation}
By  Theorems~\ref{thm:shift} and~\ref{thm:int-nek-block}, for $\chi=\frac{\gamma}{2}$ and $\frac{-\alpha}{\gamma}\in \NN$ we have
\begin{equation} \label{eq:shift-Z}
\xin{\cZ_{2\chi, P, n}(\alpha - \chi)} = Y_n(\chi, \alpha) \xin{\cZ_{2\chi, P, n}(\alpha + \chi)} + Z_n(\chi, \alpha).
\end{equation}
We will show that both sides of~\eqref{eq:shift-Z}  admit meromorphic extension in $\chi$ around $[0,\infty)$ and the validity of~\eqref{eq:shift-Z}  can extend to $\chi\in (0,\infty)$. Therefore 
$\cZ_{\gamma, P, n}(\alpha)$ satisfies the shift equation~\eqref{eq:abs-shift}   for $\chi=\frac{\gamma}{2}$, namely
\begin{equation} \label{eq:shift-Z-1}
\xin{\cZ_{\gamma, P, n}(\alpha -\chi)} = Y_n(\chi, \alpha) \xin{\cZ_{\gamma, P, n}(\alpha +\chi)} + Z_n(\chi, \alpha)
\end{equation}with $\chi=\frac{\gamma}{2}$.
Since $\cZ_{\gamma, P, n} =\cZ_{4/\gamma, P, n}$ by~\eqref{eq:dual},  we see from~\eqref{eq:shift-Z} that
$\cZ_{\gamma, P, n}(\alpha)$ satisfies  \eqref{eq:shift-Z-1} with $\chi=\frac2{\gamma}$ as well.
Now the uniqueness in Proposition~\ref{prop:shift-unique} implies 
	$\cZ_{\gamma, P, n}(\alpha)=\wcA_{\gamma,P,n}(\alpha)$.}

\xin{We use an inductive argument to implement this outline.} We say a function $f(\chi)$
is $\chi$-\emph{good} if it admits meromorphic extension to a complex neighborhood of $[0,\infty)$.
We say a function $f(w,\chi)$ is $(w,\chi)$-\emph{good} if there exist a complex neighborhood $U$ of $[0,\infty)$ and 
\revise{countable subset $V \subset U$} such that 
\xin{\begin{itemize}
\item for each $w\in\overline\D$,
$f(w,\chi)$ is meromorphic in $\chi$ on $U$ with poles in $V$;
\item the 
multiplicity  of the poles are  uniformly bounded in $w\in \overline \D$;
\item  $f$ admits an extension \revise{ to $\overline \D \times \left(U\setminus V\right)$} where $f$ is $(w,\chi)$-regular in the sense of Definition~\ref{def:walpha} with $\chi$ in place of $\alpha$.
\end{itemize}}

For $\chi \in \{\frac{\gamma}{2}, \frac{2}{\gamma}\}$ and $j\in \{1,2\}$ define the normalized expressions
\begin{align}\label{eq:norm-Gphi}
\wG^{\alpha, j}_{\chi, n, 1}(w) &:= G^{\alpha, j}_{\chi, n, \revise{1}}(w) W^-_\chi(\alpha, \gamma)^{-1}\eta_{\chi, 0}^-(\alpha)^{-1}\cA_{\gamma, P, 0}(\alpha - \chi)^{-1}\\
\wphi^{\alpha, j}_{\chi, n, 1}(w) &:= \phi^{\alpha, j}_{\chi, n, 1}(w) W^-_\chi(\alpha, \gamma)^{-1}\eta_{\chi, 0}^-(\alpha)^{-1}\cA_{\gamma, P, 0}(\alpha - \chi)^{-1}
\end{align}\xin{Normalizing~\eqref{eq:phi-decomp},} 
we see that for some $\widetilde{X}^j_{\chi, n}(\alpha)$ not depending on $w$
\begin{equation} \label{eq:phi-decomp1}
\wphi^{\alpha, j}_{\chi, n, 1}(w) = \wG^{\alpha, j}_{\chi, n, 1}(w) + \widetilde{X}^j_{\chi, n}(\alpha) v^\alpha_{j, \chi, n}(w) \textrm{ for } n\in \NN_0 \textrm{ and }j=1,2,
\end{equation}where $v^\alpha_{j, \chi, n}(w)$ are as in~\eqref{eq:defv1} and~\eqref{eq:defv2}.
\xin{Note that $\wV^{\alpha, j}_{\chi, n}= \wG^{\alpha, j}_{\chi, n, 1}(0)$
for  $\wV^{\alpha, j}_{\chi, n}$ defined in Theorem~\ref{thm:shift} and used in the expression~\eqref{eq:zn-def} of $Z_n(\chi,\alpha)$.}

We now inductively prove the following statements (a)$_n$ and (b)$_n$ indexed by $n=0,1,2,\cdots$.
Theorem~\ref{thm:nek-block} then follows from (a)$_n$.
\begin{itemize}
\item [(a)$_n$]:  $\cZ_{\gamma, P, n}(\alpha) = \wcA_{\gamma, P, n}(\alpha)$ for $\alpha\in(-\frac{4}{\gamma}, 2Q)$, $\gamma\in(0,2)$ and $\chi\in \{ \frac{\gamma}{2},\frac{2}{\gamma} \}$. 
\item[(b)$_n$]: $(w,\chi)\mapsto \wphi^{\alpha, j}_{\chi, n, 1}(w)$ is $(w,\chi)$-good for each $\alpha\in \RR$ and $j\in \{1,2\}$.
\end{itemize}
For $n = 0$, since $\cZ_{\gamma, P, 0}(\alpha) = \wcA_{\gamma, P, 0}(\alpha)=1$, (a)$_0$ holds. 
Since $\wG^{\alpha, j}_{\chi, 0, i}(w)=0$, we have
$\widetilde{X}^1_{\chi, n}(\alpha)=\wphi^{\alpha, 1}_{\chi, n, 1}(0)=1$ and 
$\widetilde{X}^2_{\chi, n}(\alpha)=\wphi^{\alpha, 1}_{\chi, n, 2}(0)=-1$. By the expression for $v^\alpha_{j, \chi, n}(w)$ in~\eqref{eq:defv1} and~\eqref{eq:defv2},
we get (b)$_0$. 

For $n\in \NN$, assume (a)$_m$ and (b)$_m$ hold for $0 \leq m \leq n-1$.
 \xin{We first show that $\wG_{\chi, n, 1}^{\alpha, 1}(w)$ and $\wG_{\chi, n,1}^{\alpha, 2}(w)$ are $(w,\chi)$-good.} By Definition~\ref{def:particular} and~\eqref{eq:norm-Gphi},
\[
(\hpg_\chi  - ( l_\chi^2 + \chi^2 (P^2 + 2n))/4 )\, \wG_{\chi, n, 1}^{\alpha, 1}(w)
= \widetilde{g}^{\alpha, 1}_{\chi, n, 1}(w)
\]
where $\widetilde{g}^{\alpha, 1}_{\chi, n, 1}(w)$ is defined as $g^{\alpha, 1}_{\chi, n, 1}(w)$
in~\eqref{eq:defg} with $\wphi_{\chi, n - l, 1}^{\alpha, j}(w)$ in place of $\phi_{\chi, n - l, 1}^{\alpha, j}(w)$. 
Since (b)$_m$ holds for $m<n$, $\widetilde{g}^{\alpha, 1}_{\chi, n, 1}(w)$ is $(w,\chi)$-good.  
Applying Lemma~\ref{lem:alpha-ext2} to $\prod_{z \in V} (\chi - z)^{M} \widetilde{g}^{\alpha, 1}_{\chi, n, 1}(w)$,
where $M$ is a uniform bound on the multiplicity of poles of \xin{$\chi\mapsto \widetilde{g}^{\alpha, 1}_{\chi, n, 1}(w)$,} we see that
$\wG_{\chi, n, 1}^{\alpha, 1}(w)$ is $(w,\chi)$-good.
The same argument shows that $\wG_{\chi, n,1}^{\alpha, 2}(w)$ is $(w,\chi)$-good as well.

We now  prove (a)$_n$. First fix $\alpha < 0$. By our induction hypothesis, for $m<n$,  $\wcA_{\gamma, P, m}(\alpha)=\cZ_{\gamma, P, m}(\alpha)$.
Let $\chi_k=\frac{-\alpha}{2k}$. 
For $k\in\NN$ large enough we have  $ \alpha>-\frac{4}{\gamma_k} + \frac{\gamma_k}{2}$ with $\gamma_k=2\chi_k=-\frac{\alpha}{k}$. 
\xin{For such $k$, Theorems~\ref{thm:shift} and~\ref{thm:int-nek-block} yield that \eqref{eq:shift-Z} holds for $\chi=\chi_k$. Namely for $n\in\NN$ and $\chi=\chi_k$ we have 
\[
\cZ_{2\chi, P, n}(\alpha - \chi) = Y_n(\chi, \alpha) \cZ_{2\chi, P, n}(\alpha + \chi) + Z_n(\chi, \alpha).
\]}

\xin{We now show that both sides of~\eqref{eq:shift-Z} are $\chi$-good for a fixed $\alpha<0$.}
By their explicit expressions, $l_\chi$, $\Gamma_{n, 1}$, $\Gamma_{n, 2}$ \xin{and $Y_n(\chi, \alpha)$}
are meromorphic in $\chi\in \CC$.  Moreover,~\eqref{eq:Theta'0} and the definition
of $\eta_{\chi, n }^{\pm}(\alpha)$ \revise{in (\ref{eq:eta-m})} yield 
\begin{multline*}
\prod_{k = 1}^\infty (1 - q^{2k})^{4\frac{l_\chi(l_\chi + 1)}{\chi^2} + 2l_\chi +2} =\sum_{n=1}^{\infty}\frac{\eta_{\chi, n }^-(\alpha)}{\eta_{\chi, 0}^-(\alpha)} q^n\\ \textrm{ and } 
\prod_{k = 1}^\infty (1 - q^{2k})^{4 \frac{l_\chi(l_\chi + 1)}{\chi^2} - 2 l_\chi} =\sum_{n=1}^{\infty}\frac{\eta_{\chi, n }^+(\alpha)}{\eta_{\chi, 0}^+(\alpha)} q^n ,
\end{multline*}
so $\frac{\eta_{\chi, n }^{\pm}(\alpha)}{\eta_{\chi, 0}^{\pm}(\alpha)} $ are rational functions in $\chi$ for $n\in \NN_0$.
\xin{For $j=1,2$, since $\wV^{\alpha, j}_{\chi, n}=\wG_{\chi, n,1}^{\alpha, j}(0)$ and $\wG_{\chi, n,1}^{\alpha, 2}(w)$ is $(w,\chi)$-good, $\wV^{\alpha, j}_{\chi, n}$ is $\chi$-good.}
By expression~\eqref{eq:zn-def} for $Z_n(\chi, \alpha)$, we see that $\chi\mapsto Z_n(\chi, \alpha)$ is $\chi$-good.

\xin{Recall the following standard fact on meromorphic functions.}
\begin{lemma}\label{lem:mero}
	If $f$ and $g$ are meromorphic functions on an open set $U\subset \CC$ with $f(z_k) = g(z_k)$ for some $z_k \in U$
	with an accumulation point in  $U$, then $f = g$ on all of $U$.
\end{lemma}
Since $\lim_{k\to\infty}\chi_k=0$ and both sides of~\eqref{eq:shift-Z} are $\chi$-good, \xin{Lemma~\ref{lem:mero} implies that \eqref{eq:shift-Z} holds for $\chi\in[0,\infty)$. As explained at the beginning of this subsection, by~\eqref{eq:dual} and~\eqref{eq:shift-Z}, for a fixed $\gamma\in(0,2)$ and $\chi\in\{\frac{\gamma}{2}, \frac{2}{\gamma}\} $ the shift equation~\eqref{eq:shift-Z-1} holds} for $\alpha\in (-\frac{4}{\gamma}+\chi,0)$. 
By Theorem~\ref{thm:series-opes} and Corollary~\ref{corr:gap-extend},  both sides
of~\eqref{eq:shift-Z} are meromorphic in $\alpha$ in a complex neighborhood
of  $(-\frac{4}{\gamma}+\chi,2Q-\chi)$. Thus~\eqref{eq:shift-Z} holds for
$\gamma\in(0,2)$, $\chi\in\{\frac{\gamma}{2}, \frac{2}{\gamma}\}$, and 
$(-\frac{4}{\gamma}+\chi,2Q-\chi)$. 
Now Theorem~\ref{thm:shift} and \xin{the uniqueness in Proposition~\ref{prop:shift-unique} imply (a)$_n$.}

It remains to prove (b)$_n$.  The argument  is parallel to the proof of Corollary~\ref{corr:gap-extend}
with $\chi$ in place of $\alpha$. By Theorem~\ref{thm:series-opes} and~\eqref{eq:shift-0}, we have 
\begin{align*}
\wphi_{\chi, n, 1}^{\alpha, 1}(0)  &=   \Big[\wcA_{\gamma, P, n}(\alpha - \chi)
+ \sum_{m = 0}^{n - 1} \frac{\eta_{\chi, n - m}^-(\alpha)}{\eta_{\chi, 0}^-(\alpha) } \wcA_{\gamma, P, m}(\alpha - \chi)\Big]; \\
\wphi^{\alpha, 2}_{\chi, n, 1}(0)&=- \Big[\wcA_{\gamma, P, n}(\alpha + \chi)
+ \sum_{m = 0}^{n - 1} \frac{\eta_{\chi, n - m}^+(\alpha)}{\eta_{\chi, 0}^+(\alpha)} \wcA_{\gamma, P, m}(\alpha + \chi)\Big] \frac{\Gamma_{0, 1}}{\Gamma_{0, 2}}\frac{1 - e^{\pi \chi P - \ii \pi l_\chi}}{1 + e^{\pi \chi P + \ii \pi l_\chi}}. 
\end{align*}
Since (a)$_m$ holds for $m\le n$, we see that
$\wphi^{\alpha, j}_{\chi, n, 1}(0)$ is $\chi$-good for $j=1,2$.  \revise{Since
$\wV^{\alpha, j}_{\chi, n}= \wG_{\chi, n, 1}^{\alpha, j}(0)$ is $\chi$-good for $j=1,2$,
by~\eqref{eq:phi-decomp1}} we have \(\widetilde{X}^j_{\chi, n}(\alpha) = \wphi^{\alpha, j}_{\chi, n, 1}(0) - \wV^{\alpha, j}_{\chi, n}\), which is $\chi$-good.
Therefore $\widetilde{X}^j_{\chi, n}(\alpha) v^\alpha_{j, \chi, n}(w)$ is $(w,\chi)$-good. Again by~\eqref{eq:phi-decomp1} we get~(b)$_n$. \qed

\appendix

\section{Conventions and facts on special functions} \label{sec:theta}

This appendix collects facts on the special functions we use in the main text.
See \cite[Chapters 20 and 23]{NIST:DLMF} and \cite{Bar04} for more details.

\subsection{Theta and Weierstrass elliptic functions}\label{subsec:special}
For \xin{$q\in (0,1)$}, the \notion{Dedekind eta function} is 
\(\eta(q)= q^{\frac{1}{12}} \prod_{k = 1}^\infty (1 - q^{2k}).\)
For  $\tau\in \bbH=\{\mathrm{Im}\,z>0\}$,  set $q = e^{\ii \pi \tau}$. \xin{Although $q^{\frac{1}{12}} \prod_{k = 1}^\infty (1 - q^{2k})$ is multi-valued for complex $q$, we interpret it as single-valued in $\tau\in \bbH$ via}  
\begin{equation}\label{eq:eta}
\eta(e^{\ii \pi \tau})= e^{\frac{\ii \pi \tau}{12}} \prod_{k = 1}^\infty (1 - e^{2k\ii \pi \tau}).
\end{equation}
For $u \in \CC$ \xin{and $\tau\in \bbH$},
the \notion{Jacobi theta function} is given by 
\begin{align}\label{eq:def-theta}
\Theta_\tau(u)= 
- 2 \xin{e^{\frac{\ii \pi \tau}4}} \sin(\pi u) \prod_{k = 1}^\infty (1 - \xin{e^{\ii 2k\pi \tau}}) (1 - 2 \cos(2\pi u) \xin{e^{\ii 2k\pi \tau}} + \xin{e^{\ii 4k\pi \tau}}).
\end{align}
Let $\Theta_\tau'(u)$ denote the $u$-derivative of $\Theta_\tau(u)$.
In these terms, we have 
\begin{equation}\label{eq:Theta'0}
\Theta_\tau'(0) = -2\pi \xin{e^{\frac{\ii \pi \tau}4}} \prod_{k = 1}^\infty (1 - e^{2k\ii \pi \tau})^3 = -2\pi\eta(\xin{e^{\ii \pi \tau}})^3.
\end{equation}
\xin{For a fixed $\tau$,} \notion{Weierstrass's elliptic function} $\wp$ is given by
\begin{equation} \label{eq:wp-theta}
\wp(u) \defeq \frac{\Theta_\tau'(u)^2}{\Theta_\tau(u)^2} - \frac{\Theta_\tau''(u)}{\Theta_\tau(u)} + \frac{1}{3} \frac{\Theta_\tau'''(0)}{\Theta_\tau'(0)}.
\end{equation}
It admits the following expansion  (see e.g.\ \cite[Equation (23.8.1)]{NIST:DLMF})
\begin{equation} \label{eq:wp-expand}
\wp(u) = \frac{\pi^2}{\sin^2(\pi u)} -  \sum_{n = 1}^\infty
\frac{8 \pi^2nq^{2n}}{1 - q^{2n}} \cos(2\pi nu) - \frac{\pi^2}{3} +  \sum_{n = 1}^\infty \frac{8\pi^2q^{2n}}{(1 - q^{2n})^2},
\end{equation}
\revise{where $\wp(u) = 1/u^2 + O(u^2)$ due to the identity
$\sum_{n = 1}^\infty \frac{n q^{2n}}{1 - q^{2n}} = \sum_{n = 1}^\infty \frac{q^{2n}}{(1 - q^{2n})^2}$.}
This expansion implies that $\wp(u)$ further admits a $q$-expansion 
\begin{equation}\label{eq:wpn}
\wp(u) = \sum_{n = 0}^\infty \wp_{n}(u) q^{n}, \quad \textrm{where } \wp_n(u) \equiv 0 \textrm{ for odd } n.
\end{equation}
By \eqref{eq:wp-expand}, there are polynomials $\wwp_n(w)$ with
$\wwp_{n}(w) = \wp_n(u)$ for $w = \sin^2(\pi u)$. 

\subsection{Gamma function and double gamma function}\label{subsec:gamma}
The gamma function is $\Gamma(z)\defeq\int_{0}^{\infty} t^{z-1}e^{-t} dt$ for $\Re z>0$. In particular, $\Gamma(n)=(n-1)!$ for $n\in\NN$.
It has meromorphic extension to $\CC$ with simple poles at $\{0,-1,-2,\cdots\}$
and satisfies $\Gamma(z+1)=z\Gamma(z)$, the reflection formula
\begin{equation}\label{eq:reflection-sin}
\Gamma(z)\Gamma(1-z)=\frac{\pi}{\sin(\pi z)} \qquad \textrm{for }z\notin \ZZ,
\end{equation}
and the Legendre duplication formula
\begin{equation}\label{eq:duplication}
\Gamma(2z) = 2^{2z-1} \pi^{-1/2} \Gamma(z)\Gamma(z+\frac12).
\end{equation}
By~\eqref{eq:reflection-sin}, $\Gamma(z)$ has no zeros and $\Gamma(z)^{-1}$ is entire with simple zeros at $\{0, -1,-2,\cdots\}$
 \revise{so that
\begin{equation} \label{eq:gamma-poles}
\Res_{x = - n} \Gamma(x) = \frac{(-1)^n}{n!}.
\end{equation}}
The $z\to\infty$ asymptotics of $\Gamma(z)$ are governed by Stirling's approximation; for each $\delta\in(0,\pi)$, with an error term $O(|z|^{-1})$ depending on $\delta$, we have
\begin{equation}\label{eq:stiring}
\Gamma(z) \sim \sqrt{\frac{2\pi}{z}}e^{-z}z^z(1+O(|z|^{-1})) \quad  \textrm{for }|z|>1 \textrm{ and } \arg z\in (\delta-\pi, \pi -\delta).
\end{equation}

Finally, for $\mathrm{Re}(z) >0$ the \notion{double gamma function} $\Gamma_{\frac{\gamma}{2}}(z)$ is
\begin{equation} \label{eq:dgamma-def}
\log \Gamma_{\frac{\gamma}{2}}(z) := \int_0^\infty \frac{dt}{t}
\Big[\frac{e^{-zt} - e^{-\frac{Qt}{2}}}{(1 - e^{-\frac{\gamma t}{2}})(1 - e^{-\frac{2t}{\gamma}})}
- \frac{(\frac{Q}{2} - z)^2}{2} e^{-t} + \frac{z - \frac{Q}{2}}{t}\Big].
\end{equation}
The function $\Gamma_{\frac{\gamma}{2}}(z)$ admits meromorphic extension to $\CC$ which
has no zeros and has simple poles at $\{ -\frac{\gamma n}{2} - \frac{2m}{\gamma} \mid n,m \in \mathbb{N} \}$. Moreover  
\begin{equation}\label{eq:sh_double_gamma}
\Gamma_{\frac{\gamma}{2}}(z + \chi)
= \sqrt{2 \pi} \chi^{\chi z - \frac{1}{2}} \Gamma(\chi z)^{-1} \Gamma_{\frac{\gamma}{2}}(z)\quad\textrm{for }\chi \in \{\frac{\gamma}{2}, \frac{2}{\gamma}\}.
\end{equation}
For $\gamma^2 \notin \mathbb{Q}$, $\Gamma_{\frac{\gamma}{2}}(z)$ is completely specified by (\ref{eq:sh_double_gamma})
and $\Gamma_{\frac{\gamma}{2}}(\frac{Q}{2}) = 1$. Other values of $\gamma$ can be recovered
by continuity. We also use the function 
\begin{align}
S_{\frac{\gamma}{2}}(z) := \Gamma_{\frac{\gamma}{2}}(z) \Gamma_{\frac{\gamma}{2}}(Q - z)^{-1}.
\end{align}

\subsection{Identities on $\Theta_\tau(u)$}\label{subsec:identities}
In Section~\ref{sec:bpz} we use the following identities on the theta function $\Theta_\tau(u)$.
Here we use the notations $\Theta'_{\tau}(u) : = \partial_u \Theta_{\tau}(u)$
and $\Theta''_{\tau}(u) : = \partial_{uu} \Theta_{\tau}(u)$, where $\partial_{\tau}$ and $\partial_u$
are holomorphic derivatives.  
\begin{equation} \label{eq:theta-heat}
\ii \pi \partial_\tau \Theta_\tau(u) = \frac{1}{4} \Theta_\tau''(u).
\end{equation}
\begin{multline} \label{eq:theta-iden}
\frac{\Theta_\tau''(a - b)}{\Theta_\tau(a - b)} + \frac{\Theta_\tau''(a)}{\Theta_\tau(a)} + \frac{\Theta_\tau''(b)}{\Theta_\tau(b)} - 2 \frac{\Theta_\tau'(a - b)}{\Theta_\tau(a - b)} \Big(\frac{\Theta_\tau'(a)}{\Theta_\tau(a)} - \frac{\Theta_\tau'(b)}{\Theta_\tau(b)}\Big)\\ - 2\frac{\Theta_\tau'(a)}{\Theta_\tau(a)} \frac{\Theta_\tau'(b)}{\Theta_\tau(b)} - \frac{\Theta_\tau'''(0)}{\Theta_\tau'(0)} = 0.
\end{multline}
\begin{equation}\label{eq:half-up}
\Theta_{\tau}(u + \frac{\tau}2) = - \ii e^{-\ii \pi u-\frac{\ii \pi\tau}{3}} \eta(\xin{e^{\ii \pi \tau}})
\prod_{n=1}^{\infty}  (1 -   e^{\xin{(2n-1)\ii\pi \tau} + 2 \pi i u} )(1 -  e^{\xin{(2n-1)\ii\pi \tau}-2 \pi i u} ).
\end{equation}
(\ref{eq:theta-heat}) comes \revise{from
\cite[Section 21.4]{WW02}}, (\ref{eq:theta-iden}) is stated in \cite[Equation (A.10)]{FLNO09} and may
be derived by applying the operator $\partial_y \partial_u - \frac{2 \ii}{\pi} \partial_\tau$ to
\cite[Exercise 21.13]{WW02}, and (\ref{eq:half-up}) comes from direct substitution.

\subsection{Fractional powers of $\Theta_\tau(u)$}\label{subsec:fractional}
We first recall the following fact.
\begin{lemma}\label{lem:log}
\xin{Fix a positive integer $k$.}
Suppose $f$ is analytic on a simply connected domain \xin{$U\subset \mathbb C^k$} such that $f(z)\neq 0$ for each $z\in U$. Then
there exists an analytic function $g$ on $U$ such that $f=e^g$. \xin{Moreover, if $\tilde g$ is an analytic function on $U$ such that $f = e^{\tilde{g}}$, then $\tilde g =  g + 2 k \pi \ii$ for a $k \in \mathbb{Z}$.} 
\end{lemma}

\xin{The function $g$ can be viewed as choice of branch for $\log f$, which satisfies $\mathrm{Re}(g)=\log |f|$ and  
 $\mathrm{Im}(g)=\arg f$ modulo $2\pi $. For $c\in \RR$, we define the fractional power $f^c$ by $f^c=e^{c\log f}$ for a specific choice of branch for $\log f$. For $f(q)=q^{-\frac{1}{12}} \eta(q)=\prod_{k=1}^\infty(1-q^{2k})$, which is  analytic on $\D$, we specify $\log f$ by requiring $\log f(q) = \sum_{k=1}^\infty \log(1-q^{2k})$ for $q\in (0,1)$. By Lemma~\ref{lem:log}, we have:}
\begin{lemma}\label{lem:eta}
\xin{For each $c\in\RR$, the function $[q^{-\frac{1}{12}} \eta(q)]^c$ is analytic  in $\mathbb D$}
\end{lemma}
\noindent \xin{We now specify $\Theta_\tau(u)^c$ as an analytic function of $(u,\tau)$ in a certain domain. Note that $\Theta_{\tau}(u)<0$ for $u\in(0,1)$ and $\tau\in \ii\RR$ by~\eqref{eq:def-theta}. Moreover, $\Theta_\tau(u)=0$ if and only if $u=m+n\tau$ for some  $m,n\in \ZZ$; see \cite[Section 20.2]{NIST:DLMF}.}
\begin{define}\label{def:log}
\xin{Let $D\defeq\{x+\tau y: x\in(0,1) \textrm{ or } y\in (0,1)\revise{, x, y \in \RR} \}$. Define $\Theta_\tau(u)^c=e^{c\log\Theta_\tau(u)}$ for  $c\in\RR$ and $(u,\tau)\in D\times \bbH$,  where $\log\Theta_\tau(u)$  is specified by requiring  $\Im \log\Theta_\tau(u)=\pi$ for $u\in (0,1)$ and $\tau\in\ii\RR$. For $u\in \RR\setminus \ZZ $ so that $\Theta_\tau(u)\neq 0$, we define $\log\Theta_\tau(u)$ and $\Theta_\tau(u)^c$ by continuous extension.
}
\end{define}
By Definition~\ref{def:log}, we have 
\begin{equation}\label{eq:phase}
\Theta_{\tau}(x)^{c} = \revise{e^{\ii c\pi }} |\Theta_{\tau}(x)|^{c}\qquad \textrm{for }x\in(0,1) \textrm{ and }\tau \in \ii\RR.
\end{equation}
Moreover, for $c\in \RR$, we have by \cite[Chapter 20.2]{NIST:DLMF} that
\begin{align}
\Theta_\tau(u+ 1)^c &= e^{-\ii \pi c} \Theta_\tau(u)^c   &&\textrm{if } u\in \{ x+y\tau: y\in (0,1)\}; \label{eq:theta-shift0}\\
\Theta_\tau(u + \tau)^c &= e^{-2 \pi \ii c(u - \frac{1}{2} + \frac{\tau}{2})} \Theta_\tau(u)^c   &&\textrm{if } u\in \{ x+y\tau: x\in (0,1)\}\label{eq:theta-shift1}.
\end{align}
\xin{The following lemma is needed in the next subsection.}
\begin{lemma}\label{lem:frac-bound}
\xin{Fix  $c>0$ and $\tau\in \bbH$.  Let $\overline{\band} = \{   u: 0\le \Im u \le \frac34\Im \tau \}$. Then there exists a constant $C$  such that  $|\Theta_\tau (u)^c|\le C|\sin (\pi u )|^c$ for $u\in \overline{\band}$.}
\end{lemma}
\begin{proof}
\xin{Define $\phi_\tau(u)$ by $\Theta_\tau(u)=\sin(\pi u)\phi_\tau(u)$. 
By~\eqref{eq:def-theta}, we have  $\phi_\tau(u)=- 2q^{1/4}\prod_{k = 1}^\infty (1 - q^{2k}) (1 - 2 \cos(2\pi u) q^{2k} + q^{4k})$, which is continuous on $\overline{\band}$.
Moreover $|\phi_\tau(u+1)|=|\phi_\tau(u)|$ by~\eqref{eq:theta-shift0}. Therefore $|\phi_\tau|$ is bounded on 
$\overline{\band} $ hence $|\Theta_\tau (u)^c|\le C|\sin (\pi u )|^c$ with $C=\max_{u\in \overline{\band}}{|\phi_\tau|^c}$.}
\end{proof}

\subsection{On the definition of the $u$-deformed block}\label{subsec:u-theta}
We \xin{first prove  Lemma~\ref{lem:g} and then use it to make sense of fractional powers of $f(u,q)$  needed for the definition of  the $u$-deformed block in Section~\ref{sec:bpz}; see Definition~\ref{def:fractional-power}.}
\begin{proof}[\xin{Proof of Lemma~\ref{lem:g}}]
 \xin{By \cite[Equation (20.5.10)]{NIST:DLMF}, we have}
\begin{equation}\label{eq:log-Theta}
\frac{\Theta'_{\tau}(u)}{\Theta_{\tau}(u)} = \pi \frac{\cos( \pi u )}{\sin( \pi u )}  + 4 \pi \sum_{n=1}^{\infty} \frac{q^{2n}}{1 - q^{2n}} \sin( 2 \pi n u).
\end{equation}	
Since \revise{ $\mathrm{Re}(\cot(z)) = \frac{e^{-4\pi \mathrm{Im}(z)}-1}{|e^{2\pi \ii z}-1|^2}$} and $\mathrm{Im} (\sin z)= \cos(\mathrm{Re}(z))(e^{\mathrm{Im}(z) } -  e^{-\mathrm{Im}(z) })$,  
\begin{align*}
\Im \left( \frac{\Theta'_{\tau}(z)}{\Theta_{\tau}(z)} \right) 
&=  -\pi \frac{1 - e^{-4 \pi \Im(z)}}{|e^{2 \ii \pi z }-1|^2}\\
&\phantom{==} + 2 \pi  \sum_{n=1}^{\infty} \frac{q^{2n}}{1 - q^{2n}} (e^{2 \pi n \Im(z)} - e^{-2 \pi n \Im(z)} ) \cos(2 \pi n \Re(z)).
\end{align*}

Since $\mathrm{Im}(z)>0$,  we have \(\pi \frac{1 - e^{-4 \pi \Im(z)}}{|e^{2 \ii \pi z }-1|^2} > \frac{\pi}{4}( 1 - e^{-4 \pi \Im(z)})\). 
Note that 
\begin{multline*}
2 \pi  \sum_{n=1}^{\infty} \frac{q^{2n}}{1 - q^{2n}} (e^{2 \pi n \Im(z)} - e^{-2 \pi n \Im(z)} ) \cos(2 \pi n \Re(z))\\
< \frac{2 \pi}{1 - q^2} \Big(\frac{q^2 e^{2 \pi  \Im(z)}}{1 - q^2 e^{2 \pi  \Im(z)}} - \frac{q^2 e^{- 2 \pi  \Im(z)}}{1 - q^2 e^{- 2 \pi  \Im(z)}} \Big).
\end{multline*}
For $h(x)=\frac{q^2 x}{1 - q^2 x }$, since $h'(x) = \frac{q^2}{(1-q^2x)^2}\le  \frac{q^2}{(1-q^2e^{2 \pi  \Im(z)})^2}$ for $x\in [e^{-2 \pi  \Im(z)},e^{2 \pi  \Im(z)}]$, we have 
\begin{align*}
\frac{q^2 e^{2 \pi  \Im(z)}}{1 - q^2 e^{2 \pi  \Im(z)}} - \frac{q^2 e^{- 2 \pi  \Im(z)}}{1 - q^2 e^{- 2 \pi  \Im(z)}} 
&< \frac{q^2}{(1-q^2e^{2 \pi  \Im(z)})^2}\Big(e^{2 \pi \Im(z)} - e^{-2 \pi  \Im(z)}\Big)  \\
&=\frac{q^2e^{2 \pi  \Im(z)}}{(1-q^2e^{2 \pi  \Im(z)})^2} \Big(1 - e^{-4 \pi  \Im(z)}  \Big).
\end{align*}
For $\Im z<\frac34\Im \tau $, we have $q^2e^ {2\pi \Im z}<q^{\frac12}$.
Since $\frac{x}{(1-x)^2}$ is monotone on $(0,1)$,
\[
\frac{2 \pi}{1 - q^2}\frac{q^2e^{2 \pi  \Im(z)}}{(1-q^2e^{2 \pi  \Im(z)})^2}  \Big(1 - e^{-4 \pi  \Im(z)}  \Big)<\frac{2 \pi}{1 - q^2} \frac{q^{\frac12}}{(1-q^{\frac12})^2} \Big(1 - e^{-4 \pi  \Im(z)}  \Big).
\]
Since $\lim_{q\to0 }\frac{2 \pi}{1 - q^2} \frac{q^{\frac12}}{(1-q^{\frac12})^2}=0$, $q_0$ with the desired property exists.
\end{proof}
\xin{From Lemma~\ref{lem:g} we immediately have the following.}
\begin{lemma}\label{arg-theta}
Fix $q\in (0,q_0)$ and $u \in\band$.  Let the straight line between $\Theta_\tau(u)$ and $\Theta_\tau(u + 1) = - \Theta_\tau(u)$
divide the complex plane into two open half planes $\mathbb H^-_u $ and $\mathbb H^+_u$, where $\mathbb{H}^-_u$ contains
a small clockwise rotation of $\Theta_{\tau}(u)$ viewed as a vector.  For $x \in (0, 1)$, we have
$\Theta_{\tau}(u+x)\in \mathbb H^-_u$.
\end{lemma}

\begin{proof}
Let \revise{$g(x)=\mathrm{Im} \log \Theta_\tau(u+x)$} for $x\in \mathbb R$. 
\xin{Since $(\log \Theta_\tau)'= \frac{\Theta'_{\tau}(u)}{\Theta_{\tau}(u)}$, 
by Lemma~\ref{lem:g},} $g(1)-g(0)=-\pi$ and $g'(x)<0$ for $x\in \mathbb R$. 
By the definition of $\mathbb H^-_u$, for $x \in (0, 1)$ we have $\Theta_{\tau}(u+x)\in \mathbb H^-_u$.
\end{proof} 
\xin{Fix $\gamma\in(0,2)$ and  $\chi \in \{\frac{\gamma}{2}, \frac{2}{\gamma}\}$. 
Recall $f(u,q)$ from~\eqref{eq:fnu-def}. For notational simplicity, we set 
 $c=\frac{\gamma\chi}{2}$ and write
\begin{equation}\label{eq:fnu}
f(u,q)=\int_0^1 \Theta(u+x)^{c} \nu(dx), 
\end{equation}
where  $\nu(dx)$ is the measure on $[0, 1]$ given by}
\begin{equation}\label{eq:mu}
\nu(dx)\defeq |\Theta_{\tau}(x)|^{-\frac{\alpha\gamma}{2}} e^{\pi \gamma P x} e^{\frac{\gamma}{2} Y_\tau(x)} dx.
\end{equation}
\xin{Using Lemma~\ref{lem:f-nu} below we can define  the fractional power of $f(u,q)$.}
\begin{lemma}\label{lem:f-nu}
Fix $q\in (0,q_0)$.  Then $f(u,q)$  is analytic in $u$ on $\band$ and \xin{can be continuously extended to}  $\overline{\band}$. Moreover, $f(u+1,q)=e^{-c\pi \ii} f(u,q)$ and $f(u,q)\neq 0$ for $u\in \band$.  Finally, $f(1,q)>0$.
\end{lemma}
\begin{proof}
\xin{By~Lemma~\ref{lem:frac-bound},}  $\Theta_\tau^c$ is  bounded on $\overline\band$ and continuous except at integers. Therefore $f(u,q)$ in~\eqref{eq:fnu} can be continuously extended to $\overline{\band}$. 
By~\eqref{eq:theta-shift0}, $f(u+1,q)=e^{-c\pi \ii} f(u,q)$ for $u\in \overline{\band}$. By Lemma~\ref{arg-theta}, $\Im (\Theta_\tau(u+x)^c\Theta_\tau(u)^{-c} )<0$ for $u\in \band$, hence $\Im (\Theta_\tau(u)^{-c}  f(u,q))$ $<0$.
Since \(\Theta_\tau(u)^{-c}  f(u,q)=\int_0^1  \Theta_\tau(u+x)^c\Theta_\tau(u)^{-c}  \nu(dx)\),
we get $f(u,q)\neq 0$. Since $\nu$ is supported on $[0,1]$ and $\Theta_\tau(1+x)>0$ for $x\in(0,1)$, we have $f(1,q)>0$.
\end{proof}
 
\begin{define}\label{def:fractional-power}
Fix $q\in(0,q_0)$ as in Lemma~\ref{arg-theta}. We define $u\mapsto \log f(u,q)$ for $u\in\band$ by requiring $\lim_{u\to 1}\Im [\log f(u,q)]=0$.
For each $\beta\in \RR$, we define $f(u,q)^\beta := e^{\beta \log f(u,q)}$.
\end{define}

\section{Background on log-correlated fields and Gaussian multiplicative chaos} \label{sec:gmc-app}
Let us first provide a general definition of log-correlated fields.
\begin{define}\label{def:LGF}
A centered Gaussian process $X$ on a domain $U \subset \mathbb{R}^d $ is called a log-correlated field if it admits a covariance kernel of the form
\begin{equation}\label{eq:def-log-kernel}
\mathbb{E}[X(x) X(y)] = c\log \frac{1}{|x-y|} + g(x,y),
\end{equation}
where $c$ is a positive constant and $g: U\times U \mapsto \mathbb{R}$ is a continuous function.
\end{define}
Due to the singularity of the log kernel, these fields cannot be defined as pointwise functions but only as random generalized functions (distributions). Given a random variable $\cX$, we use $\mathbb{E}[\cX X(x)]$ to denote the distribution such that  $\int dx h(x) \mathbb{E}[\cX X(x)]= \mathbb{E}\left[\cX \left( \int dx h(x) X(x) \right) \right]$ for all suitable test functions $h$. In a similar fashion, the covariance kernel \eqref{eq:def-log-kernel} means that for all test functions $h_1, h_2$, one has 
\[
\mathbb{E}\left[ \left(  \int dx  h_1(x) X(x) \right) \left(\int dy h_2(y) X(y) \right) \right]=\iint dx dy h_1(x) h_2(y)  \mathbb{E}[X(x) X(y)].
\]

Consider a field $X$ as in Definition~\ref{def:LGF} with $d=1, c=2$, and fix $\gamma\in (0,2)$.
For a large class of regularizations $\{X_n\}$ of $X$, $e^{\frac{\gamma}{2}X_n-\frac{\gamma^2}{8}\EE[X_n(x)^2] } dx$ converges
in probability to the unique GMC measure $e^{\frac{\gamma}{2}X}dx$ associated with $X$, see, e.g.\ \cite{NB-GMC}.
Definition~\ref{def:gmc} is a special case of such limiting procedures. 
Consider the log-correlated field $X_{\mathbb{H}}$ on the upper half plane \revise{with its boundary included} with covariance
\begin{equation}\label{eq:cov-X}
\mathbb{E}[X_{\mathbb{H}}(x) X_{\mathbb{H}}(y)] = \log \frac{1}{|x-y| |x- \bar y|} - \log | x +\ii|^2 - \log |y + \ii|^2 + 2\log 2
\end{equation}
for $x,y\in \bbH \cup \mathbb{R}$. The field $X_\bbH$ is an example of a \notion{free boundary Gaussian free field} (GFF) on $\bbH$. For \revise{$x\in \bbH \cup \mathbb{R}$}, let $\overline X(x)$ be the average of $X_\bbH$ over the semi-circle $\{z\in \bbH: |z|=|x|\}$. Let $Z_\bbH\defeq X_\bbH-\overline X$. 
Then $\overline X$ and $Z_\bbH$ are independent and $Z_\bbH$ is a log-correlated field with covariance 
\begin{equation}\label{eq:cov-Z}
\mathbb{E}[Z_\bbH(x) Z_\bbH(y)] = 2 \log \frac{|x| \vee |y|}{|x-y|}.
\end{equation}
We use the field $Z_\bbH$  and $\overline X$ in Appendix \ref{sec:ope-proof}.

We now state a general result of existence of moments of GMC measure covering all situations encountered in the main text.
Concretely, we will use the case when $F(x)$ below equals $\frac{\gamma}{2}F_{\tau}(x) $ or $0$, where $F_\tau$ is as in~\eqref{eq:f-tau-def}.
\begin{lemma}[Moments of GMC]\label{lem:GMC-moment}
Fix $\gamma \in (0, 2)$ and $\alpha < Q$. 
Fix $\chi \in \{\frac{\gamma}{2}, \frac{2}{\gamma}\}$.  
Let $F: [0,1] \mapsto \mathbb{R}$  be a continuous Gaussian field independent of $Y_\infty(x) $, 
and $f:[0,1] \mapsto (0,+ \infty)$ be a continuous bounded function. Then for $p <  \frac{4}{\gamma^2} \wedge \frac{2}{\gamma} (Q -\alpha)$ we have 
\begin{equation}\label{eq:GMC-moment1}
\E\left[\left(\int_0^1 e^{ F(x)} \sin(\pi x)^{-\frac{\alpha\gamma}{2}} f(x) e^{\frac{\gamma}{2} Y_\infty(x) }dx \right)^{p}\right]<\infty.    
\end{equation}
\xin{For $q\in (0,q_0)$ with $q_0$ defined in Lemma~\ref{lem:q-bound},}
let $\band=\{z: 0<\Im(z)< \frac34\Im(\tau) \}$ \xin{and $K\subset \band$ be compact.}
For $p <  \frac{4}{\gamma^2} \wedge \frac{2}{\gamma} (Q -\alpha)$, we have \xin{uniformly in $u\in K$} that 
 \begin{equation}\label{eq:GMC-moment2}
\E\left[ \left| \int_0^1 e^{ F(x)} \sin(\pi x)^{-\frac{\alpha\gamma}{2}}  \Theta_{\tau}( x + u)^{\frac{\chi \gamma}{2}} f(x)e^{\frac{\gamma}{2} Y_{\infty}(x) }dx  \right|^{p}\right] <\infty.
\end{equation} 
For $p <  \frac{4}{\gamma^2} \wedge \frac{2}{\gamma} (Q -\alpha \vee \gamma)$, we have \xin{uniformly in $u\in K$  and 
$y \in [0,1]$} that
\begin{equation}\label{eq:GMC-moment3}
\E\left[ \left| \int_0^1 e^{ F(x) + \frac{\gamma^2}{4} \EE[Y_{\infty}(x) Y_{\infty}(y) ]} \sin(\pi x)^{-\frac{\alpha\gamma}{2}}  \Theta_{\tau}( x + u)^{\frac{\chi \gamma}{2}} f(x)e^{\frac{\gamma}{2} Y_{\infty}(x) }dx  \right|^{p}\right] <\infty.    
\end{equation}
\end{lemma}
\begin{proof} 
For~\eqref{eq:GMC-moment1} we are in the classical
case of the existence of moments of GMC with insertion of weight $\alpha$. The assertion follows from the argument for \cite[Lemma~3.10]{DKRV16},
adapted to the case of one-dimensional GMC.

\xin{For~\eqref{eq:GMC-moment2}, when $p\ge 0$, by Lemma~\ref{lem:frac-bound}  $|\Theta_{\tau}( x + u)^{\frac{\chi \gamma}{2}}|$ is uniformly bounded from above for $u\in K$ and $x\in [0,1]$ therefore~\eqref{eq:GMC-moment2} follows from~\eqref{eq:GMC-moment1}. When $p<0$, let $h_1(x,u):=-\Im (\Theta_\tau(u+x)^\frac{\chi \gamma}{2}\Theta_\tau(u)^{-\frac{\chi \gamma}{2}} )$ and $h_2(x,u):=\Re (\Theta_\tau(u+x)^\frac{\chi \gamma}{2}\Theta_\tau(u)^{-\frac{\chi \gamma}{2}} )$. Since $ \frac{\chi \gamma}{2} \leq 1$,  by Lemma~\ref{arg-theta}, we have $h_1(x,u)>0$ for $u\in \band$ and $x\in (0,1)$. On the other hand, we have $|h_2(x,u)|>0$ for $u\in \band$ and $x\in\{0,1\}$. Therefore  $|h_1(x,u)| \wedge  |h_2(x,u)|$ is uniformly bounded from below for $x\in [0,1]$ and $u\in K$. Now  we again get~\eqref{eq:GMC-moment2} from~\eqref{eq:GMC-moment1}.}


Lastly~\eqref{eq:GMC-moment3} can be treated the same way as~\eqref{eq:GMC-moment2} except the bound on $p$ changes to $p < \frac{4}{\gamma^2} \wedge \frac{2}{\gamma} (Q -\alpha \vee \gamma)$ due to the $\frac{\gamma^2}{4} \EE[Y_{\infty}(x) Y_{\infty}(y) ]$ term. This is equivalent to adding a $\gamma$ insertion which results in a modification on the bound on $p$ as shown in \cite[Lemma~3.10]{DKRV16}.
\end{proof}

Finally, we state a form of Cameron-Martin's theorem used in the main text. See \cite{shamov}, \cite[Section 3.3]{BP-notes} or \cite[Theorem 0.2]{Aru2017} for details.
\begin{theorem}\label{thm:Girsanov}
Let $Y(x)$ be either of the Gaussian fields $Y_\infty(x)$ or $Y_\tau(x)$ on $[0,1]$ defined in Section~\ref{sec:gmc-def}. 
Let $\cX$ be a Gaussian variable such that the joint process $(\cX, Y)$ is Gaussian, and let $F$ be a non-negative measurable function. Then 
\begin{equation}\label{eq:Girsanov}
\E\left[e^{\cX-\frac{1}{2}\EE[\cX^2]} F( (Y(x))_{x \in [0,1]} )\right]  = \E\left[ F( (Y(x) + \EE[\cX Y(x)])_{x \in [0,1]} )\right]. 
\end{equation}
Moreover, non-negative measurable functions  $f$ and $G$, we have
\begin{equation}\label{eq:Girsanov2}
\E\left[e^{\cX-\frac{1}{2}\EE[\cX^2]} G\left(\int_0^1 f(x) e^{\frac{\gamma}{2}Y(x)}dx \right)\right]  = \E\left[ G\left( \int_0^1 f(x) e^{\frac{\gamma}{2}(Y(x) + \EE[\cX Y(x)]) }dx  \right)\right]. 
\end{equation}
\end{theorem}
Theorem~\ref{thm:Girsanov} means that under the reweighing by the Radon-Nikodym derivative $e^{\cX-\frac{1}{2}\EE[\cX^2]}$, the law of $Y$ is that of \revise{$(Y(x) + \EE[\cX Y(x)])_{x \in [0,1]} $} under the original probability. 
We frequently apply \eqref{eq:Girsanov2} to the case where $G(x) = x^p$ where the functional of interest becomes a moment of GMC. 

\section{Hypergeometric differential equations} \label{sec:hgf}
For parameters $A, B, C$ and a function $g$ on $\CC$, the 
hypergeometric
equation (see e.g.\ \cite[Chapter 15]{NIST:DLMF}) with inhomogeneous part $g$ is
\begin{equation} \label{eq:gauss-hgf}
(w(1 - w) \partial_{ww} + (C - (1 + A + B) w)\partial_w - AB)\, f(w) = g(w)
\end{equation}
This appendix reviews the background on the solution theory to~\eqref{eq:gauss-hgf}. \xin{We will omit the proof of  basic facts such as Lemma~\ref{lem:1D-R}.}

\subsection{Homogeneous hypergeometric differential equations}\label{subsec:homo}
If $g(w) = 0$, the homogeneous equation (\ref{eq:gauss-hgf}) can be solved by the \emph{Gauss hypergeometric function}  ${_2F_1}(A, B, C; w) $, whose power series coefficients  are characterized by $a_0={_2F_1}(A, B, C; 0)=1$ and $\frac{a_{n+1}}{a_n}=\frac{(n+A)(n+B)}{(n+1)(n+C)}$ for $n\in \NN$.  It solves~\eqref{eq:gauss-hgf} with $g=0$. Moreover, we have the following.
\begin{lemma}\label{lem:1D-R}
\xin{Set $v_1(w)\defeq {_2F_1}(A, B,C; w)$ and $v_2(w)\defeq{_2F_1}(1 + A - C, 1 + B - C, 2 - C; w)$. 
Then 
both $v_1(w)$ and $w^{1-C}v_2(w)$ solve~\eqref{eq:gauss-hgf} with $g=0$, and 
they form a linear basis of the space of solutions to
\eqref{eq:gauss-hgf}   defined on an open subset of $\CC \setminus ((-\infty,0]\cup [1,\infty))$.}
\end{lemma}


If $\Re(C) > \Re(A + B)$, the series for ${_2F_1}(A, B, C; w) $   converges absolutely on the closed unit disk $\overline\D$. Therefore we have:
\begin{lemma}\label{lem:v-R}
If $\Re(C) > \Re(A + B)$, both $v_1(w)$ and $v_2(w)$ satisfy Property (R) from Definition~\ref{def:property-R}. 
\end{lemma}
A separate basis of solutions to~\eqref{eq:gauss-hgf} with good behavior at $w = 1$ is 
\[
{_2F_1}(A, B, 1 + A + B - C, 1 - w), \qquad
(1 - w)^{C - A - B}{_2F_1}(C - A, C - B, 1 + C - A - B, 1 - w).
\]
The two bases are related by \emph{connection equations}. In particular, we have
\begin{multline} \label{eq:app-ccoeff}
{_2F_1}(A, B, 1 + A + B - C, 1 - w) = \frac{\Gamma(C) \Gamma(C - A - B)}{\Gamma(C - A) \Gamma(C - B)} v_1(w)\\
+ \frac{\Gamma(2 - C) \Gamma(C - A - B)}{\Gamma(1 - A) \Gamma(1 - B)} w^{1 - C} v_2(w).
\end{multline}
If $\Re(C) > \Re(A + B)$, the coefficients in~\eqref{eq:app-ccoeff} satisfy \emph{Euler's identity}
\begin{equation} \label{eq:gauss}
v_1(1)={_2F_1}(A, B, C, 1) = \frac{\Gamma(C) \Gamma(C - A - B)}{\Gamma(C - A) \Gamma(C - B)}\quad \textrm{and}\quad  v_2(1)= \frac{\Gamma(2 - C) \Gamma(C - A - B)}{\Gamma(1 - A) \Gamma(1 - B)}.
\end{equation}
Moreover, $\frac{{_2F_1}(A, B, C, w)}{\Gamma(C)}$ is holomorphic in $A, B, C$.
Since $\Gamma$ is meromorphic on $\CC$ with poles at $\{0,-1,-2,\cdots\}$ and has no zeros, we have the following. 
\begin{lemma}\label{lem:v-analytic}
Let $V=\{(A,B,C)\in \CC^3:  \Re(C) > \Re(A+B)\textrm{ and }C\notin\ZZ \}$.
Both functions  $(w, A,B,C)\mapsto v_1$ and $(w, A,B,C)\mapsto v_2$ are continuous on $\overline \D\times V$ and analytic on $\D \times V$.  
\revise{Moreover, if $C-A$ and $C-B$ (resp. $1-A, 1-B$) are not in $\{0,-1,-2,\cdots\}$, then 
$v_1(1)\neq 0$ (resp. $v_2(1)\neq 0$).}
\end{lemma}

\subsection{Inhomogeneous hypergeometric
differential equations}\label{subsec:inhomo}
If $g(w)$ is not identically zero, then any solution to~\eqref{eq:gauss-hgf} can be written as
\begin{equation} \label{eq:hgf-decomp}
f(w) = f_{\text{homog}}(w) + f_{\text{part}}(w),
\end{equation}
where $f_{\text{part}}(w)$ is a particular solution to (\ref{eq:gauss-hgf}) and $f_{\text{homog}}(w)$ solves
the homogeneous version of (\ref{eq:gauss-hgf}).  We will give a particular solution to~\eqref{eq:gauss-hgf}
using power series.  
We need the following extended notion of integration.
Suppose $f(w)=\sum_{n = 0}^{\infty} a_nw^n$ for $w\in \D$ such that 
\begin{equation}\label{eq:series}
\sum_{n = 1}^{\infty}\frac{|a_n|}{n}<\infty,
\end{equation} and $\beta\in \CC\setminus \{1,2,\cdots \}$.
We define $\int_0^w t^{-{\beta}}f(t)dt$ by 
\begin{equation} \label{eq:series-notation}
\int_0^w t^{-{\beta}}f(t)dt := w^{1 - {\beta}}\sum_{n = 0}^{\infty} \frac{a_n}{n - {\beta} + 1} w^n \qquad \text{for } w \in {\overline\D}.
\end{equation}
The following lemma, which is left as an exercise to the reader,  justifies the integral notation $\int_0^w t^{-{\beta}}f(t)dt$.
\begin{lemma}\label{lem:integral}
In the setting of (\ref{eq:series-notation}), we have
\[
\frac{\partial}{\partial w}\Big(\int_0^w t^{-{\beta}}f(t)dt\Big) = w^{-{\beta}}f(w) \qquad \text{for } w\in \D.
\]
Moreover, $w^{{\beta}}\int_0^w t^{-{\beta}}f(t)dt=\sum_{n = 0}^{\infty} \frac{a_n}{n-{\beta} + 1} w^{n+1}$ satisfies Property (R).
\end{lemma}
\xin{The following lemma is used in the proof of Corollary~\ref{cor:decompose}.}
\begin{lemma}\label{lem:property-R}
Let $U\subset \CC^3$ be such that $(A,B,C)\in U$ if and only if
	\begin{equation*}\label{eq:ABC}
	C\notin \ZZ, \;  \Re(C - A - B)\in (0,1),\textrm{and } C-A, C-B, 1-A,1-B \notin \{0,-1,\cdots \}.
	\end{equation*}
Fix  $X \in \{0, 1 - C\}$. 
Suppose $g(w) = w^X \wg(w)$, and $\wg(w)$ is a function  satisfying Property (R). 
Then for each $a\in \CC$, there exists a unique function  $f_a$ satisfying Property (R) such that $f_a(1)=a$ and
$w^{X} f_a(w)$  solves equation~\eqref{eq:gauss-hgf}.
Moreover, $(w, A,B,C)\mapsto f_a(w)$ is  \revise{continuous on $\overline \D\times U$ and analytic on $\D \times U$.} 
\end{lemma}
\begin{proof}
We check easily that if $\sum_{n =0}^{\infty}a_nt^n$  satisfies Property (R) and $\sum_{n = 0}^{\infty}b_nt^n$
satisfies~\eqref{eq:series}, then $\left(\sum_{0}^{\infty}a_nt^n\right) \left(\sum_{0}^{\infty}b_nt^n\right)$
satisfies~\eqref{eq:series}. By~\eqref{eq:ABC}, the series expansion of
$(1 - t)^{A + B-C}$ around $0$ satisfies~\eqref{eq:series}. Thus $\frac{v_2(t) g(t)}{(1 - t)^{C - A - B}}$ (resp., $\frac{v_1(t) g(t)}{t^{1 - C} (1 - t)^{C - A - B}} $) 
is $t^X$ (resp., $t^{C-1+X}$) times a power series satisfying~\eqref{eq:series}. Let  
\begin{multline}\label{eq:particular}
f_{\text{part}}(w):=-\frac{v_1(w)}{1 - C} \int_0^w \frac{v_2(t) g(t)}{(1 - t)^{C - A - B}} dt\\
+ \frac{v_2(w)}{1 - C} w^{1 - C}\int_0^w \frac{v_1(t) g(t)}{t^{1 - C} (1 - t)^{C - A - B}} dt, \textrm{ for }w\in\overline \D,
\end{multline}
where both integrals in~\eqref{eq:particular} are defined via (\ref{eq:series-notation}).
\revise{By Lemma~\ref{lem:integral},}
\begin{multline}\label{eq:R-part}
\textrm{both $w^{-X}\int_0^w \frac{v_2(t) g(t)}{(1 - t)^{C - A - B}} dt$}\\
\textrm{and $w^{1 - C-X}\int_0^w \frac{v_1(t) g(t)}{t^{1 - C} (1 - t)^{C - A - B}} dt$ satisfy Property (R).}
\end{multline}
Thus $w^{-X}f_{\text{part}}(w)$ satisfies Property (R)
and is analytic in $A,B,C$.

\revise{Since the Wronskian determinant of $v_1$ and $w^{1-C} v_2$ is given by  $(1-C) w^{-C}(1-w)^{C-1-A-B}$,
by the  variation-of-parameters method in differential equations,}
$f_{\text{part}}$ is a particular solution
to~\eqref{eq:gauss-hgf}.
Since $v_1(1)\neq 0$ by~\eqref{eq:gauss}, if $X=0$, then
\begin{equation}\label{eq:particular2}
f_a(w)= f_{\text{part}}(w) + (a-f_{\text{part}}(1)) \frac{v_1(w)}{v_1(1)}
\end{equation}
is the desired  function, which is  unique by Lemma~\ref{lem:1D-R}.
If $X=1-C$, we conclude  similarly with $w^{1-C} v_2(w)$ in place of $v_1(w)$. 
\end{proof}
\xin{We now extend Lemma~\ref{lem:property-R} to incorporate the $(w,\alpha)$-regularity.}
\begin{lemma}\label{lem:alpha-ext2}
Suppose $U\subset\CC$ is an open set and $g(w,\alpha)$ is a  function which is $(w,\alpha)$-regular on $\overline \D \times U$ in the sense of Definition~\ref{def:walpha}.
Suppose we are in the setting of Lemma~\ref{lem:property-R}. 
For \revise{$a\in\CC$}, let $f(w,\alpha)$ be defined as $f_a(w)$ in Lemma~\ref{lem:property-R} with $g=g(w,\alpha)$.  Then $f(w,\alpha)$ is $(w,\alpha)$-regular on $\overline \D \times U$.
\end{lemma}
\begin{proof}
Recall that ${_2F_1}(A, B, C, w)$ is holomorphic for $C \notin \{0, -1, -2, \ldots\}$.  
{By the same argument as in~\eqref{eq:R-part}, we see that 
both $w^{-X}\int_0^w \frac{v_2(t) g(t)}{(1 - t)^{C - A - B}} dt$ and $w^{1 - C-X}\int_0^w \frac{v_1(t) g(t)}{t^{1 - C} (1 - t)^{C - A - B}} dt$ are $(w,\alpha)$-regular on $\overline{\D}\times U$.}
Therefore $w^{-X}f_{\textrm{part}}$ is $(w,\alpha)$-regular on $\overline D\times U$ with $f_{\textrm{part}}$  from equation~\eqref{eq:particular}.
If $X=0$, we obtain Lemma~\ref{lem:alpha-ext2} by~\eqref{eq:particular2}. If $X=1-C$, we can use the counterpart of~\eqref{eq:particular2} with $w^{1-C} v_2$ in place of $v_1$.
\end{proof}

We need the following variant of
Lemma~\ref{lem:property-R}, which can be proved by the same argument as Lemma~\ref{lem:property-R}.
\begin{lemma}\label{lem:weak}
Suppose $C$ is not an integer. 	Fix  $X \in \{0, 1 - C\}$. 
Suppose $g(w) = w^X \wg(w)$, and $\wg(w)$ is an analytic function  on $\D$. 
Let $f_{\text{part}}$  be defined as in \eqref{eq:particular}. Then   $f_{\text{part}}$ is a particular solution
to~\eqref{eq:gauss-hgf}. Moreover, $w^{-X}f_{\text{part}}(w)$ is an analytic function on $\D$.
\end{lemma}
\xin{The next lemma is used in the proof of    Lemma~\ref{lem:ext-phi}.}
\begin{lemma}\label{lem:continuity}
Suppose $A,B,C,X,g$ are as in Lemma~\ref{lem:weak} with $\Re(1 - C) > 0$. Given $\theta_0\in [0,2\pi)$, let $D=\{z=re^{\ii \theta}: r\in (0,1), \theta\neq \theta_0\}$. 
Suppose $f$ solves~\eqref{eq:gauss-hgf} on an open set $U\subset D$. 
Then $f$ can be extended to an analytic function on $D$ such that  as $w\in D$ tends to 0, $f(w)$ tends to a finite number.
\end{lemma}
\begin{proof}
Since $w^X$ restricted to $U$ can be analytically extended to $D$, Lemma~\ref{lem:continuity} follows from Lemma~\ref{lem:weak}, \eqref{eq:hgf-decomp}, and the solution structure of the homogeneous hypergeometric equation from Section~\ref{subsec:homo}.  
\end{proof}

\section{Proof of OPE lemmas} \label{sec:ope-proof}
In \xin{this section we prove Lemmas~\ref{thm:opes} and~\ref{thm:opes2}.}
We will be brief when the arguments  are straightforward adaptation of those in \cite{KRV19b,RZ20}.
\begin{proof}[Proof of Lemma~\ref{thm:opes}] 
Set $\lgg = l_{\frac{\gamma}{2}}$. Let $g(u) := \int_0^1 e^{\frac{\gamma}{2} Y_\tau(x)} \Theta_\tau(x)^{-\frac{\alpha\gamma}{2}} \Theta_\tau(u + x)^{\frac{\gamma^2}{4}} e^{\pi \gamma P x} dx$ and $f(u) = \EE[g(u)^{-\frac{\alpha}{\gamma} + \frac{1}{2}}]$. 
Using $\eqref{eq:q-block}$ we can write $\phi^\alpha_{\frac{\gamma}{2}}(u, q)$ as $\Sigma(u)f(u)$ where $\Sigma(u) =  \cW(e^{\ii \pi \tau}) e^{\pi P u \pi} \sin(\pi u)^{\lgg} \Theta_{\tau}(u)^{-\lgg}$ and $\cW(e^{\ii \pi \tau})$ is defined under \eqref{eq:q-block}. $\Sigma$ is thus differentiable at $0$. For $t\in [0,1]$, let $g(t, u) := (1 - t) g(0) + t g(u)$. The claim of interest can be reduced to \revise{determining} the leading order of:
\begin{align*}
f(u) - f(0) = \left(-\frac{\alpha}{\gamma} + \frac{1}{2}\right) \int_0^1 \EE[(g(u) - g(0)) g(t, u)^{-\frac{\alpha}{\gamma} - \frac{1}{2}}] dt.
\end{align*}
By a direct asymptotic analysis as done in \revise{\cite[Equation (3.15)]{RZ20}}, we have
uniformly in $t\in[0,1]$ that
\begin{align}
&\lim_{u \to 0} u^{-2\lgg - 1} \EE[(g(u) - g(0)) g(t, u)^{-\frac{\alpha}{\gamma} - \frac{1}{2}}] \label{eq:ope-easy}\\
=&(1 - e^{\pi \gamma P - 2\ii \pi \lgg})C\EE\left[\left(\int_0^1 e^{\frac{\gamma}{2} Y_\tau(x)} \Theta_\tau(x)^{-\frac{\alpha\gamma}{2} - \frac{\gamma^2}{4}} e^{\pi \gamma P x} dx \right)^{-\frac{\alpha}{\gamma} - \frac{1}{2}} \right], \nonumber
\end{align}
where $ C=e^{2\ii \pi \lgg - \frac{\ii \pi \gamma^2}{2}} q^{-\frac{\alpha \gamma}{12} - \frac{\gamma^2}{24}} \eta(q)^{-\frac{\alpha\gamma}{2} - \frac{\gamma^2}{4}} \Theta'_{\tau}(0)^{2\lgg} \frac{\Gamma(1 - \frac{\alpha \gamma}{2}) \Gamma(-1 + \frac{\alpha \gamma}{2} - \frac{\gamma^2}{4})}{\Gamma(-\frac{\gamma^2}{4})}$. Since  $1+2l_0\in(0,1)$ and $\Sigma$ is differentiable at $0$, \eqref{eq:ope-easy} yields that
\begin{multline*}
\lim_{u\to 0}\sin(\pi u)^{-2\lgg - 1} \Big(\phi^\alpha_{\frac{\gamma}{2}}(u, q) - \phi^\alpha_{\frac{\gamma}{2}}(0, q)\Big)\\
=\pi^{-2\lgg - 1}\Sigma(0)(1 - e^{\pi \gamma P - 2\ii \pi \lgg}) \left(-\frac{\alpha}{\gamma} + \frac{1}{2}\right) C\\
\EE\left[\left(\int_0^1 e^{\frac{\gamma}{2} Y_\tau(x)} \Theta_\tau(x)^{-\frac{\alpha\gamma}{2} - \frac{\gamma^2}{4}} e^{\pi \gamma P x} dx \right)^{-\frac{\alpha}{\gamma} - \frac{1}{2}} \right].
\end{multline*}
Recall $\Theta_\tau'(0) = -2\pi\eta(q)^3$ from  \eqref{eq:Theta'0}.
Plugging in the value of $\Sigma(0)$ and the definitions of $W^+_{\frac{\gamma}{2}}$ and $\mathcal{A}^q_{\gamma,P}$, we get Lemma~\ref{thm:opes}.
\end{proof}

\xin{Before proving Lemma~\ref{thm:opes2} we first recall a probabilistic  interpretation of  the reflection coefficient $R(\alpha,\chi,P)$ introduced in~\eqref{eq:exp-R}.} Consider the Gaussian  field  $Z_{\mathbb{H}}$  in Appendix \ref{sec:gmc-app} defined on  the upper half plane  with covariance 
\begin{equation}\label{cov_Z}
\mathbb{E}[Z_{\mathbb{H}}(x) Z_{\mathbb{H}}(y)] = 2 \log \frac{|x| \vee |y|}{|x-y|}\quad \textrm{for }x,y \in \revise{ \mathbb{H} \cup \mathbb{R} }.
\end{equation}
For $\lambda>0$ consider the process
\begin{equation}\label{mathcal B}
\mathcal{B}^{\lambda}_s := 
\begin{cases}
\hat{B}_s-\lambda s  \quad s\ge 0\\
\bar{B}_{-s}+\lambda s \quad s<0,
\end{cases}
\end{equation}
where $(\hat{B}_s-\lambda s)_{s\ge 0}$ and $(\bar{B}_{s}-\lambda s)_{s\ge 0}$ are two independent Brownian motions with negative drift conditioned to stay negative;  see \cite{DMS_mot} for more details on this process.
Consider an independent coupling of $(\mathcal{B}^{\lambda},Z_{\mathbb{H}})$ with $\lambda=\frac{Q -\alpha}{2}$. For \xin{$\alpha\in (\frac{\gamma}{2} , Q)$} , let:
\begin{multline} \label{eq:def-rho}
\rho(\alpha,1,e^{-\ii \pi\frac{\gamma\chi}{2}  + \pi \gamma P})\\ := \frac{1}{2} \int_{-\infty}^{\infty} e^{ \frac{\gamma}{2} \mathcal{B}_v^{\frac{Q -\alpha}{2}} } \left(e^{\frac{\gamma}{2} Z_{\mathbb{H}}(-e^{-v/2})} +e^{-\ii \pi\frac{\gamma\chi}{2}  + \pi \gamma P} e^{\frac{\gamma}{2} Z_{\mathbb{H}}(e^{-v/2})}  \right) dv.
\end{multline}
Then \xin{the boundary two-point function for boundary Liouville CFT (without bulk potential) is defined by} 
\begin{equation}\label{eq:reflect}
\overline{R}(\alpha, 1, e^{-\ii \pi\frac{\gamma\chi}{2}  + \pi \gamma P} ) := \EE\left[\Big(\rho(\alpha,1 ,e^{-\ii \pi\frac{\gamma\chi}{2}  + \pi \gamma P})\Big)^{\frac{2}{\gamma}(Q-\alpha)}\right].
\end{equation}
\xin{It was introduced in \cite{RZ20}. It was proved as~\cite[Theorem 1.8]{RZ20} that}
\begin{equation}
\xin{\overline{R}(\alpha, 1, e^{-\ii \pi\frac{\gamma\chi}{2}  + \pi \gamma P} )=R(\alpha,\chi,P) \quad \textrm{with }R(\alpha,\chi,P) \textrm{ in }~\eqref{eq:exp-R}.}
\end{equation}
\xin{An analogous two-point function} was first introduced for Liouville CFT on the Riemann sphere in \cite{KRV19b}. A special case of $\overline{R}(\alpha, 1, e^{-\ii \pi\frac{\gamma\chi}{2}  + \pi \gamma P} )$
was computed in \cite{RZ18}. \xin{They appears naturally in the first order asymptotics of the probability for a GMC measure to be large.}

\begin{proof}[Proof of Lemma \ref{thm:opes2}]

We write $u = \ii t$ and work with small $t>0$. For a Borel set $I \subseteq [0,1]$, we introduce the notation
\begin{equation}
K_{I}( \ii t) : =  \int_I e^{\frac{\gamma}{2} Y_{\tau}(x)} \Theta_{\tau}(x)^{-\frac{\alpha \gamma}{2}} \Theta_{\tau}( \ii t +x)^{\frac{\gamma \chi}{2}} e^{\pi  \gamma P x} dx,  \textrm{ and }s =-\frac{\alpha}{\gamma} + \frac{\chi}{\gamma}.
\end{equation}
Let $h $ be a real parameter chosen slightly larger than $\chi(Q -\alpha)$. Recalling the definition of $\overline X$ given in Appendix \ref{sec:gmc-app}, let $g_\tau(t) \defeq e^{\frac{\gamma}{2} F_{\tau}(0)} e^{\frac{\gamma}{2} \overline{X}(4 \pi t^{1+h}) } $ and $ \sigma_t:= \Theta_{\tau}'(0)^{\frac{\gamma \chi}{2} -\frac{\alpha \gamma}{2}}  (2 \pi)^{\frac{\gamma^2}{4}} 
t^{\frac{\gamma \chi}{2} +\frac{\gamma}{2}(1+h)(Q-\alpha)}   e^{- \frac{\gamma^2}{8} \EE[F_{\tau}(0)^2]} g_\tau(t)$. Let $M$ be an exponential random variable with rate $(Q-\alpha)$, and recall $\rho(\alpha,1,e^{-\ii \pi\frac{\gamma\chi}{2}  + \pi \gamma P}) $ from~\eqref{eq:def-rho}. A straightforward adaptation of the OPE method in \cite{KRV19b,RZ20}\footnote{\revise{In \cite{KRV19b}, see the proofs of Lemmas 9.5 and 9.6. In \cite{RZ20}, in the proof of Lemma 5.6, see equations (5.19), (5.25), (5.49).} } gives the following two claims for $\alpha$ sufficiently close to $Q$. First as $t \rightarrow 0$, the difference $\EE[K_{[0,1]}(\ii t)^{s}]-\EE[K_{[0,1]}(0)^{s}]$ is given by
\begin{align}\label{eq:ope_proof0}
\EE[(K_{(t,1-t)}(\ii t) + \ii^{\frac{\gamma \chi}{2}} \sigma_t e^{\frac{\gamma}{2}M} \rho(\alpha,1,e^{-\ii \pi\frac{\gamma\chi}{2}  + \pi \gamma P})  )^s] - \EE[K_{(t,1-t)}(\ii t)^s] + o(t^{\chi (Q-\alpha)}).
\end{align}
Secondly, we have
\begin{multline}\label{eq:ope_proof}
 \lim\limits_{t\to0} t^{\chi (\alpha-Q)} \EE\big[\sigma_t^{\frac{2}{\gamma}(Q -\alpha)} K_{(t,1-t)}(\ii t)^{s-\frac{2}{\gamma}(Q-\alpha)}  \big] \\=
C\,\EE\big[  \big( \int_0^1 e^{\frac{\gamma}{2} Y_{\tau}(x)} \Theta_{\tau}( x)^{-\frac{ \gamma}{2}(2Q -\alpha -\chi)} e^{\pi  \gamma P x} dx \big)^{\frac{\alpha + \chi - 2Q}{\gamma}} \big]
 \end{multline}
for the constant
\begin{multline*}
C \defeq (2\pi)^{(Q -\alpha)(\frac{\gamma}{3} - \frac{\chi}{3} + \frac{2}{3\gamma} )} q^{\frac{1}{6}(Q -\alpha)(\chi + \frac{2}{\gamma} - 2Q)} \\  \Theta'_{\tau}(0)^{(Q -\alpha)(\frac{2 \chi}{3} - \frac{4}{3 \gamma} - \frac{2}{3\chi})} e^{\ii \pi (Q - \alpha) (\frac{4}{3\gamma} - \frac{2 \chi}{3} - \frac{4}{3\chi})}.
\end{multline*}
Since  $e^{\frac{\gamma}{2}M}$ has density $\frac{2}{\gamma}(Q-\alpha)v^{-\frac{2}{\gamma}(Q-\alpha)-1} 1_{v>1}dv$ one has 
\begin{align*}
&\EE[(K_{(t,1-t)}(\ii t) + \ii^{ \frac{\gamma \chi }{2} } \sigma_t e^{\frac{\gamma}{2}M} \rho(\alpha,1,e^{-\ii \pi\frac{\gamma\chi}{2}  + \pi \gamma P}) )^s] - \EE[K_{(t,1-t)}(\ii t)^s]\\
& = \frac{2}{\gamma}(Q-\alpha) \EE\left[\int_1^{\infty} \frac{dv}{v^{\frac{2}{\gamma}(Q-\alpha)+1}} \right.\\
&\phantom{=====} \left.\left(\left(K_{(t,1-t)}(\ii t) + \ii^{ \frac{\gamma \chi }{2} } \sigma_t \rho(\alpha,1,e^{-\ii \pi\frac{\gamma\chi}{2}  + \pi \gamma P}) v\right)^s - K_{(t,1-t)}(\ii t)^s \right) \right]\\
& = \ii^{\chi(Q-\alpha)} \frac{2}{\gamma}(Q-\alpha) \overline{R}(\alpha, 1, e^{\pi \gamma P - \frac{ \ii \pi \gamma \chi }{2}})\\
&\phantom{========} \EE\left[ \revise{ \int_{I_{t}}  du \frac{ (1 +u )^s -1}{u^{\frac{2}{\gamma}(Q-\alpha)+1}} }  \sigma_t ^{\frac{2}{\gamma}(Q -\alpha)} K_{(t,1-t)}(\ii t)^{s-\frac{2}{\gamma}(Q-\alpha)} \right],
\end{align*}
where we have applied the change of variable $u = \frac{\ii^{\frac{\gamma \chi}{2}} \sigma_t \rho(\alpha,1,e^{- \ii \pi \frac{\gamma \chi}{2} + \pi \gamma P})  }{K_{(t,1-t)}(\ii t)} v $ with $u_t := \frac{\ii^{\frac{\gamma \chi}{2}} \sigma_t \rho(\alpha,1,e^{- \ii \pi \frac{\gamma \chi}{2} + \pi \gamma P})  }{K_{(t,1-t)}(\ii t)}$ being random, and where $\int_{I_t}$ is a complex integral over the half line $I_t$ obtained by rotating the real interval $(|u_t|, +\infty)$ by an angle of $\arg u_t$. Notice that $ \lim_{t \rightarrow 0} u_t =0 $ almost surely. By a standard complex analysis argument, as $t \rightarrow 0$ one has that:
$$ \int_{I_{t}} du \frac{ (1 +u )^s -1}{u^{\frac{2}{\gamma}(Q-\alpha)+1}} = \left( \int_{ |u_t| }^{+\infty} du \frac{ (1 +u )^s -1}{u^{\frac{2}{\gamma}(Q-\alpha)+1}} \right) (1 + o(1)). $$
This then implies:
\begin{align*}
&\EE[(K_{(t,1-t)}(\ii t) + \ii^{ \frac{\gamma \chi }{2} } \sigma_t e^{\frac{\gamma}{2}M} \rho(\alpha,1,e^{-\ii \pi\frac{\gamma\chi}{2}  + \pi \gamma P}) )^s] - \EE[K_{(t,1-t)}(\ii t)^s]\\
&= \ii^{\chi(Q-\alpha)} \frac{2}{\gamma}(Q-\alpha) \overline{R}(\alpha, 1, e^{\pi \gamma P - \frac{ \ii \pi \gamma \chi }{2}}) \\
&\phantom{====} \big(\int_{0}^{\infty}\!\!\!\! du \revise{ \frac{(1 +u )^s -1}{u^{\frac{2}{\gamma}(Q-\alpha)+1}} } \big)  \EE\big[ \sigma_t ^{\frac{2}{\gamma}(Q -\alpha)} K_{(t,1-t)}(\ii t)^{s-\frac{2}{\gamma}(Q-\alpha)} \big] + o(t^{\chi(Q-\alpha)})\\
& = -\ii^{\chi(Q-\alpha)}
\frac{\Gamma(\frac{2 \alpha}{\gamma} - \frac{4}{\gamma^2}) \Gamma(\frac{2Q - \alpha - \chi}{\gamma})}{\Gamma(\frac{\alpha}{\gamma} - \frac{\chi}{\gamma})} 
\overline{R}(\alpha, 1, e^{\pi \gamma P - \frac{ \ii \pi \gamma \chi }{2}})\\
&\phantom{========} \EE\left[\sigma_t^{\frac{2}{\gamma}(Q -\alpha)} K_{(t,1-t)}(\ii t)^{s-\frac{2}{\gamma}(Q-\alpha)} \right] + o(t^{\chi(Q-\alpha)}).
\end{align*}
By using the two claims \eqref{eq:ope_proof0} and \eqref{eq:ope_proof} we arrive at:
\begin{align*}
&\EE[K_{[0,1]}(\ii t)^{s}]-\EE[K_{[0,1]}(0)^{s}] \\
& = - t^{\chi(Q-\alpha)} \ii^{\chi(Q-\alpha)}
\frac{\Gamma(\frac{2 \alpha}{\gamma} - \frac{4}{\gamma^2}) \Gamma(\frac{2Q - \alpha - \chi}{\gamma})}{\Gamma(\frac{\alpha}{\gamma} - \frac{\chi}{\gamma})} 
\overline{R}(\alpha, 1, e^{\pi \gamma P - \frac{ \ii \pi \gamma \chi }{2}})\\
&\phantom{====}C\,\EE\big[  \big( \int_0^1 e^{\frac{\gamma}{2} Y_{\tau}(x)} \Theta_{\tau}( x)^{-\frac{ \gamma}{2}(2Q -\alpha -\chi)} e^{\pi  \gamma P x} dx \big)^{\frac{\alpha + \chi - 2Q}{\gamma}} \big]+ o(t^{\chi(Q-\alpha)}).
\end{align*}
Now the asymptotic of the difference in  Lemma~\ref{thm:opes2} can be reduced to that of $\EE[K_{[0,1]}(\ii t)^{s}]-\EE[K_{[0,1]}(0)^{s}]$.
Since $1+2l_\chi=\chi(Q-\alpha)$, by  \eqref{eq:a-extend} \xin{and $\overline{R}(\alpha, 1, e^{-\ii \pi\frac{\gamma\chi}{2}  + \pi \gamma P} )=R(\alpha,\chi,P)$}, we obtain Lemma~\ref{thm:opes2} for $\chi = \frac{\gamma}{2}$. 

\guillaume{We now check the case $\chi = \frac{2}{\gamma}$, this will require a small manipulation on the $\overline{R}$ function. Set $\tilde{l}_0 = l_{\frac{2}{\gamma}}$. The asymptotic expansion of $\EE[K_{[0,1]}(\ii t)^{s}]-\EE[K_{[0,1]}(0)^{s}]$ we have just derived implies this time:}
\begin{align}\label{eq:gluing1}
&\lim_{u \to 0} \sin(\pi u)^{- 2\tilde{l}_0 - 1} \Big(\phi^\alpha_{\frac{2}{\gamma}}(u, q) - \phi^\alpha_{\frac{2}{\gamma}}(0, q)\Big) \\ \nonumber
& = - q^{\frac{P^2}{2} + \frac{\gamma^2}{24} \tilde{l}_0(1- \tilde{l}_0) } \Theta'_{\tau}(0)^{-\frac{\gamma^2}{6} \tilde{l}_0^2}   \pi^{-1 - \tilde{l}_0} (2\pi)^{\frac{\gamma}{3}(Q -\alpha)} q^{-\frac{\gamma}{6}(Q -\alpha)}   \Theta'_{\tau}(0)^{- \frac{\gamma}{3}(Q -\alpha)} e^{-\ii \pi \frac{2 \gamma}{3} (Q - \alpha) }\\ \nonumber
&  \times \frac{\Gamma(\frac{2 \alpha}{\gamma} - \frac{4}{\gamma^2}) \Gamma(\frac{2}{\gamma^2} +1 - \frac{\alpha}{\gamma})}{\Gamma(\frac{\alpha}{\gamma} - \frac{2}{\gamma^2})} \overline{R}(\alpha, 1, e^{-\ii \pi  + \pi\gamma P})\\ \nonumber
&\phantom{==} \EE\left[\Big( \int_0^1 e^{\frac{\gamma}{2} Y_{\tau}(x)} \Theta_{\tau}( x)^{-\frac{ \gamma}{2}(2Q -\alpha -\frac{2}{\gamma})} e^{\pi  \gamma P x} dx \Big)^{\frac{\alpha}{\gamma} - \frac{2}{\gamma^2} -1} \right].
\end{align}
By the definition~\eqref{eq:a-extend} of $ \mathcal{R}^{q}_{\gamma, P}(\alpha + \frac{2}{\gamma})$, for $\alpha > \frac{\gamma}{2}$ we have
\begin{align}\label{eq:gluing2}
&\mathcal{R}^{q}_{\gamma, P}(\alpha + \frac{2}{\gamma})
=  -q^{-\frac{1}{6}( \frac{\alpha}{\gamma} + \frac{2}{\gamma^2} - 1 + Q(\gamma - \alpha))} \eta(q)^{\frac{13 \alpha \gamma}{2} -1 - \frac{2 \alpha}{\gamma} - \frac{2}{\gamma^2} - \frac{9\alpha^2}{2} - 2 \gamma^2 } \Theta'_{\tau}(0)^{(\frac{\gamma}{2} -\alpha)(\gamma - \frac{2}{\gamma} -\alpha)}  \\ \nonumber
& \times \frac{e^{\ii \pi (\frac{\alpha \gamma}{2} +1  - (\alpha - \gamma)(\alpha - \gamma - \frac{2}{\gamma}))} (2 \pi)^{(\alpha - \gamma)(\frac{\gamma}{2} -\alpha)}}{  (  1 -\frac{\alpha}{\gamma} - \frac{2}{\gamma^2})(1+e^{\pi \gamma P - \ii \pi  \frac{\gamma^2}{2} + \ii \pi  \frac{\alpha \gamma}{2}}) }  \frac{\Gamma(-\frac{\gamma^2}{4}) \Gamma(\frac{2 \alpha}{\gamma} -1 ) \Gamma(1 + \frac{2}{\gamma^2} - \frac{\alpha}{\gamma})}{\Gamma(\frac{\alpha \gamma}{2}  - \frac{\gamma^2}{2})\Gamma( \frac{\gamma^2}{4} - \frac{\alpha \gamma}{2} ) \Gamma(\frac{\alpha}{\gamma} + \frac{2}{\gamma^2} - 1)} \\ \nonumber
& \times \overline{R}(\alpha + \frac{2}{\gamma} - \frac{\gamma}{2}, 1, e^{- \ii \pi\frac{\gamma^2}{4} + \pi \gamma P}) 
\mathbb{E} \Big[  \Big( \int_0^1 e^{\frac{\gamma}{2} Y_{\tau}(x)} \Theta_{\tau}( x)^{-\frac{ \gamma}{2}(2Q -\alpha - \frac{2}{\gamma} )} e^{\pi  \gamma P x} dx \Big)^{ \frac{\alpha}{\gamma} - \frac{2}{\gamma^2} - 1} \Big].
\end{align}
\revise{To obtain the desired answer of Lemma \ref{thm:opes2}, using the shift equations satisfied by $\overline{R}$ - see equations (3.22) and (3.74) of \cite{RZ20} - and simply canceling out common factors, we compute a ratio of reflection coefficients as}
\begin{multline*}
\frac{\overline{R}(\alpha , 1,  e^{ - \ii \pi + \pi \gamma P})}{\overline{R}(\alpha + \frac{2}{\gamma} -\frac{\gamma}{2}, 1, e^{-\ii \pi \frac{ \gamma^2}{4} + \pi \gamma P})} = \frac{\overline{R}(\alpha , 1,  e^{ - \ii \pi + \pi \gamma P})}{\overline{R}(\alpha + \frac{2}{\gamma} , 1,  e^{ \pi \gamma P})} \frac{\overline{R}(\alpha + \frac{2}{\gamma} , 1,  e^{ \pi \gamma P})}{\overline{R}(\alpha + \frac{2}{\gamma} -\frac{\gamma}{2}, 1, e^{-\ii \pi \frac{ \gamma^2}{4} + \pi \gamma P})}\\
= \frac{2}{\gamma(Q -\alpha)} (2 \pi)^{\frac{4}{\gamma^2} -1}  \frac{\Gamma(\frac{2\alpha}{\gamma}) \Gamma(1 - \frac{2\alpha}{\gamma})}{\Gamma(1 - \frac{\gamma^2}{4})^{\frac{4}{\gamma^2} -1 } \Gamma(\frac{\gamma \alpha}{2} - \frac{\gamma^2}{2}) \Gamma( 1 - \frac{\gamma \alpha}{2} + \frac{\gamma^2}{4}) } \frac{1 - e^{\frac{4 \pi P}{\gamma} - \frac{4 \ii \pi}{\gamma^2} + \ii \pi \frac{2 \alpha}{\gamma}} }{1 + e^{ \pi \gamma  P - \frac{ \ii \pi \gamma^2}{2} + \ii \pi \frac{\gamma \alpha}{2}} }.
\end{multline*}
Substituting \eqref{eq:gluing2} into \eqref{eq:gluing1} and simplifying yields the
desired claim.
\end{proof}

\bibliographystyle{alpha}
\bibliography{liou-bib}
\end{document}